\documentclass[12pt,dvipsnames]{yalephdnew}

\usepackage{fancyhdr}
\usepackage{amsfonts}
\usepackage{mathrsfs}
\usepackage{amssymb}
\usepackage{amsmath}
\usepackage{array}
\usepackage{extarrows}
\usepackage{color}
\usepackage{cite}
\usepackage{mathtools}
\usepackage{graphicx}
\usepackage{caption}
\usepackage{subcaption}
\usepackage{slashbox}
\usepackage{stfloats}
\usepackage{epstopdf}
\usepackage{arydshln}
\usepackage{ifthen}
\usepackage{url}
\usepackage{pict2e}
\usepackage{lscape}
\usepackage{times}
\usepackage{float}
\usepackage{afterpage}
\usepackage{epsfig}
\usepackage{longtable}
\usepackage[pagebackref=false]{hyperref} 
\usepackage{setspace}
\usepackage{titlesec}

\usepackage{tikz}
\usetikzlibrary{calc, matrix, shapes, arrows, positioning, chains, decorations, 
	shapes.gates.logic.US}

\definecolor{linkColor}{rgb}{0.0,0.11,0.22}
\definecolor{YaleBlue}{rgb}{0.0,0.22,0.444}

\DeclareMathOperator{\relu}{relu}

\usepackage[utf8]{inputenc} 
\usepackage[T1]{fontenc}    
\usepackage{hyperref}       
\usepackage{url}            
\usepackage{booktabs}       
\usepackage{amsfonts}       
\usepackage{nicefrac}       
\usepackage{microtype}      

\usepackage{ctable}

\usepackage{amsmath}
\usepackage{amssymb}
\usepackage{amsthm}
\usepackage{pifont}
\usepackage{algorithm}
\usepackage{algpseudocode}
\usepackage[capitalize]{cleveref}
\usepackage{breqn}
\usepackage[square,numbers]{natbib}
\usepackage{xspace}
\usepackage{forloop}
\usepackage{multirow}
\usepackage{bbm}
\usepackage{graphicx}
\usepackage{paralist}
\usepackage{subcaption}
\usepackage{dsfont}
\usepackage{csquotes}

\newcommand{\sfw}{\texttt{1-SFW}\xspace}
\newcommand{\AlgQFW}{\texttt{Quantized Frank-Wolfe}\xspace}

\newcommand{\sncAlg}{\texttt{Stochastic Non-Convex Quantized 
		Frank-Wolfe}\xspace}
\newcommand{\qfw}{QFW\xspace}

\newcommand{\sscheme}{\texttt{Sign Encoding Scheme}\xspace}
\newcommand{\mscheme}{\texttt{Partition Encoding Scheme}\xspace}

\newcommand{\fl}{\texttt{FL}\xspace}
\newcommand{\AlgCG}{\texttt{Continuous Greedy}\xspace}
\newcommand{\Algblack}{\texttt{Black-Box Continuous Greedy}\xspace}
\newcommand{\Algdis}{\texttt{Discrete Black-Box Greedy}\xspace}

\newcommand{\revise}[1]{{#1}}
\newcommand{\domainsh}{\domain_\delta}
\DeclareMathOperator*{\diam}{diam}
\newcommand{\dif}{\mathop{}\!\mathrm{d}}

\newtheorem{theorem}{Theorem}
\newtheorem{lemma}{Lemma}
\newtheorem{assump}{Assumption}
\newtheorem{remark}{Remark}

\crefname{assump}{Assumption}{Assumptions}

\newcommand{\tnabla}{\tilde{\nabla}}
\newcommand{\tF}{\tilde{F}}
\newcommand{\expect}{\mathbb{E}}
\newcommand{\constraint}{\mathcal{K}}
\newcommand{\one}{\mathbf{1}}

\newcommand{\domain}{\mathcal{X}}
\newcommand{\ie}{\emph{i.e.}}
\newcommand{\eg}{\emph{e.g.}}
\newcommand{\matroid}{\mathcal{I}}

\DeclareMathOperator*{\argmax}{arg\,max}
\DeclareMathOperator*{\argmin}{arg\,min}

\input{required.tex}
\input{mysymbol.sty}

\begin{document}

\newcommand{\unit}[1]{\ensuremath{\, \mathrm{#1}}}

\titleformat{\chapter}[display]
    {\fontfamily{ptm}\huge\bfseries}{\chaptertitlename\ \thechapter}{5pt}{\centering\singlespacing\huge}
\titlespacing*{\chapter}{0pt}{50pt}{30pt}

	\title{Scalable Projection-Free Optimization}
	\author{Mingrui Zhang}
	%
	\advisor{Amin Karbasi}
	\date{June 2021}


    \pagestyle{plain}	

	\newpage    
\begin{abstract}
As a projection-free algorithm, Frank-Wolfe (FW) method, also known as 
conditional gradient, has recently received considerable attention in the 
machine learning community. In this dissertation\footnote{\revise{This 
dissertation includes the following publications 
\citep{zhang2020one,zhang2020quantized,chen2020black}. Some passages in this 
dissertation
have been quoted verbatim from the above papers.}}, we study several topics on 
the FW variants for scalable projection-free optimization.

We first propose \sfw, the first projection-free method that requires only one
sample per iteration to update the optimization variable and yet achieves the 
best known complexity bounds for convex, non-convex, and monotone DR-submodular 
settings. Then we move forward to the distributed setting, and develop \AlgQFW
(QFW), a general communication-efficient distributed FW framework for both
convex and non-convex objective functions. We study the performance of 
QFW in two widely recognized settings: 1) stochastic optimization and 2) 
finite-sum optimization. Finally, we propose \Algblack, a derivative-free and 
projection-free algorithm, that maximizes a monotone continuous DR-submodular
function over a bounded convex body in Euclidean space.

\end{abstract}

    \maketitle
        
    \thispagestyle{empty}    
    \makecopyright

	\pagenumbering{roman}	    
	\setcounter{page}{3}

	\newpage
	\begin{center}
{\bf \large Acknowledgments}
\end{center}
During the six years for pursuing a Ph.D. degree, I have received a lot of 
help, support and encouragement.

First and foremost I would like to express my sincere gratitude to my advisors, 
Professor Amin Karbasi and Professor Sekhar Tatikonda, for their constant 
guidance, regular meetings, helpful advice and support for my research and this 
dissertation. I am deeply grateful to Professor Huibin Zhou, the other member 
in my dissertation committee. I would also like to thank all my research 
collaborators: Professor Hamed Hassani, Lin Chen, Yifei Min, Aryan Mokhtari, 
Zebang Shen. 

I would like to express my sincere thanks to all my teachers from the  
elementary school to the graduate school. Especially I would like to offer my 
thanks to the following people, listed in a chronological order: Lan Yang, 
Xiabing Shen, Weizi Zheng, Panfeng Wang, Professor Chunwei Song, Professor 
Houhong Fan, Professor Minping Qian, Professor Daquan Jiang.  

I would also like to thank my friends at Yale University, Peking University, 
Kaifeng High School, etc. Here I would like to give thanks to the following 
people, listed in an alphabetic order: Peiliang Bai, Lin Chen, Siyuan Dong, 
Xialiang Dou, Bei Fan, Peizhen Guo, Chris Harshaw, Bo Hu, Wenmian Hua, Dingjue 
Ji, Pengcheng Li, 
Yitong Li, Yupeng Li, Naijia Liu, Zhan Liu, Yifei Min, Marko Mitrovic, Zhichao 
Peng, Weiqi 
Shi, Chuan Tian, 
Kaizheng Wang, Qingcan Wang, Shuai Wang, Xingyan Wang, Zeyu Wang, Zhixin Wang, 
Zheng Wei,
Yujun Xie, 
Wenjie Xiong, Ruitu Xu, Sheng Xu, Yihan Zhou, Yun Zhou, Zihan Zhuo.

In the end, I would like to express my deepest gratitude to my parents for 
their endless love and support.

    \tableofcontents %
    \listoffigures   
    \listoftables    %

    \mainmatter      
    
	\chapter{Introduction}\label{cha:intro}
In many modern machine learning scenarios, the task of learning is usually 
converted to an optimization problem, where the loss function is defined as the 
empirical loss function (plus some possible regularization). In many 
cases, the corresponding constraint set is not the whole Euclidean space, but 
some bounded convex set. 

In order to solve these \emph{constrained} optimization problems, methods like 
Projected Gradient Descent (PGD) are quite popular and effective in practice.
In these methods, projection oracle is applied once the proposed iterates land 
outside the 
feasibility region. However, the projection operation can be computationally 
expensive for some 
special constraint sets. For example, in recommender systems and matrix 
completion, projections amount to expensive linear algebraic operations.
Similarly, projections onto matroid polytopes with exponentially many linear 
inequalities are daunting tasks in general. This difficulty has motivated the
use of projection-free algorithms.

As a projection-free algorithm for various constrained convex 
\citep{jaggi2013revisiting,lacoste2015global,garber2015faster,hazan2016variance,mokhtari2018stochastic}
 and non-convex 
\citep{lacoste2016convergence,reddi2016stochastic,mokhtari2018escaping,zhang2019one,hassani2019stochastic}
optimization problems, the Frank-Wolfe (FW) method \citep{frank1956algorithm}, 
also known as conditional gradient, has recently received considerable 
attention in the machine learning community.

In this dissertation, we investigate various topics on the FW variants for 
scalable 
projection-free optimization.

\section{One-Sample Stochastic Frank-Wolfe}
Although Frank-Wolfe (FW) methods have been widely used for solving constrained 
optimization problems 
\citep{frank1956algorithm,jaggi2013revisiting,lacoste2015global}, exact 
gradient evaluations are required in order to guarantee convergence. In many 
cases, however, exact gradients are difficult to compute or even inaccessible. 
This challenge motivates the study of FW variants which can be fed with 
stochastic gradient information. 

Indeed, extending the original FW methods to 
the stochastic setting is a 
challenging task as it is known that FW-type methods are highly sensitive to 
stochasticity in gradient computation \citep{hazan2012projection}.
To resolve this issue, several stochastic variants of FW methods have been 
studied in the literature 
\citep{hazan2012projection,hazan2016variance,reddi2016stochastic,
	lan2016conditional,braun2017lazifying,hassani2019stochastic,
	shen2019complexities,yurtsever2019conditional}.
In all these stochastic methods, the basic idea is to provide an accurate 
estimate of the gradient by using some variance reduction techniques that 
typically rely on large mini-batches of samples where the size grows with the 
number of iterations or is reciprocal of the desired accuracy. A  growing 
mini-batch, however, is undesirable in practice as requiring a large collection 
of samples per iteration  may easily prolong the duration of each 
iterate  without updating optimization parameters frequently enough 
\cite{defazio2018ineffectiveness}.   A notable exception to this trend is the 
the work of \cite{mokhtari2018stochastic} which employs a momentum 
variance reduction technique requiring only one sample per iteration; however, 
this method suffers from suboptimal convergence rates.

In this dissertation, we present the first projection-free method 
that requires only one sample per iteration to update the optimization 
variable and yet achieves the best known complexity bounds for convex, 
non-convex, and monotone DR-submodular settings.

\section{Communication-Efficient Frank-Wolfe in the Distributed Setting}
Thanks to the numerous information-sensors and many other modern information 
technologies, 
the sizes of available datasets have been growing fast recently. As a result, 
efficient FW methods are also motivated to be applied to large-scale problems 
(\emph{e.g.}, training deep neural networks 
\citep{ravi2018constrained,schramowski2018neural,berrada2018deep}, 
RBMs \citep{ping2016learning}). To this end, distributed FW variants have been 
proposed for specific problems, \emph{e.g.}, online learning 
\citep{zhang2017projection}, learning low-rank matrices 
\citep{zheng2018distributed}, and optimization under block-separable 
constraint sets \citep{wang2016parallel}. 

As is well known, a significant performance bottleneck of distributed 
optimization methods is the 
cost of communicating gradients, which is typically handled by using a 
parameter-server 
framework.  Intuitively, if each worker in the distributed 
system transmits the entire gradient, then at least $d$ floating-point numbers 
are communicated for each worker, where $d$ is the dimension of the problem. 
This 
communication 
cost can be a huge burden on the performance of parallel optimization 
algorithms 
\citep{chilimbi2014project,seide20141,strom2015scalable}. To circumvent this 
drawback, communication-efficient parallel algorithms have received 
significant attention. One major approach is to quantize the 
gradients while maintaining sufficient information 
\citep{de2015taming,abadi2016tensorflow,wen2017terngrad}. For 
\textit{unconstrained} optimization, when projection is not required for 
implementing Stochastic Gradient Descent (SGD), several communication-efficient 
distributed methods have been proposed, including 
QSGD \citep{alistarh2017qsgd}, SIGN-SGD 
\citep{bernstein2018signsgd}, and Sparsified-SGD~\citep{stich2018sparsified}.

In the constrained setting, and in particular for distributed FW 
methods, the communication-efficient versions were only studied for 
specific problems such as sparse learning 
\citep{bellet2015distributed,lafond2016d}. 
In this dissertation, we develop \AlgQFW (QFW), a general 
communication-efficient distributed FW framework for both convex and non-convex 
objective 
functions. We also study the performance of QFW in two widely 
recognized settings: 1) stochastic optimization and 2) finite-sum optimization.

\section{Black-Box Submodular Maximization}
Black-Box optimization, also known as zeroth-order or derivative-free
optimization\footnote{\revise{We note that black-box 
		optimization (BBO) and 
		derivative-free 
		optimization (DFO) are not identical terms. 
		\citet{audet2017derivative} defined DFO as ``the mathematical study of 
		optimization 
		algorithms that do not use derivatives'' and BBO as ``the study of 
		design 
		and analysis of algorithms that assume the objective and/or constraint 
		functions are given by blackboxes''. However, as the differences are 
		nuanced in 
		most scenarios, this dissertation uses them interchangeably. For a 
		detailed 
		review of 
		DFO and BBO, interested readers refer to 
		book \citep{audet2017derivative}.}}, has been 
extensively studied in the literature 
\citep{conn2009introduction,bergstra2011algorithms,rios2013derivative,shahriari2016taking}.
In this setting, we assume that the objective function is unknown 
and we can only obtain 
zeroth-order information such as  (stochastic) function evaluations.

Fueled by a growing number of machine learning applications, black-box 
optimization  methods are usually 
considered in scenarios where 
gradients (\emph{i.e.}, 
first-order information) are 1) difficult or slow to compute, \emph{e.g.}, 
graphical 
model inference \citep{wainwright2008graphical}, structure predictions 
\citep{taskar2005learning,sokolov2016stochastic}, or 2) inaccessible, 
\emph{e.g.}, 
hyper-parameter 
turning for natural language processing or image 
classifications \cite{snoek2012practical,thornton2013auto}, black-box attacks 
for finding adversarial examples \cite{chen2017zoo,ilyas2018black}. Even though 
heuristics such as 
random or grid search, with undesirable dependencies on the dimension, are 
still used in some applications 
(\emph{e.g.}, 
parameter tuning for deep networks), there have been a growing number of 
rigorous methods to address the 
convergence rate of black-box optimization in convex and non-convex settings 
\citep{wang2017stochastic,balasubramanian2018zeroth,sahu2018towards}.

Continuous DR-submodular functions are an important subset of non-convex 
functions that can be 
minimized exactly \cite{bach2016submodular,staib2017robust}
and maximized approximately 
\cite{bian2017continuous,bian2017guaranteed,hassani2017gradient,
	mokhtari2018conditional,hassani2019stochastic,zhang2019one}.
This class of 
functions generalizes the notion of diminishing returns, 
usually defined over discrete set functions, to the continuous domains. They 
have found 
numerous applications in machine learning including MAP inference in 
determinantal point processes (DPPs) 
\cite{kulesza2012determinantal}, experimental design 
\cite{chen2018online}, 
resource allocation \cite{eghbali2016designing},  
mean-field inference in probabilistic models 
\cite{bian2018optimal}, among many others. 

In this dissertation, we propose a derivative-free FW method for continuous 
DR-submodular maximization over a bounded convex body, which also avoids the 
expensive projection operations.


\section{Organization}
In this dissertation, we study several topics on scalable projection-free 
optimization algorithms. The rest of this 
dissertation is organized as follows.

In \cref{cha:background}, we introduce important background and 
related work.

In \cref{cha:1sfw}, we present the first one-sample stochastic Frank-Wolfe 
method, called \sfw, which attains the best known complexity bounds for 
convex, non-convex, and monotone DR-submodular settings, while requiring only 
\emph{one single} stochastic oracle 
query per iteration and avoiding large batch sizes altogether. In particular, 
we show 
that \sfw achieves the 
optimal  convergence rate  of $\mathcal{O}(1/\epsilon^2)$  for reaching an  
$\epsilon$-suboptimal solution in the stochastic convex setting, and  a 
$(1-1/e)-\epsilon$ approximate solution for a stochastic monotone DR-submodular 
maximization problem. In a general non-convex setting, \sfw finds an 
$\epsilon$-first-order stationary point after at most 
$\mathcal{O}(1/\epsilon^3)$ iterations. We also empirically validate the 
efficiency of \sfw algorithm by comparing it with baseline methods in 
\cref{sec:experiments-sfw}. This chapter is based on our work in 
\citep{zhang2020one}\footnote{\revise{This work was done in collaboration with 
Zebang Shen, Aryan Mokhtari, Hamed Hassani, and Amin Karbasi. I proposed the 
algorithms and also completed the 
work of theoretical analysis.}}.

In \cref{cha:qfw}, we propose a novel distributed projection-free framework, 
\AlgQFW (QFW), which handles quantization for constrained convex and 
non-convex optimization problems in finite-sum and stochastic cases. We show 
that with quantized gradients, we can obtain a 
provably convergent method which preserves the 
convergence rates of the state-of-the-art vanilla centralized methods in all 
the considered cases 
\citep{zhang2019one,shen2019complexities,hassani2019stochastic,yurtsever2019conditional}.
In \cref{sec:experiments-qfw}, we evaluate the performance of algorithms by 
visualizing their loss vs.\ the number of transmitted bits on two problems: 
multinomial logistic regression and three-layer neural network under $\ell_1$ 
constraint. This chapter is based on our work in 
\citep{zhang2020quantized}\footnote{\revise{This work was done in collaboration 
with Lin Chen, Aryan Mokhtari, Hamed Hassani, and Amin Karbasi. I proposed the 
algorithms and also completed the work of theoretical analysis.}}.

In \cref{cha:black}, we propose a derivative-free and projection-free algorithm 
\Algblack (BCG), that maximizes a monotone continuous DR-submodular 
function over a bounded convex body in Euclidean space. We study three 
scenarios:

(1) In the deterministic setting, where function evaluations can be obtained 
exactly, BCG achieves the tight $[(1-1/e)OPT-\epsilon]$ approximation guarantee 
with $\mathcal{O}(d/\epsilon^3)$ function evaluations.

(2) In the stochastic setting, where function evaluations are noisy, BCG 
achieves the tight $[(1-1/e)OPT-\epsilon]$ approximation guarantee with 
$\mathcal{O}(d^3/\epsilon^5)$ function evaluations.

(3) In the discrete setting, \Algdis (DBG), the discrete version of BCG, 
achieves the tight 
$[(1-1/e)OPT-\epsilon]$ 
approximation guarantee with $\mathcal{O}(d^5/\epsilon^5)$ function 
evaluations. 

In \cref{sec:experiment-black}, numerical experiments show that 
empirically, our proposed algorithm often requires significantly fewer function 
evaluations and less running time compared with baselines, while achieving a 
practically similar utility. This chapter is based on our work in 
\citep{chen2020black}\footnote{\revise{This work was done in collaboration 
with Lin Chen, Hamed Hassani, and Amin Karbasi. I proposed the algorithms, 
completed the work of theoretical analysis and numerical experiments 
jointly with Lin Chen.}}.

	\chapter{Preliminaries}\label{cha:background}
In this chapter, we present important preliminaries and related work for the 
topics of this dissertation.

\section{Frank-Wolfe Algorithm}
Frank-Wolfe method \citep{frank1956algorithm}, also known as the conditional 
gradient method, has been studied for both convex 
optimization 
\citep{jaggi2013revisiting,lacoste2015global,garber2015faster,hazan2016variance,mokhtari2018stochastic}
and non-convex optimization problems
\citep{lacoste2016convergence,reddi2016stochastic,mokhtari2018escaping,
	shen2019hessian,hassani2019stochastic}.

As a projection-free algorithm, FW method replaces the projection operations by 
solving a linear optimization problem for each iteration. To be precise, 
suppose that we want to solve the following constrained convex optimization 
problem:
\begin{align*}\label{eq:fw-intro}
\min_{x \in \constraint} F(x), 
\end{align*}
where $\constraint \subseteq \mathbb{R}^d$ is a compact convex set, and $F$ is 
convex.
Assuming that we have access to the \emph{exact} gradient 
$\nabla F$, the FW method starts from some initial point $ x^{(1)} \in 
\constraint$, and at the $ k $-th 
iteration, solves a linear optimization problem
\begin{equation*}\label{eq:linear_opt}
v^{(k)}\gets \argmin_{v\in 
	\constraint} \langle 
v, \nabla F(x^{(k)})\rangle,
\end{equation*}
which is used to update 
$x^{(k+1)} \gets x^{(k)} + \eta_k (v^{(k)} - x^{(k)})$, 
where $ \eta_k \in [0,1]$ is the step size. Note that $x^{(k+1)}$ is a convex 
combination of $x^{(k)}$ and $v^{(k)}$, which are defined to fall in 
$\constraint$. Therefore, we have $x^{(k+1)} \in \constraint$, thus avoid 
the projection operation which is used to guarantee all the iterates land 
inside the constraint set. 

Since linear optimization problems can 
usually be solved fast by various algorithms, thus for many practical problems, 
FW methods can be much more computationally efficient 
than the projected gradient-based methods like PGD.

In terms of theoretical convergence rate, it can be shown that the iterates of 
the FW method above satisfies that 
\begin{equation*}
F(x^{(k)}) - F(x^*) \leq \mathcal{O}(\frac{1}{k}),
\end{equation*}
where $x^*$ is the global minimizer of the convex function $F$ in $\constraint$ 
\citep{frank1956algorithm,dunn1978conditional}. FW methods can also be utilized 
to solve non-convex minimization and monotone continuous DR-submodular 
maximization problems 
with slight modifications 
\citep{lacoste2016convergence,bian2017guaranteed}.

When fed with \emph{stochastic} gradient, however, FW methods may diverge 
\citep{hazan2016variance,mokhtari2018stochastic}. In order to establish 
guaranteed convergences, stochastic FW methods are usually incorporated with 
various 
variance reduction techniques.
A detailed summary of 
convergence rates for various stochastic FW-type algorithms for convex 
minimization, non-convex minimization, and monotone continuous DR-submodular 
maximization problems can be found in \cref{table-sfw}.

\subsection{Related Work on Frank-Wolfe Algorithm}
Frank-Wolfe methods are very sensitive to noisy gradients. This issue was 
recently resolved in centralized \citep{mokhtari2018stochastic} and online 
settings \citep{Chen2018Projection,chen2019projection}. 
In large-scale settings, distributed FW methods were proposed to solve 
specific 
problems, including optimization under block-separable constraint set 
\citep{wang2016parallel}, and learning low-rank matrices 
\citep{zheng2018distributed}. The communication-efficient distributed FW 
variants were proposed for specific sparse learning problems in 
\cite{bellet2015distributed,lafond2016d}, and for general constrained 
optimization problems in our paper \citep{zhang2020quantized}.
Zeroth-order FW methods 
were studied in \citep{chen2020black,balasubramanian2018zeroth,sahu2018towards}.

Stochastic FW methods are strongly associated with variance reduction 
techniques. Several works have studied different ideas for reducing variance. 
The SVRG method was proposed by
\cite{Johnson2013Accelerating} for the convex setting and then extended to the 
non-convex setting in 
\citep{allen2016variance,reddi2016stochastic,zhou2018stochastic}. The StochAstic
Recursive grAdient algoritHm (SARAH)  was studied in 
\citep{nguyen2017sarah,nguyen2017stochastic}. Then as a variant of SARAH, the 
Stochastic Path-Integrated Differential Estimator (SPIDER) technique was 
proposed by \cite{fang2018spider}. Based on SPIDER, various algorithms for 
convex and non-convex optimization problems have been studied 
\citep{shen2019complexities,hassani2019stochastic,yurtsever2019conditional}. 

\section{Submodular Functions}
We say a set function $f:2^{\Omega}\to 
\mathbb{R}$ is submodular, if it satisfies the diminishing returns property: 
for any $A\subseteq B\subseteq \Omega$ and $x\in \Omega\setminus B$, we have 
\begin{equation}
f(A\cup\{x\})-f(A)\geq f(B\cup \{x\})-f(B).
\end{equation}
In words, the marginal gain of adding an element $x$ to a subset $A$ is no less 
than that of adding $x$ to its superset $B$. A submodular set function $f: 
2^{\Omega} \to \mathbb{R}$ is called monotone if for any two sets 
$A\subseteq B\subseteq \Omega$ we have $f(A)\le f(B)$. 

For the continuous analogue, consider a function $F:\mathcal{X} \to 
\mathbb{R}_+$, where $\mathcal{X}=\Pi_{i=1}^d \mathcal{X}_i$, and each 
$\mathcal{X}_i$ is a compact subset of $\mathbb{R}_+$. We define $F$ to be 
continuous submodular if $F$ is continuous and for all $x, y \in \mathcal{X}$, 
we have 
\begin{equation}
F(x) + F(y) \geq F(x\vee y) + F(x \wedge y),
\end{equation}
where $\vee$ and $\wedge$ are the component-wise maximizing and minimizing 
operators, respectively.

For two vectors $x, y \in \mathbb{R}^d$, we write $x\le y$ if $x_i\le y_i$ 
holds for every 
$i \in \{1,2,\cdots, d\}$, where $x_i$ is the $i$-th coordinate of $x$. Then a 
continuous function $F$ is  called 
DR-submodular 
\cite{bian2017guaranteed} 
if $F$ is differentiable and $\forall\, x \leq y: \,\nabla F(x) \geq \nabla 
F(y).$ The function $F$ is called monotone if for $x \leq y$, we have
$F(x) \leq F(y).$    

An important implication of DR-submodularity is that the function $F$ is 
concave in any non-negative directions, \emph{i.e.}, for $x \leq y$, we have 
\begin{equation}
F(y) \leq F(x)+ \langle \nabla F(x), y-x\rangle.
\end{equation}

For continuous DR-submodular maximization, it has been shown that approximated 
solution within a factor of $(1-e^{-1} + \epsilon)$ can not be obtained in 
polynomial time \citep{bian2017guaranteed}.

\subsection{Related Work on Submodular Functions}
Submodular functions \cite{nemhauser1978analysis}, that capture the intuitive 
notion of diminishing returns, 
have become increasingly 
important in various machine learning applications. 
Examples include graph cuts in computer vision 
\cite{jegelka2011submodularity, jegelka2011approximation},  
data summarization \cite{lin2011word, lin2011class, 
	tschiatschek2014learning,chen2018weakly,chen2017interactive},
influence maximization 
\cite{kempe2003maximizing,rodriguez2012influence,zhang2016influence},
feature compression \cite{bateni2019categorical}, 
network inference \cite{chen2017submodular},
active and semi-supervised 
learning 
\cite{guillory2010interactive, golovin2011adaptive, wei2015submodularity}, 
crowd 
teaching 
\cite{singla2014near},  
dictionary 
learning \cite{das2011submodular}, fMRI parcellation 
\cite{salehi2017submodular}, compressed sensing and structured 
sparsity 
\cite{bach2010structured, bach2012optimization}, fairness in machine learning 
\cite{balkanski2015mechanisms, celis2016fair}, learning causal structures 
\cite{steudel2010causal,zhou2016causal}, experimental design 
\citep{chen2018online}, MAP inference in determinantal
point processes (DPPs) \citep{kulesza2012determinantal}, and mean-field 
inference in probabilistic models \citep{bian2019optimal}, to name a few.

Continuous DR-submodular functions naturally extend the notion of 
diminishing returns to  the continuous domains \cite{bian2017guaranteed}. 
Monotone continuous DR-submodular functions can be 
minimized exactly \citep{bach2015submodular,staib2017robust}, and maximized 
approximately \citep{bian2017guaranteed,bian2017continuous,hassani2017gradient,
	mokhtari2017conditional,mokhtari2018conditional,niazadeh2018optimal,hassani2019stochastic,
	chen2019unconstrained}. Among those works, \citet{bach2015submodular} 
	derived connections between continuous 
submodularity and convexity, while \citet{bian2017guaranteed} studied the 
offline 
continuous DR-submodular maximization and proposed a variant of the 
Frank-Wolfe 
algorithm to achieve the tight $(1-1/e)$ approximation ratio.
A derivative-free and projection-free algorithm for monotone continuous 
DR-submodular maximization was proposed in our work \citep{chen2020black}.

In the online setting, maximization of submodular set functions was studied 
in~\citep{streeter2009online,golovin14online}. Adaptive submodular bandit 
maximization was analyzed in \citep{gabillon2013adaptive}. The
linear submodular bandit problems were studied in 
\citep{yue2011linear,yu2016linear}. The first online and bandit algorithms for 
general continuous submodular maximization problems were proposed in our work 
\citep{zhang2019online}.

	\chapter{One Sample Stochastic Frank-Wolfe}\label{cha:1sfw}
\section{Introduction}
Recall that FW algorithms are very sensitive to noisy gradients 
\citep{hazan2016variance,mokhtari2018stochastic}. As a result, 
many stochastic FW variants are fed with an accurate estimation of gradients by 
utilizing various variance reduction techniques 
\citep{hazan2012projection,hazan2016variance,reddi2016stochastic,
lan2016conditional,braun2017lazifying,hassani2019stochastic,
shen2019complexities,yurtsever2019conditional}. These variance 
reduction methods usually rely on large mini-batches of samples where the size 
grows with the 
number of iterations or is reciprocal of the desired accuracy. A growing 
mini-batch, 
however, is undesirable in practice as requiring a large collection of samples 
per iteration  may easily prolong the duration of each 
iterate  without updating optimization parameters frequently enough 
\cite{defazio2018ineffectiveness}. A notable exception to this trend is the the 
work of \cite{mokhtari2018stochastic} which employs a momentum 
variance-reduction technique requiring only one sample per iteration; however, 
this method suffers from suboptimal convergence rates. At the heart of this 
chapter\footnote{\revise{This chapter is based on our work in 
		\citep{zhang2020one}.}} is the answer to the following question:

\begin{quote}
	\textit{Can we achieve the best known complexity bounds for a stochastic 
	variant 
	of Frank-Wolfe while using a single stochastic sample per iteration?}
\end{quote}

We show that the answer to the above question is positive and present the first projection-free method 
that  requires only one sample per iteration to update the optimization variable 
and yet  achieves the best known complexity bounds for convex, non-convex, and 
monotone
DR-submodular settings.

More formally, we focus on a general \emph{non-oblivious} constrained 
stochastic optimization problem
\begin{equation}\label{eq:formation}
\min_{x \in \constraint}F(x) \triangleq \min_{x \in \constraint} 
\expect_{z \sim p(z;x)}[\tF(x;z)],
\end{equation}
where $x \in \mathbb{R}^d$ is the optimization variable, $\constraint 
\subseteq \mathbb{R}^d$ is the convex constraint set, and the objective 
function $F: \mathbb{R}^d \to \mathbb{R}$ is defined as the expectation over a 
set of functions $\tF$. The function $\tF: \mathbb{R}^d \times 
\mathcal{Z} \to \mathbb{R}$ is determined by $x$ and a 
random variable $z \in \mathcal{Z}$ with distribution $z \sim p(z;x)$. 
We refer to problem~\eqref{eq:formation} as a non-oblivious stochastic 
optimization problem 
as the 
distribution of the random variable $z$ depends on the choice of $x$. When 
the distribution $p$ is independent of $x$,  we are in the standard 
oblivious 
stochastic optimization regime where the goal is to solve
\begin{equation}\label{eq:ob_op}
\min_{x \in \constraint} F(x) \triangleq \min_{x \in \constraint} 
\expect_{z \sim p(z)}[\tF(x;z)].
\end{equation}
Hence, the oblivious problem \eqref{eq:ob_op} can be considered as a special 
case of the non-oblivious problem \eqref{eq:formation}. Note that 
non-oblivious stochastic optimization has broad applications in machine 
learning, including multi-linear extension of a discrete submodular 
function \citep{hassani2019stochastic}, MAP inference in determinantal
point processes (DPPs) \citep{kulesza2012determinantal}, and reinforcement 
learning 
\citep{du2017stochastic,sutton2018reinforcement,papini2018stochastic,shen2019hessian}.

%

Our goal is to propose an efficient FW-type method for the non-oblivious 
optimization problem  \eqref{eq:formation}. Here, the efficiency is measured 
by the number of 
stochastic oracle queries, \ie, the sample complexity of $z$. 
As we mentioned earlier, among the stochastic variants of FW, the momentum 
stochastic Frank-Wolfe method proposed in 
\citep{mokhtari2017conditional,mokhtari2018stochastic} is the only method that 
requires only one sample per iteration. However, the stochastic oracle 
complexity of this  algorithm is suboptimal, \ie, 
$\mathcal{O}(1/\epsilon^3)$ stochastic queries are required for both convex 
minimization and monotone DR-submodular maximization problems. This suboptimal 
rate is due to the fact that the gradient estimator in momentum FW is biased 
and it is necessary to use a more conservative averaging parameter to control the effect of the bias term.

To resolve this issue, we propose a one-sample stochastic Frank-Wolfe method, 
called \sfw, which modifies the gradient approximation in momentum FW to ensure 
that the resulting gradient estimation is an unbiased estimator of the gradient 
(Section~\ref{sec:sfw}). This goal has been achieved by adding an unbiased 
estimator of the gradient variation  $\Delta_t = \nabla F(x_t) - \nabla 
F(x_{t-1})$ to the gradient approximation vector 
(Section~\ref{sec:sg_approx}). We later explain why coming up with an unbiased 
estimator of the gradient difference $\Delta_t$ could be a challenging task in the non-oblivious setting and 
show how we overcome this difficulty (Section~\ref{sec:grad_var_est}). We also 
characterize the convergence guarantees of \sfw for convex minimization, 
non-convex minimization, and monotone DR-submodular maximization 
(Section~\ref{sec:results}). In particular, we show that \sfw achieves the 
optimal  convergence rate  of $\mathcal{O}(1/\epsilon^2)$  for reaching an  
$\epsilon$-suboptimal solution in the stochastic convex setting, and  a 
$(1-1/e)-\epsilon$ approximate solution for a stochastic monotone DR-submodular 
maximization problem. Moreover, in a general non-convex setting, \sfw finds an 
$\epsilon$-first-order stationary point after at most 
$\mathcal{O}(1/\epsilon^3)$ iterations, achieving the current best known 
convergence rate.  Finally, we study the oblivious problem in \eqref{eq:ob_op} 
and show that our proposed \sfw method becomes significantly simpler and  the
corresponding theoretical results hold under less strict assumptions. For 
example, in the non-oblivious setting, we require second-order information as 
the nature of the problems requires; while in the oblivious setting, 
we only need access to first-order information (\cref{thm:no}).
We further highlight 
the similarities between the variance reduced method in 
\citep{cutkosky2019momentum} also known as STORM and the oblivious variant of 
\sfw. Indeed, our algorithm has been originally inspired by STORM.

\begin{table*}[t!]
	\begin{center}
		\caption{Convergence guarantees of stochastic Frank-Wolfe methods for 
			constrained convex minimization, non-convex
			minimization, and stochastic monotone continuous 
				DR-submodular function maximization.}\label{table-sfw}
		\begin{tabular}{cccccc}
			\specialrule{.15em}{.05em}{.05em}
			Function& Ref.&   Batch  
			&Complexity & 
			Non-oblivious &Utility\\
			\specialrule{.1em}{.05em}{.05em}
			Convex &\citep{hazan2012projection} &  $\mathcal{O} 
			(1/\epsilon^2)$  
			&$\mathcal{O}(1/ \epsilon^{4}) $ & \ding{55}&-\\
			\hline
			Convex &\citep{hazan2016variance} & $\mathcal{O} (1/\epsilon^2)$  
			&$\mathcal{O}(1/ \epsilon^{3}) $ &  \ding{55}&-\\
			\hline
			Convex &\citep{mokhtari2018stochastic} &  $1$ &$\mathcal{O}(1/ 
			\epsilon^{3}) $ & \ding{55}&-\\
			\hline
			Convex &\citep{yurtsever2019conditional}&  
			$\mathcal{O}(1/\epsilon)$ &$\mathcal{O}(1/\epsilon^2) $ & 
			\ding{55}&-\\
			\hline
			Convex &\citep{hassani2019stochastic}&  
			{{$\mathcal{O}(1/\epsilon)$}}  & 
			$\mathcal{O}(1/ \epsilon^{2}) $ &  \checkmark&-\\
			\hline
			Convex &\textbf{This diss.} &  1 & $\mathcal{O}(1/ \epsilon^{2}) $ 
			&  \checkmark &-\\			
			\specialrule{.1em}{.05em}{.05em} 
			Non-convex&\citep{hazan2016variance} &  $\mathcal{O} 
			(1/\epsilon^2)$ 
			&$\mathcal{O}(1/ \epsilon^{4}) $ &  \ding{55}&-\\
			\hline
			Non-convex&\citep{hazan2016variance}  & $\!\!\!\mathcal{O} 
			(1/\epsilon^{4/3})\!\!\!$ &$\mathcal{O}(1/ \epsilon^{10/3}) $ & 
			\ding{55}&-\\
			\hline
			Non-convex&\citep{shen2019complexities} & 
			{{$\mathcal{O}(1/\epsilon)$}}   
			&$\mathcal{O}(1/ \epsilon^{3}) $ &  \ding{55}&-\\
			\hline
			Non-convex&\citep{yurtsever2019conditional}&  
			{{$\mathcal{O}(1/\epsilon)$}}  
			&$\mathcal{O}(1/\epsilon^3) $ &  \ding{55}&-\\
			\hline
			Non-convex&\citep{hassani2019stochastic} &  
			{{$\mathcal{O}(1/\epsilon)$}}   
			&$\mathcal{O}(1/ \epsilon^{3}) $ &  \checkmark&-	\\
			\hline
			Non-convex&\textbf{This diss.} & 1 & $\mathcal{O}(1/ \epsilon^{3}) 
			$ & 
			\checkmark&-\\
			\specialrule{.1em}{.05em}{.05em}
			Submodular &\citep{hassani2017gradient} & 1 & 
			$O(1/{\epsilon^{2}})$ &\ding{55} & 
			$(1/2)\rm{OPT}$$-\epsilon$\\
			\hline
			Submodular &\citep{mokhtari2018stochastic}  & 1 & 
			$O(1/{\epsilon^{3}})$ &\ding{55} 
			&$(1-1/e)\rm{OPT}$$-\epsilon$ \\
			\hline
			Submodular &\citep{hassani2019stochastic}  
			&{{$\mathcal{O}(1/\epsilon)$}}& 
			{\color{black}{$\mathcal{O}(1/\epsilon^2)$}}& \checkmark
			&$(1-1/e)\rm{OPT}$$-\epsilon$\\
			\hline
			Submodular &\textbf{This diss.} & 1 &  $\mathcal{O}(1/ 
			\epsilon^{2}) $  &\checkmark 
			&$(1-1/e)\rm{OPT}$$-\epsilon$ \\
			\specialrule{.15em}{.05em}{.05em}
		\end{tabular}
	\end{center}
\end{table*}

Theoretical results of \sfw and other related works are 
summarized in \cref{table-sfw}. The 
complexity shows the required number of stochastic queries to obtain an 
$\epsilon$-suboptimal solution in convex case; an $\epsilon$-first-order 
stationary point in non-convex case; and an $\alpha\cdot \text{OPT}-\epsilon$ 
utility in monotone DR-submodular case, where $\alpha=1/2$ or $(1-1/e)$. These 
results show 
that  \sfw attains the best known complexity bounds in all the 
considered settings, while requiring only \emph{one single} stochastic oracle 
query per iteration and avoiding large batch sizes altogether. Even though the 
focus of this chapter is the fundamental theory behind \sfw, we provide some 
empirical evidence in \cref{sec:experiments-sfw}. All the proofs in this 
chapter are provided in \cref{sec:proof-1sfw}.



\section{One Sample SFW Algorithm}\label{sec:sfw}


\subsection{Stochastic Gradient Approximation}\label{sec:sg_approx}


In our work, we build on the momentum variance 
reduction approach proposed in 
\citep{mokhtari2017conditional,mokhtari2018stochastic}
to reduce the 
variance of the one-sample method. 
To be more precise, in the momentum FW method 
\citep{mokhtari2017conditional}, we update the gradient approximation $d_t$ 
at round 
$t$ as follows
\begin{equation}\label{eq:momentum}
d_t = (1-\rho_t) d_{t-1} + \rho_t \nabla \tF(x_t; z_t),
\end{equation}
where $\rho_t$ is the averaging parameter and $\nabla \tF(x_t; z_t)$ is a 
\emph{one-sample} estimation of the gradient. Since $d_t$ is a weighted 
average of the previous gradient estimation $d_{t-1}$ and the newly updated 
stochastic gradient, it has a lower variance comparing to one-sample 
estimation $\nabla \tF(x_t; z_t)$. In particular, it was shown in 
\citep{mokhtari2017conditional} that the variance of gradient approximation in 
\eqref{eq:momentum} approaches zero at a sublinear rate of $O(t^{-2/3})$. The 
momentum approach reduces the variance of gradient approximation, but it leads 
to a \emph{biased} gradient approximation, \ie, $d_t$  is not an 
unbiased estimator of the gradient $\nabla F(x_t)$. Consequently, it is 
necessary to use a conservative averaging parameter $\rho_t$ for momentum FW to control the effect of the bias term which leads to a 
sublinear error rate of $\mathcal{O}(t^{-1/3})$ and overall complexity of 
$\mathcal{O}(1/\epsilon^3)$.


To resolve this issue and come up with a faster momentum based FW method for 
the non-oblivious problem in \eqref{eq:formation}, we slightly modify the 
gradient estimation in \eqref{eq:momentum} to ensure that the resulting 
gradient estimation is an unbiased estimator of the gradient $\nabla F(x_t)$. 
Specifically, we add the term  
$\tilde{\Delta}_t$, which is an unbiased estimator of the gradient variation  
$\Delta_t = \nabla F(x_t) - \nabla F(x_{t-1})$, to $d_{t-1}$. This modification 
leads to the following gradient approximation
\begin{equation}\label{eq:momentum_new}
d_t = (1-\rho_t) (d_{t-1}+\tilde{\Delta}_t) + \rho_t \nabla \tF(x_t; z_t).
\end{equation}
To verify that $d_t $ is an unbiased estimator of $\nabla F(x_t)$ we can use a 
simple induction argument. Assuming that $d_{t-1}$ is an unbiased estimator of 
$\nabla F(x_t)$ and $\tilde{\Delta}_t$ is an unbiased estimator of $\nabla 
F(x_t) - \nabla F(x_{t-1})$ we have $\expect [d_t] = (1-\rho_t)( \nabla 
F(x_{t-1})+(\nabla F(x_t) - \nabla F(x_{t-1})))+\rho_t \nabla F(x_t)= \nabla 
F(x_t)$. Hence, the gradient approximation in \eqref{eq:momentum_new} leads to 
an unbiased approximation of the gradient. Let us now explain how to compute an 
unbiased estimator of the gradient variation $\Delta_t = \nabla F(x_t) - \nabla 
F(x_{t-1})$ in the non-oblivious setting. 

\subsection{Gradient Variation Estimation}\label{sec:grad_var_est}

The most natural approach for estimating the gradient variation  
$\Delta_t=\nabla F(x_t) - \nabla F(x_{t-1})$ using only one sample $z$ is 
computing the difference of two consecutive stochastic gradients, \ie, 
$\nabla \tF(x_t; z)-\nabla \tF(x_{t-1}; z)$. However, this approach 
leads to an unbiased estimator of the gradient variation  $\Delta_t$ only in 
the oblivious setting where $p(z)$ is independent of the choice of $x$, and 
would introduce bias in the more general non-oblivious case.  To better 
highlight this issue, assume that $z$ is sampled according to distribution 
$p(z;x_t)$. Note that $\nabla \tF(x_t; z)$ is an unbiased estimator of 
$\nabla F(x_t)$, \ie, $\expect[\nabla F(x_t; z)] = \nabla 
F(x_t)$, 
however, $\nabla \tF(x_{t-1};z)$ is not an unbiased estimator of $\nabla 
F(x_{t-1})$ since $p(z;x_{t-1})$ may be different from $p(z;x_t)$.

To circumvent this obstacle, an \emph{unbiased} estimator of $\Delta_t$ was 
introduced in \citep{hassani2019stochastic}. To explain their proposal for 
approximating the 
gradient variation using only one sample, note that the difference 
$\Delta_t=\nabla F(x_t) - \nabla F(x_{t-1})$ can be written as
\begin{equation*}
\begin{split}
\Delta_t &= \int_0^1 \nabla ^2 F(x_t(a))(x_t -x_{t-1}) \mathrm{d}a \\
&= \left[\int_0^1 \nabla ^2 F(x_t(a))\mathrm{d}a \right] (x_t -x_{t-1}),
\end{split}
\end{equation*}
where $x_t(a) = a x_t + (1-a)x_{t-1}$ for  $a\in [0,1]$. According to this 
expression, one can find an unbiased estimator of $\int_0^1 \nabla 
^2 F(x_t(a))\mathrm{d}a$ and use its product with $(x_t -x_{t-1})$ to 
find an unbiased estimator of $\Delta_t$. It can be easily verified that 
$\nabla ^2 F(x_t(a))(x_t -x_{t-1})$ is an unbiased estimator of 
$\Delta_t$ if $a$ is chosen from $[0,1]$ uniformly at random. Therefore, all we 
need is to come up with an unbiased estimator of the Hessian $\nabla ^2 F$.

By basic calculus, we can show that  $\forall\, x \in 
\constraint$ and $z$ with distribution $p(z;x)$, the matrix 
$\tilde{\nabla}^2F(x;z)$ defined as
\begin{align}\label{eq:gradient_diff_estimation2}
\tilde{\nabla}^2F(x;z) &= \tF(x;z)[\nabla \log p(z;x)][\nabla \log 
p(z;x)]^\top \nonumber\\
&\ + \nabla^2 \tF(x;z)  + [\nabla \tF (x; z)][\nabla \log 
p(z;x)]^\top\nonumber\\
&\ + \tF(x;z) \nabla^2 \log p(z;x)  \nonumber\\
&\ + [\nabla \log p(z;x)][\nabla \tF(x;z)]^\top,
\end{align}
is an \emph{unbiased} estimator of $\nabla^2F(x)$.  Note that the above 
expression requires only one sample of $z$. As a result, we can construct 
$\tilde{\Delta}_t$ as an unbiased estimator of ${\Delta}_t$ using only one 
sample
\begin{equation}\label{eq:gradient_diff_estimation}
\tilde{\Delta}_t \triangleq \tilde{\nabla}_t^2 (x_t - x_{t-1}),
\end{equation}
where $\tilde{\nabla}_t^2 = \tilde{\nabla}^2F(x_t(a);z_t(a))$,
 and $z_t(a)$ follows the distribution $p(z_t(a);x_t(a))$. By using this 
procedure, we can indeed compute the vector $d_t$ in \eqref{eq:momentum_new} 
with only one sample of $z$ per iteration.
Through a completely
different analysis from the ones in \citep{mokhtari2017conditional,hassani2019stochastic}, we 
show that the modified $d_t$ is still a good gradient estimation 
(\cref{lem:graident_error}), which allows
the establishment of the best known stochastic oracle
complexity for our proposed algorithm.

Another issue of this scheme is that in 
\eqref{eq:gradient_diff_estimation2} and \eqref{eq:gradient_diff_estimation}, 
we need to calculate 
$\nabla^2 \tF(x_t(a);z_t(a))\cdot$
\newline
$(x_t-x_{t-1})$ and $\nabla^2 \log 
p(x_t(a);z_t(a))(x_t-x_{t-1})$, where computation of Hessian is 
involved. When exact Hessian is not accessible, however, we can resort to an 
approximation by the difference of two gradients.
Precisely, for any function $\psi: \mathbb{R}^d \to \mathbb{R}$, any vector 
$u \in \mathbb{R}^d$ with $\|u \|\leq D = \max_{x,y \in \constraint}\|x-y\|$, 
and some $\delta >0$ small 
enough, we have
\begin{equation*}
\phi(\delta;\psi) \triangleq \frac{\nabla \psi(x+\delta u) - \nabla \psi(x 
- \delta u)}{2\delta} \approx \nabla^2 \psi(x) u.
\end{equation*}
If we assume that $\psi$ is $L_2$-second-order smooth, \ie, 
$\|\nabla^2\psi(x) - \nabla^2\psi(y)) \| \leq L_2 \|x-y \|,\ \forall\, 
x, y \in \mathbb{R}^d$, we can 
upper bound the 
approximation error quantitatively:
\begin{align}\label{eq:gradient_difference_bound}
\|\nabla^2 \psi(x) u - \phi(\delta;\psi) \| = \|\nabla^2 
\psi(x) u - 
\nabla^2 \psi(\tilde{x}) u) \| \leq D^2L_2\delta,
\end{align}
where $\tilde{x}$ is obtained by the mean-value theorem. In other words, the 
approximation error can be sufficiently small for proper $\delta$. So we can 
estimate $\Delta_t$ by
\begin{align}\label{eq:gradient_diff_estimation3}
\tilde{\Delta}_t &= \tF(x;z)[\nabla \log p(z;x)][\nabla \log 
p(z;x)]^\top u_t \nonumber\\
&\ + \phi(\delta_t,\tF(x;z))  + [\nabla \tF (x; z)][\nabla \log 
p(z;x)]^\top u_t \nonumber\\
&\ + \tF(x;z) \phi(\delta_t,\log p(z,x))  \nonumber\\
&\ + [\nabla \log p(z;x)][\nabla \tF(x;z)]^\top u_t,
\end{align}
where $u_t = x_t -x_{t-1}$, $x, z, \delta_t$ are chosen 
appropriately. We also 
note that since computation of gradient difference has a computational 
complexity of $\mathcal{O}(d)$, while that for Hessian is $\mathcal{O}(d^2)$, 
this approximation strategy can also help to accelerate the optimization 
process.

\subsection{Variable Update}

Once the gradient approximation $d_t$ is computed, we can follow the update of 
conditional gradient methods for computing the iterate $x_t$. In this section, 
we introduce two different schemes for updating the iterates depending on the 
problem that we aim to solve. 

For minimizing a general (non-)convex function using one sample stochastic FW, 
we update the iterates according to 
\begin{equation}\label{eq:option1}
x_{t+1} = x_{t}+\eta_{t} (v_t-x_{t}),
\end{equation} 
where $v_t = 
\argmin_{v \in \constraint}\{v^\top d_{t}\}$. In this case, we find the 
direction that minimizes the inner product with the 
current gradient approximation $d_t$ over the constraint set $\constraint$, 
and update the variable $x_{t+1}$ by descending in the direction of 
$v_t-x_{t}$ with step size $\eta_t$.

For monotone DR-submodular maximization, the update rule is slightly different, 
and a stochastic variant of the continuous greedy method 
\citep{vondrak2008optimal} can be used. 
Using the same stochastic estimator $d_t$ as in the (non-)convex case, the 
update rule for DR-Submodular optimization is given by
\begin{equation}\label{eq:option2}
x_{t+1} = x_{t}+ \eta_t v_t, 
\end{equation}
where $v_t = 
\argmax_{v \in \constraint}\{v^\top d_{t}\}, \eta_t=1/T, T$ is the total 
number of 
iterations.
Hence, if we start from the origin, after $T$ steps the outcome 
will be a feasible point as it can be written as the average of $T$ feasible 
points. 

The description of  our proposed \sfw method for smooth (non-)convex 
minimization as well as monotone DR-submodular maximization is outlined in 
\cref{alg:one_sample}.

\begin{algorithm}[t]
	\caption{One-Sample SFW (\sfw)}
	\begin{algorithmic}[1]
		\Require Step sizes $\rho_t \in 
		(0,1),   
		\eta_t \in (0,1)$, initial point $x_1 \in 
		\constraint$, total number 
		of iterations $T$
		\Ensure $x_{T+1}$ or $x_o$, where $x_o$ is chosen from $\{x_1, 
		x_2,\cdots, x_T\}$ uniformly at random  
		\For{$t=1,2,\dots, T$}
		\If{$t=1$} 
		\State Sample a point $z_1$ according 
		to $p(z_1,x_1)$
		\State  Compute $d_1 = \nabla \tilde{F}(x_1; 
		z_1)$		
		\Else
		\State Choose $a$ uniformly at random from $[0,1]$
		\State Compute $x_t(a) = a x_t + (1-a)x_{t-1}$
		\State Sample a point $z_t$ according to $p(z; x_t(a))$
		\State Compute $\tilde{\Delta}_t$ either by $\tnabla_t^2 = \tnabla^2 
		F(x_t(a); z_t)$ based on \eqref{eq:gradient_diff_estimation2} and 
		$\tilde{\Delta}_t =\tilde{\nabla}_t^2 (x_t - x_{t-1})$ (Exact 
		Hessian Option); or by \cref{eq:gradient_diff_estimation3} with 
		$x=x_t(a), z=z_t$ (Gradient Difference Option)
		\State $d_t = (1-\rho_t)(d_{t-1}+\tilde{\Delta}_t) + \rho_t \nabla 
		\tilde{F}(x_t,z_t)$
		\EndIf
		\State (non-)convex minimization: Update $x_{t+1}$ based on 
		\eqref{eq:option1} 
		\State  DR-submodular maximization: Update $x_{t+1}$ based on 
		\eqref{eq:option2} 
		\EndFor
	\end{algorithmic}
	\label{alg:one_sample}
\end{algorithm}

\section{Main Results}\label{sec:results}
Before presenting the convergence results of our algorithm, we first state our 
assumptions on the constraint set $\constraint$, the 
stochastic function $\tF$, and the distribution $p(z;x)$. 

\begin{assump}\label{assum:constraint}
	The constraint set $\constraint \subseteq \mathbb{R}^d$ is compact with 
	diameter $D = \max_{x,y \in \constraint}\|x-y\|$, and radius $R = 
	\max_{x\in\constraint} \|x\|$. 
\end{assump}


\begin{assump}\label{assum:stoch_bound}
The stochastic function $\tF(x;z)$ has uniformly bounded
function value, \ie, $| \tF(x;z)| \leq B$ for all $x \in 
\constraint,  z \in \mathcal{Z}$. 
\end{assump}

\begin{assump}\label{assum:gradient_norm}
The stochastic gradient $\nabla \tF$ has uniformly bound norm: $\|\nabla 
\tF(x;z) \| \leq G_{\tF}, \forall\, x \in \constraint, \forall\, z \in 
\mathcal{Z}$. The norm of the gradient of $\log p$ has bounded fourth-order 
moment: $\expect_{z \sim p(z;x)} \|\nabla \log p(z;x) \|^4 \leq 
G_p^4$. We also define $G=\max\{G_\tF, G_p\}$.
\end{assump}

\begin{assump}\label{assum:second_order}
The stochastic Hessian $\nabla^2 \tF$ has uniformly bounded spectral norm: $\| 
\nabla^2 \tF(x;z)\| \leq L_{\tF}, \forall\, x \in \constraint, \forall\, 
z \in \mathcal{Z}$. The spectral norm of the Hessian of $\log p$ has bounded 
second-order moment: $\expect_{z \sim p(z;x)} \|\nabla^2 \log p(z;x) 
\|^2 \leq L_p^2$. We also define $L = \max \{L_\tF, L_p\}$. 
\end{assump}

We note that in Assumptions~\ref{assum:stoch_bound}-\ref{assum:second_order}, we 
assume that the stochastic function $\tF$ has uniformly bounded function 
value, gradient norm, and second-order differential. 
We also note that all these assumptions are necessary, and not restrictive. We 
elaborate on the reasons as below:
\begin{itemize}
	\item \cref{assum:constraint}: The compactness of the feasible set has been 
	assumed in all 
	projection-free papers. It is indeed needed for the convergence of the 
	linear 
	optimization subroutine in the Frank-Wolfe method, otherwise, $v_t$ in 
	\eqref{eq:option1} can 
	be unbounded.
	
	\item \cref{assum:gradient_norm,assum:second_order} about $\tilde F$: 
	Bounded gradient and Hessian of the 
	stochastic function $\tilde F$ are the customary assumptions for all the 
	variance reduction methods when we solve the problem over a compact set. 
	The boundedness of the function values (\cref{assum:stoch_bound}) is a 
	direct implication of bounded gradient and compact constraint set. 
	
	
	\item \cref{assum:gradient_norm,assum:second_order} about the distribution 
	$p$: 
	We emphasize these assumptions hold 
	trivially for the oblivious setting \eqref{eq:ob_op}, where $p$ is 
	not a function of the variable $x$. For the non-oblivious case 
	\eqref{eq:formation}, consider the reinforcement learning as an 
	example where $p$ is the 
	distribution of a trajectory given the policy parameter $x$. It can be 
	verified that for common Gaussian policy with bounded mean and variance, 
	the smoothness of the parameterization of the policy (\eg, neural network 
	with smooth activation function) can imply
	\cref{assum:gradient_norm,assum:second_order}.

\end{itemize}

Now with these 
assumptions,  we can establish an upper bound for the second-order moment of 
the spectral norm of the Hessian estimator $\tnabla^2 F(x;z)$
in \eqref{eq:gradient_diff_estimation2}.

\begin{lemma}\label{lem:scg++8.1}[Lemma 7.1 of 
	\citep{hassani2019stochastic}]
Under Assumptions~\ref{assum:stoch_bound}-\ref{assum:second_order}, for all 
$x \in \constraint$, we have
\begin{equation*}
\begin{split}
\expect_{z \sim p(z;x)}[\| \tnabla^2F(x;z) \|^2] \leq 4B^2G^4 + 
16G^4 + 4L^2 + 4B^2L^2 
\triangleq \bar{L}.
\end{split}
\end{equation*}
\end{lemma}

Note that the result in Lemma \eqref{lem:scg++8.1} also implies the $\bar{L}$-smoothness of $F$, since 
\begin{align*}
\|\nabla^2 
F(x) \|^2 &= \|\expect_{z \sim p(z;x)}[\tnabla^2 F(x;z)] \|^2\\
& \leq 
\expect_{z \sim p(z;x)}[\|\tnabla^2 F(x;z) \|^2] \\
& \leq \bar{L}^2.
\end{align*}
In other words, the conditions in 
Assumptions~\ref{assum:stoch_bound}-\ref{assum:second_order} implicitly imply 
that the objective function $F$ is $\bar{L}$-smooth. 

To establish the convergence guarantees for our proposed \sfw algorithm, the 
key step is to derive an upper bound on the errors of the estimated gradients. To do so, we prove the following lemma, which provides the required upper bounds in 
different settings of parameters.

\begin{lemma}\label{lem:graident_error}
Consider the gradient approximation $d_t$ defined in \eqref{eq:momentum_new}. 
Under Assumptions~\ref{assum:constraint}-\ref{assum:second_order},
 if we run Algorithm~\ref{alg:one_sample} with Exact Hessian Option in Line 9, 
 and with parameters $\rho_t = (t-1)^{-\alpha}\ (\forall\, t \ge 2)$, and 
$\eta_t \leq t^{-\alpha}\ (\forall\, t \ge 1$ and for some $\alpha \in (0,1])$, 
then the gradient estimation $d_t$ satisfies
\begin{equation}
\expect[\|\nabla F(x_t) - d_t\|^2] \leq C t^{-\alpha},
\end{equation}
where 
 $C\! =\! \max 
 \!\left\{\!\frac{2(2G\!+\!D\bar{L})^2}{2\!-\!2^{-\alpha}\!-\!\alpha}, 
\left[\frac{2}{2\!-\!2^{-\alpha}\!-\!\alpha} \right]^4\!\!, 
[2D(\bar{L}\!+\!L)]^4 \!\right\}\!.$	
\end{lemma}

Lemma \eqref{lem:graident_error} shows that with an appropriate parameter 
setting, the 
gradient error converges to zero at a rate of $\mathcal{O}(t^{-\alpha})$. With 
this 
unifying upper bound, we can obtain the convergence rates of our algorithm for 
different kinds of objective functions. 

If in the update of \sfw we use the Gradient Difference Option in Line 9 of 
\cref{alg:one_sample} to estimate $\tilde{\Delta}_t$, as pointed out above, we 
need one further assumption on second-order smoothness of the functions $\tF$ 
and $\log p$. 
\begin{assump}\label{assum:hessian}
The stochastic function $\tF$ is uniformly $L_{2,\tF}$- second-order smooth: 
$\|\nabla^2 \tF(x;z) - 
\nabla^2 \tF(y;z) \|\leq L_{2,\tF}\|x-y \|,\ \forall\, x,y \in 
\constraint, \forall\, z \in \mathcal{Z}$. The log probability $\log 
p(z;x)$ is uniformly $L_{2,p}$-second-order smooth: $\|\nabla^2 \log 
p(z;x) - 
\nabla^2 \log p(z;y) \|\leq L_{2,p}\|x-y \|,\ \forall\, x,y \in 
\constraint, \forall\, z \in \mathcal{Z}$. We also define $L_2 = 
\max\{L_{2,\tF}, L_{2,p}\}$.
\end{assump}

We note that under \cref{assum:hessian}, the approximation bound in 
\eqref{eq:gradient_difference_bound} holds for both 
$\tF$ and $\log p$. So for $\delta_t$ sufficiently small, the error introduced 
by the Hessian approximation can be ignored. Thus similar upper bound for 
errors of estimated gradients still holds. 
\begin{lemma}\label{lem:graident_error_2}
	Consider the gradient approximation $d_t$ defined in 
	\eqref{eq:momentum_new}. Under 
	Assumptions~\ref{assum:constraint}-\ref{assum:hessian},
	if we run Algorithm~\ref{alg:one_sample} with Gradient Difference Option in 
	Line 9, and with parameters $\rho_t = (t-1)^{-\alpha}, \delta_t = 
	\frac{\sqrt{3}\eta_{t-1\bar{L}}}{DL_2(1+B)}\ (\forall\, t \ge 2)$, and 
	$\eta_t \leq t^{-\alpha}\ (\forall\, t \ge 1$ and for some $\alpha \in 
	(0,1])$, 
	then the gradient estimation $d_t$ satisfies
	\begin{equation}
	\expect[\|\nabla F(x_t) - d_t\|^2] \leq C t^{-\alpha},
	\end{equation}
	where 
	$C= \max \
	\!\bigg\{\!\frac{8(D^2\bar{L}^2+G^2+GD\bar{L})}{2\!-\!2^{-\alpha}\!-\!\alpha},
	\left(\frac{2}{2\!-\!2^{-\alpha}\!-\!\alpha} \right)^4\!\!, 
	 (4D(\bar{L}\!+\!L))^4 \bigg\}.$
\end{lemma}
\cref{lem:graident_error_2} shows that with Gradient Difference Optionin Line 9 
of \cref{alg:one_sample}, the error of estimated gradient has the same 
order of convergence rate as that with Exact Hessian Option. So in the 
following three subsections, we will present the theoretical results of our proposed \sfw 
algorithm with Exact Hessian Option, for convex minimization, non-convex 
minimization, and monoton DR-submodular maximization, respectively. The results 
of Gradient Difference Option only differ in constant factors.

\subsection{Convex Minimization}
For convex minimization problems, to obtain an $\epsilon$-suboptimal solution, 
\cref{alg:one_sample} only requires at most $\mathcal{O}(1/\epsilon^2)$ 
stochastic oracle queries, and $\mathcal{O}(1/\epsilon^2)$ linear optimization 
oracle calls. Or precisely, we have 
\begin{theorem}[Convex]\label{thm:convex}
Consider the \sfw method outlined in \cref{alg:one_sample}  
with Exact Hessian Option in Line 9. Further, suppose the conditions in 
Assumptions~\ref{assum:constraint}-\ref{assum:second_order} hold, and assume 
that $F$ is convex on $\constraint$. If we set the algorithm parameters 
as $\rho_t = 
 (t-1)^{-1}$ and $\eta_t = t^{-1}$, then the output $x_{T+1} \in \constraint$ 
 is feasible and satisfies
\begin{equation*}
\expect[F(x_{T+1}) - F(x^*)] \leq \frac{2\sqrt{C}D}{\sqrt{T}} + 
\frac{\bar{L}D^2(1+\ln T)}{2T} ,
\end{equation*}
where $C=\max \{4(2G+D\bar{L})^2, 256, [2D(\bar{L}+L)]^4 \}$, and $x^*$ is a 
minimizer of $F$ on $\constraint$. 	
\end{theorem}

The result in Theorem~\ref{thm:convex} shows that the proposed one sample 
stochastic Frank-Wolfe method, in the convex setting, has an overall complexity 
of $\mathcal{O}(1/\epsilon^2)$ for finding an $\epsilon$-suboptimal solution. 
Note that to prove this claim we used the result in 
Lemma~\ref{lem:graident_error} for the case where $\alpha=1$, \ie, the 
variance of gradient approximation converges to zero at a rate of 
$\mathcal{O}(1/t)$. We also highlight that \sfw is \emph{parameter-free}, as 
the learning rate $\eta_t$ and the momentum parameter $\rho_t$ do not depend on 
the parameters of the problem.

\subsection{Non-Convex Minimization}
For non-convex minimization problems, showing that the gradient norm approaches 
zero, \ie, $\|\nabla F(x_t) \| \to 0$, implies convergence to a 
stationary point in the \emph{unconstrained} setting. Thus, it is usually used 
as a measure for convergence. In the constrained setting, however, the norm of 
gradient is not a proper measure for defining stationarity and we instead use 
the Frank-Wolfe Gap 
\citep{jaggi2013revisiting,lacoste2016convergence}, which is defined by
\begin{equation*}
\mathcal{G}(x) = \max_{v \in \constraint} \langle v-x,-\nabla F(x) 
\rangle.
\end{equation*}
We note that by definition, $\mathcal{G}(x) \geq 0,  \forall\, x \in 
\constraint$. If some point $x \in \constraint$ satisfies $\mathcal{G}(x) = 
0$, then it is a first-order stationary point.

In the following theorem, we formally prove the number of iterations required 
for one sample stochastic FW to find an $\epsilon$-first-order stationary point 
in expectation, \ie, a point $x$ that satisfies 
$\expect[\mathcal{G}(x)] \leq \epsilon$. 

\begin{theorem}[Non-Convex]\label{thm:nonconvex}
Consider the \sfw method outlined in Algorithm~\ref{alg:one_sample}  
with Exact Hessian Option in Line 9. Further, suppose the conditions in 
Assumptions~\ref{assum:constraint}-\ref{assum:second_order} hold. If we 
set the algorithm parameters as 
$\rho_t 
= (t-1)^{-2/3},$ and $\eta_t = T^{-2/3}$, then the 
output $x_{o} \in \constraint$ is feasible and satisfies
\begin{equation*}
\expect[\mathcal{G}(x_o)] \leq \frac{2B+3\sqrt{C}D/2}{T^{1/3}} + 
\frac{\bar{L}D^2}{2T^{2/3}},
\end{equation*}
where 
$C\! = \!\max \left\{\frac{2(2G\!+\!D\bar{L})^2}{\frac{4}{3}-2^{-\frac{2}{3}}}, 
\left[ 
\frac{2}{\frac{4}{3}-2^{-\frac{2}{3}}}\right]^4\!\!, [2D(\bar{L}\!+\!L)]^4 
\right\}\!.$
\end{theorem}
We remark that Theorem~\eqref{thm:nonconvex} shows that Algorithm~\ref{alg:one_sample} finds an 
$\epsilon$-first-order stationary points after at most 
$\mathcal{O}(1/\epsilon^3)$ iterations, while uses exactly one stochastic 
gradient per iteration. Note that to obtain the best performance guarantee in 
Theorem~\eqref{thm:nonconvex}, we used the result of 
Lemma~\ref{lem:graident_error} for the case where $\alpha=2/3$, \ie, the 
variance of gradient approximation converges to zero at a rate of 
$\mathcal{O}(T^{-2/3})$. Again, we highlight that \sfw is a 
\emph{parameter-free} algorithm.

\subsection{Monotone DR-Submodular Maximization}

In this subsection, we focus on the convergence properties of one-sample 
stochastic Frank-Wolfe or one-sample stochastic Continuous Greedy for solving a 
monotone continuous DR-submodular maximization problem.
%

Recall that for monotone continuous DR-submodular maximization, approximated 
solution within a factor of $(1-e^{-1} + \epsilon)$ can not be obtained in 
polynomial time \citep{bian2017guaranteed}. 
To achieve a $(1-e^{-1})\mathrm{OPT}- \epsilon$ approximation guarantee, 
\sfw requires at most $\mathcal{O}(1/\epsilon^2)$ 
stochastic oracle queries, and $\mathcal{O}(1/\epsilon^2)$ linear optimization 
oracle calls, which are the lower bounds of the complexity established in 
\cite{hassani2019stochastic}.
 
\begin{theorem}[Submodular]\label{thm:sub}
Consider the \sfw method outlined in Algorithm~\ref{alg:one_sample}  
with Exact Hessian Option in Line 9 for maximizing DR-Submodular functions. 
Further, suppose the conditions in 
Assumptions~\ref{assum:constraint}-\ref{assum:second_order} hold, and further 
assume that $F$ is monotone and continuous DR-submodular on 
the positive orthant. If we set the algorithm parameters as $x_1=0, 
\rho_t = (t-1)^{-1}, \eta_t = T^{-1}$, then the output $x_{T+1} \in 
\constraint$ is feasible and satisfies
\begin{equation*}
\expect[F(x_{T+1})] \geq (1-e^{-1})F(x^*) - \frac{4R\sqrt{C}}{T^{1/2}} - 
\frac{\bar{L}R^2}{2T},
\end{equation*}
where $C=\max\{4(2G+R\bar{L})^2, 256, [2R(\bar{L}+L)]^4 \}$.
\end{theorem}

Finally, we note that \cref{alg:one_sample} can also be used to solve 
stochastic discrete submodular maximization 
\citep{karimi2017stochastic,mokhtari2017conditional}. Precisely, we can apply 
\cref{alg:one_sample} on the multilinear extension of the discrete submodular 
functions, and round the output to a feasible set by lossless rounding schemes 
like pipage rounding \citep{calinescu2011maximizing} and contention resolution 
method \citep{chekuri2014submodular}.

\section{Oblivious Setting}\label{sec:obl}

In this section, we specifically study the oblivious problem introduced in 
\eqref{eq:ob_op} which is a special case of the non-oblivious problem defined 
in \eqref{eq:formation}. In particular, we show that our proposed \sfw method 
becomes significantly simpler and  the
corresponding theoretical results hold under less strict assumptions.


\subsection{Algorithm}

As we discussed in Section \ref{sec:sfw}, a major challenge that we face for 
designing a variance reduced Frank-Wolfe method for the non-oblivious setting 
is computing an unbiased estimator of the gradient variation $\Delta_t=\nabla 
F(x_t)-\nabla F(x_{t-1})$. This is indeed not problematic in the oblivious 
setting, as in this case
$z \sim p(z)$ is independent of $x$ and therefore $\nabla \tF(x_t;z) - \nabla 
\tF(x_{t-1};z)$ is an unbiased estimator of  the gradient variation $\Delta_t = 
\nabla F(x_t) - \nabla F(x_{t-1})$. Hence, in the oblivious setting, our 
proposed one sample FW uses the following gradient approximation 
\begin{equation*}\label{eq:momentum_new_2}
d_t = (1-\rho_t) (d_{t-1}+\tilde{\Delta}_t) + \rho_t \nabla \tF(x_t; z_t),
\end{equation*}
 where $\tilde{\Delta}_t$ is given by 
 $$\tilde{\Delta}_t = \nabla \tF(x_t;z_t) - \nabla 
		\tF(x_{t-1};z_t).$$	
 The rest of the algorithm for updating the variable $x_t$ is identical to the 
 one for the non-oblivious setting. The description of our proposed algorithm 
 for the oblivious setting is outlined in \cref{alg:one_sample_ob}.

\begin{algorithm}[t]
	\caption{One-Sample SFW (Oblivious Setting)}
	\begin{algorithmic}[1]
		\Require Step sizes $\rho_t \in 
		(0,1),   
		\eta_t \in (0,1)$, initial point $x_1 \in 
		\constraint$, total number 
		of iterations $T$
		\Ensure $x_{T+1}$ or $x_o$, where $x_o$ is chosen from $\{x_1, 
		x_2,\cdots, x_T\}$ uniformly at random  
		\For{$t=1,2,\dots, T$}
		\State Sample a point $z_t$ according 
		to $p(z)$
		\If{$t=1$} 
		\State  Compute $d_1 = \nabla \tilde{F}(x_1; 
		z_1)$		
		\Else
		\State $\tilde{\Delta}_t = \nabla \tF(x_t;z_t) - \nabla 
		\tF(x_{t-1};z_t)$
		\State $d_t = (1-\rho_t)(d_{t-1}+\tilde{\Delta}_t) + \rho_t \nabla 
		\tilde{F}(x_t,z_t)$
		\EndIf
		\State (non-)convex minimization: Update $x_{t+1}$ based on 
		\eqref{eq:option1} 
		\State  DR-submodular maximization: Update $x_{t+1}$ based on 
		\eqref{eq:option2} 
		\EndFor
	\end{algorithmic}
	\label{alg:one_sample_ob}
\end{algorithm}

\begin{remark}
We note that by rewriting our proposed \sfw method for the
oblivious setting, we recover the variance reduction technique STORM 
\citep{cutkosky2019momentum} with different 
sets of parameters. In \citep{cutkosky2019momentum}, however, the STORM 
algorithm was combined with SGD to solve \emph{unconstrained} non-convex 
minimization problems, while our proposed \sfw method solves convex 
minimization, non-convex minimization, and DR-submodular maximization in a 
\emph{constrained} setting.\end{remark}

\subsection{Theoretical Results}

In this subsection, we show that the variant of one sample stochastic FW for 
the oblivious setting (described in  \cref{alg:one_sample_ob}) recovers the 
theoretical results for the non-oblivious setting with fewer assumptions. In 
particular, we only require the following condition for the stochastic 
functions $\tilde{F}$ to prove our main results. 

\vspace{2mm}

\begin{assump}\label{assum:function_ob}
The function $\tF$ has uniformly bound gradients, \ie, $ \forall\, x \in 
\constraint, \forall\, z \in 
\mathcal{Z}$,
$$\|\nabla \tF(x;z) \| \leq G.$$
Moreover, the function $\tF$  is uniformly  $L$-smooth, \ie, $\forall\, x, 
y 
\in \constraint, \forall\, z \in 
\mathcal{Z}$,
$$\|\nabla \tF(x;z) - \nabla \tF(y;z) \|\leq L \| 
x-y\|.$$ 
\end{assump}

We note that as direct corollaries of \cref{thm:convex,thm:nonconvex,thm:sub}, 
\cref{alg:one_sample_ob} achieves the same convergence rates, which is 
stated in 
\cref{thm:no} formally.

\begin{theorem}\label{thm:no}
Consider the oblivious variant of \sfw outlined in 
Algorithm~\ref{alg:one_sample_ob}, and assume that the conditions in 
\cref{assum:constraint,assum:stoch_bound,assum:function_ob} hold. Then we have
\begin{enumerate}
\item If $F$ is convex on $\constraint$, and we set $\rho_t = 
(t-1)^{-1}$ and $\eta_t = t^{-1}$, then the output $x_{T+1} \in 
\constraint$ is feasible and satisfies
\begin{equation*}
\expect[F(x_{T+1}) - F(x^*)] \leq \mathcal{O}(T^{-1/2}).
\end{equation*}

\item If $F$ is non-convex, and we set $\rho_t 
= (t-1)^{-2/3},$ and $\eta_t = T^{-2/3}$, then the 
output $x_{o} \in \constraint$ is feasible  and satisfies
\begin{equation*}
\expect[\mathcal{G}(x_o)] \leq \mathcal{O}(T^{-1/3}).
\end{equation*}

\item If $F$ is monotone DR-submodular on $\constraint$, and we set $x_1=0, 
\rho_t = 
(t-1)^{-1}$ and $ \eta_t = T^{-1}$, then the output $x_{T+1} \in 
\constraint$ is feasible and satisfies
\begin{equation*}
\expect[F(x_{T+1})] \geq (1-e^{-1})F(x^*) - \mathcal{O}(T^{-1/2}).
\end{equation*}
\end{enumerate}
\end{theorem} 

\cref{thm:no} shows that the oblivious version of \sfw requires at most $\mathcal{O}(1/\epsilon^2)$ stochastic oracle queries to find an $\epsilon$-suboptimal solution for convex minimization, at most $\mathcal{O}(1/\epsilon^2)$ stochastic gradient evaluations to achieve a $(1-1/e)-\epsilon$ approximate solution for monotone DR-submodular maximization, and at most 
$\mathcal{O}(1/\epsilon^3)$ stochastic oracle queries  to find an 
$\epsilon$-first-order stationary point for non-convex minimization.



\section{Experiments}\label{sec:experiments-sfw}
In this section, we empirically validate the efficiency of the proposed \sfw 
algorithm by 
comparing it with the baseline methods: Stochastic Frank-Wolfe (SFW) 
\cite{hazan2016variance} and Stochastic Conditional Gradient (SCG) 
\cite{mokhtari2018stochastic}.
Note that SCG is the only existing provably convergent Frank-Wolfe variant that 
accepts a constant per-iteration mini-batch size (possibly 1).
Denote the constant mini-batch size of \sfw and SCG by $m$.
The growing mini-batch size of SFW is set to $m\cdot t^2$, where $t$ is the 
iteration count.

We study three types of problems, \emph{i.e.}, $\ell_1$-constrained 
logistic-regression 
(convex), robust low rank matrix recovery (non-convex), and maximization of 
multilinear extensions of monotone discrete submodular functions  
(DR-submodular).
\subsection{Logistic Regression}
In this task, we consider $\ell_1$-constrained logistic regression problem.
Concretely, denote each data point $i$ by $(a_i, y_i)$, where $a_i \in \RBB^d$ 
is a feature vector and $y_i \in \{1, \dots, C\}$ is the 
corresponding label.
Our goal is to minimize the following loss
$$
\begin{aligned}
F(\WB)
= \frac{1}{n}\sum_{i=1}^n \mathrm{log}(1+ \mathrm{exp}(-y_i\WB_c^T a_i)),
\end{aligned}
$$
over the constraint $\CM = \{\WB \in \RBB^{d \times C}: \|\WB \|_1 \le r \}$ 
for some constant $r \in \RBB_+$, where $\|\WB\|_1$ is the matrix $\ell_1$ 
norm, \emph{i.e.}, $\|\WB\|_1 = \max_{1 \le j \le C} \sum_{i=1}^d |[\WB]_{ij}|$.
We note that the loss function $F$ is convex and smooth.

Two datasets are used in our experiments: MNIST (digit 2 and 4 as positive and 
negative class respectively) and CIFART10 (cat and dog as positive and negative 
class respectively).
In terms of the parameter setting, we grid search the step size $\eta_t$ for 
all three 
methods over the set $\{\min\{1, c/(t+1)^a\}|c\in\{0.1, 0.25, 0.5, 1.0, 2.0\}, 
a\in\{1, 2/3, 1/2\} \}$, set the mixing weights $\rho_t$ of SCG and \sfw to 
$1/(t+1)^{2/3}$, and set the constant mini-batch parameter $m = 16$.
We report the results in Figure \ref{fig_lr}. We can see the advantage of \sfw 
over its competitors.

\begin{figure*}[t]
	\centering 
	\begin{subfigure}[t!]{0.45\linewidth}
		\includegraphics[width=\columnwidth]{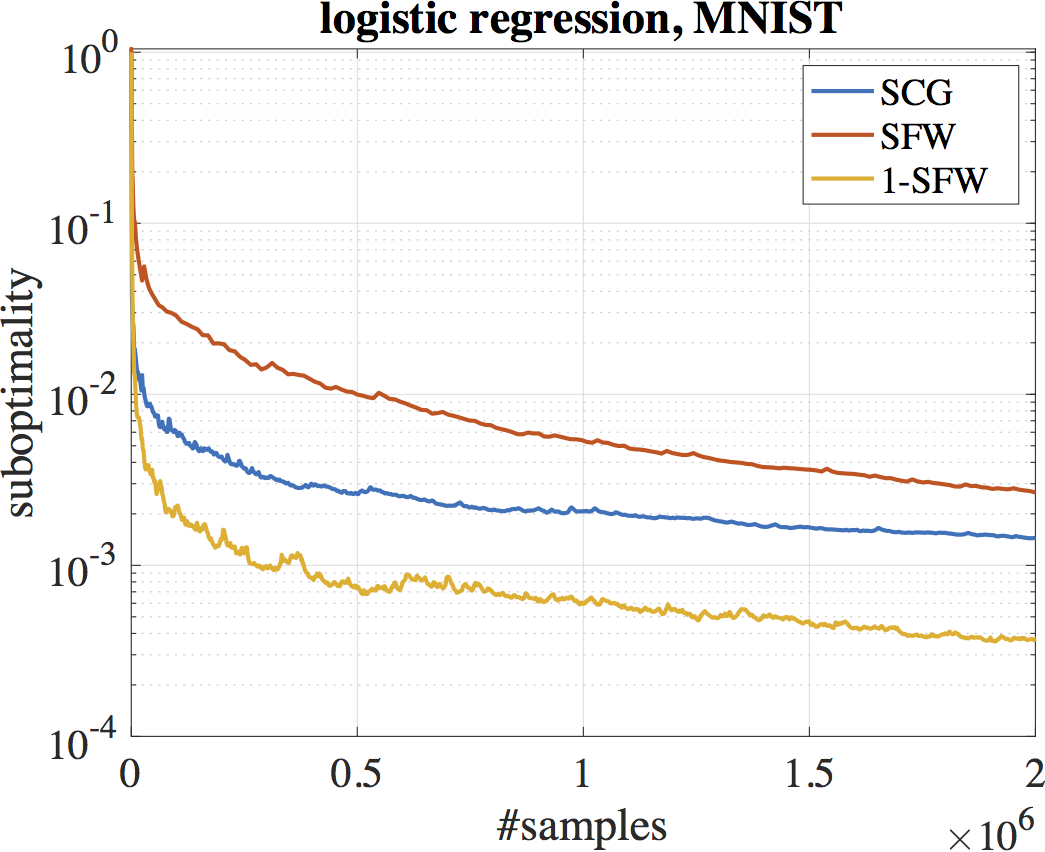}
		\caption{MNIST}
		\label{fig:sfw-1-1}
	\end{subfigure}
	\begin{subfigure}[t!]{0.45\linewidth}
		\includegraphics[width=\columnwidth]{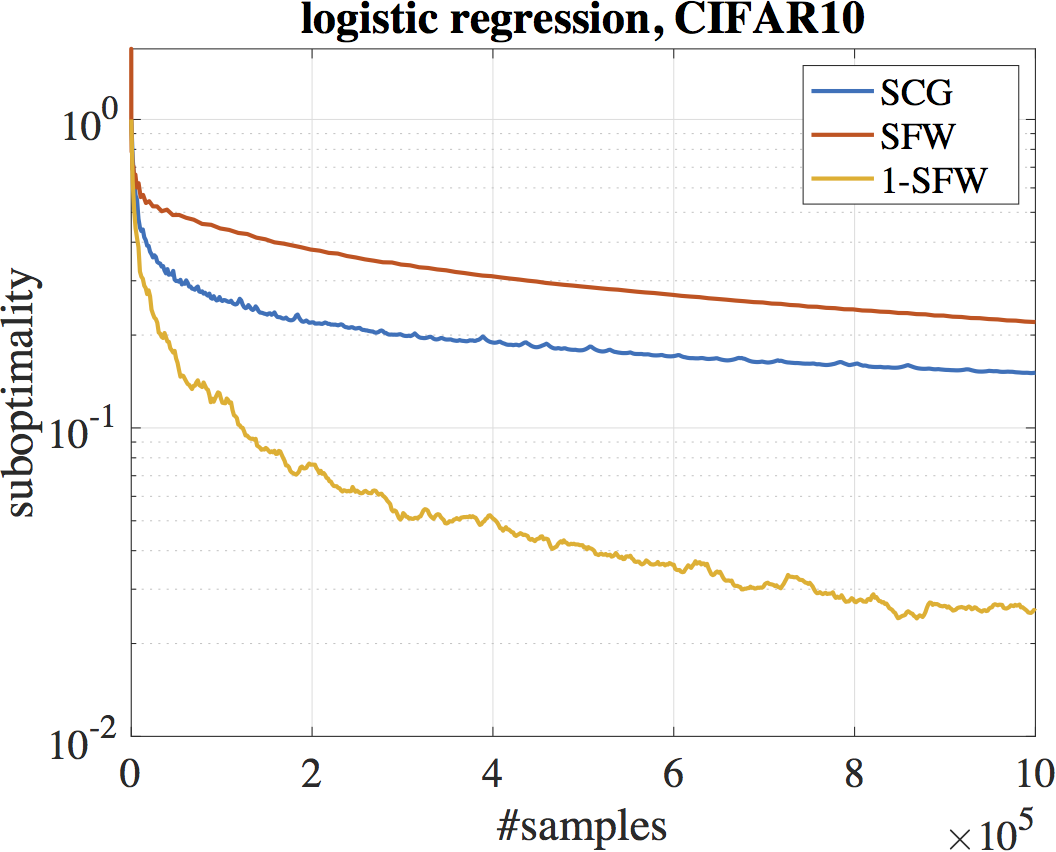}
		\caption{CIFAR10}
		\label{fig:sfw-1-2}
	\end{subfigure}
\caption{Logistic Regression. (a) uses digit 2 and 4 in MNIST, (b) uses cat 
	and dog in CIFAR10.}
\label{fig_lr}
\end{figure*}

%
%
%

\subsection{Robust Low-Rank Matrix Recovery}
LRMR plays a key role in solving many important learning tasks, such as 
collaborative filtering \citep{koren2009matrix}, dimensionality reduction 
\citep{weinberger2006unsupervised}, and multi-class learning 
\citep{xu2013speedup}.
The loss of LRMR is defined as
\begin{equation*}
\begin{aligned}
\min_{\XB\in \RBB^{M\times N}}& \quad \sum_{(i,j)\in\Omega} 
\psi(\XB_{ij}-\YB_{ij}), \\
\text{subject to} & \quad \quad \|\XB\|_*\leq B,
\end{aligned}
\label{eqn_LRMR}
\end{equation*}
where $\psi: \RBB\rightarrow\RBB$ is the potentially non-convex empirical loss 
function, $\XB_{ij}$ is the $(i,j)$-th element of matrix $\XB$, and $\Omega$ is 
the set of observed indices in target matrix $\YB \in \RBB^{M\times N}$.
Here we focus on a robust version of LRMR with the loss $\psi$ being:
\begin{equation}
\psi(z; \sigma) = 1 - \exp(-z^2/2\sigma),
\label{eqn_negative_gaussian_loss}
\end{equation}
where $\sigma$ is a tunable parameter.
Loss \eqref{eqn_negative_gaussian_loss} is less sensitive to the discrepancy 
$\XB_{ij}-\YB_{ij}$ compared to the common least square loss $\psi(z) = z^2/2$, 
and hence is robust to adversarial outliers \citep{qu2017non}.

In each trial, we first generate an underlying matrix $\MB$ of size $200\times 
200$ and rank $\gamma = 15$.
The singular values of $\MB$ are set as ${2^{[\gamma]}}/{2^{\gamma}}\times 50$ 
and hence $\|\MB\|_*\leq C = 100$, where $[\gamma] = \{1, \ldots, \gamma\}.$ 
We then inject adversarial noise into $\MB$ by (1) uniformly sampling $5\%$ of 
the entries in $\MB$ and (2) adding random noise uniformly sampled from 
$[-\rho, \rho]$ to each selected entry, where the noise level $\rho$ equals 
$10$.
Denote $\hat{\MB}$ as the matrix after noise injection.
We uniformly sample $10\%$ of the entries in $\hat{\MB}$ to obtain the 
observations, \emph{i.e.}, $\YB_{ij}$.
Hence $|\Omega|$, the number of observation is $M\times N\times10\% = 4,000$.

In terms of algorithmic parameter setting, we set the mini-batch size $m$ to 
${|\Omega|}/{20}$.
The number of epoch $T$ is set to $50$ for all cases, and the step size 
parameter $\eta_t$ is set to $1/(T*|\Omega|/m) = 1/1000$ in all cases for all 
methods.

We present the comparison of listed methods in Figure \ref{fig_rlrmr}, where we 
observe that \sfw has the best performance in terms of the Frank-Wolfe gap 
(\cref{fig:sfw-2-1}), 
gradient estimation accuracy (\cref{fig:sfw-2-2}), and the Root Mean Square 
Error (RMSE) between 
the prediction matrix and the underlying true matrix (\cref{fig:sfw-2-3}).

\begin{figure*}[t]
	\centering 
	\begin{subfigure}[t!]{0.32\linewidth}
		\includegraphics[width=\columnwidth]{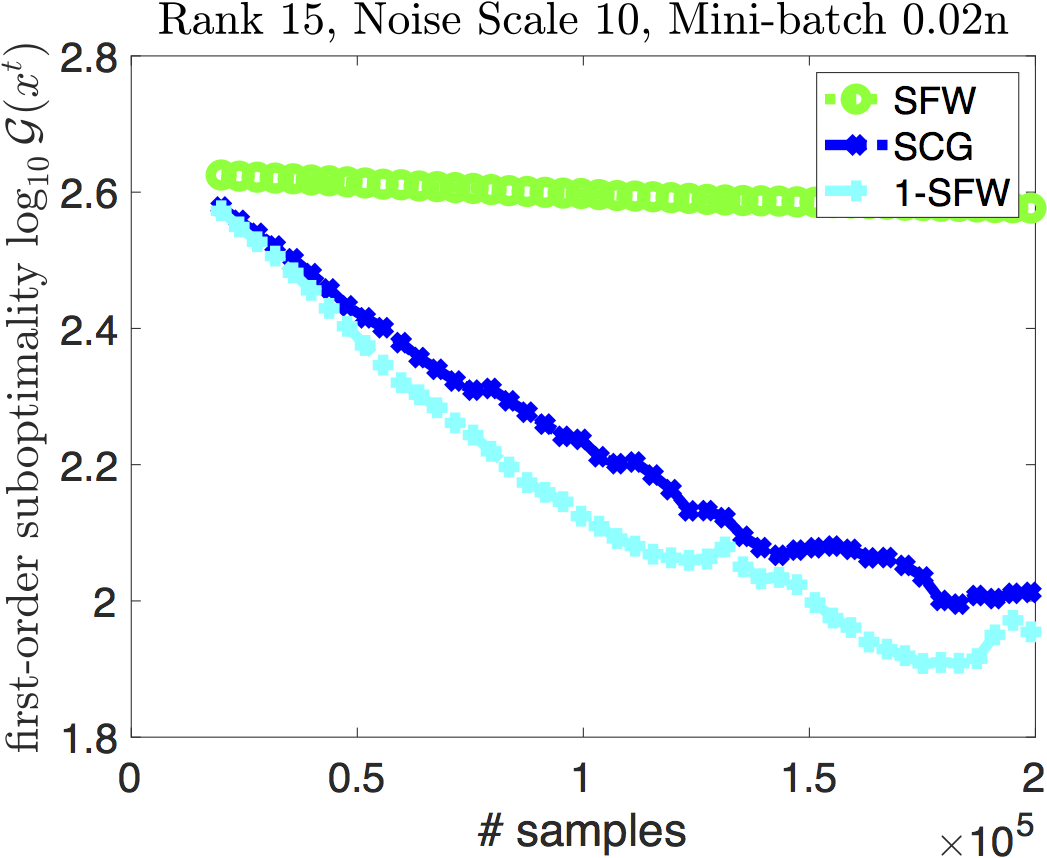}
		\caption{Frank-Wolfe gap}
		\label{fig:sfw-2-1}
	\end{subfigure}
	\begin{subfigure}[t!]{0.32\linewidth}
		\includegraphics[width=\columnwidth]{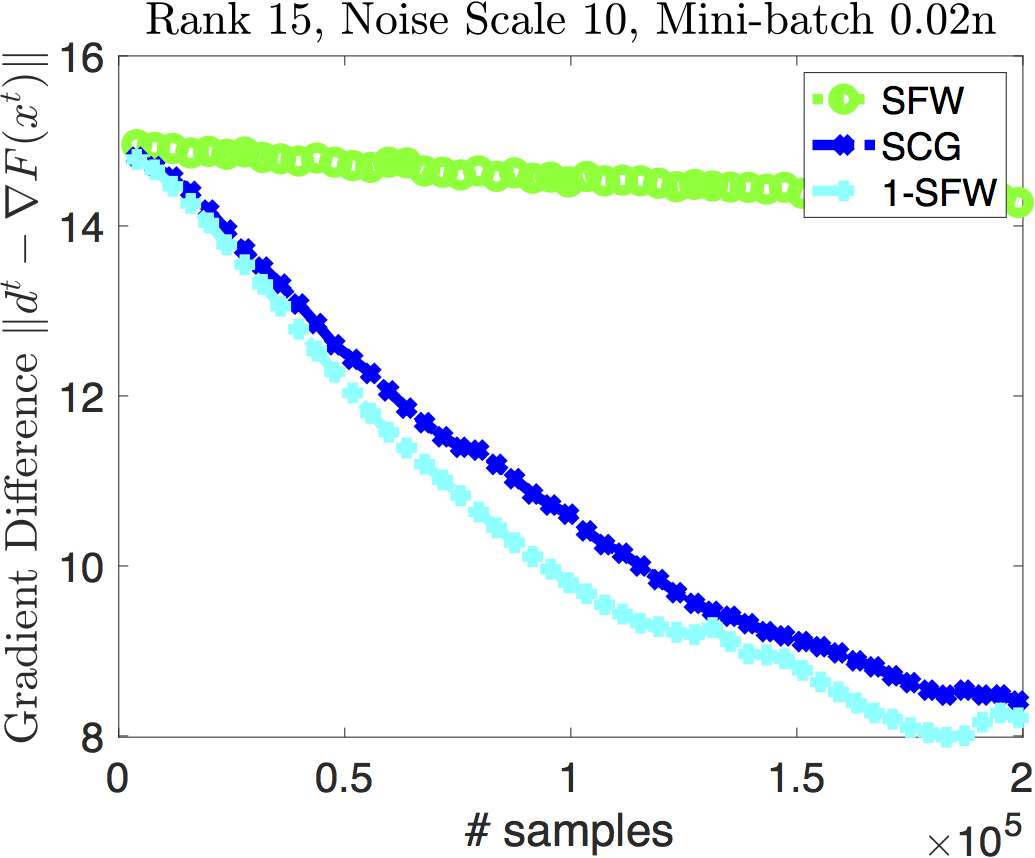}
		\caption{Gradient estimation acc.}
		\label{fig:sfw-2-2}
	\end{subfigure}
	\begin{subfigure}[t!]{0.32\linewidth}
		\includegraphics[width=\columnwidth]{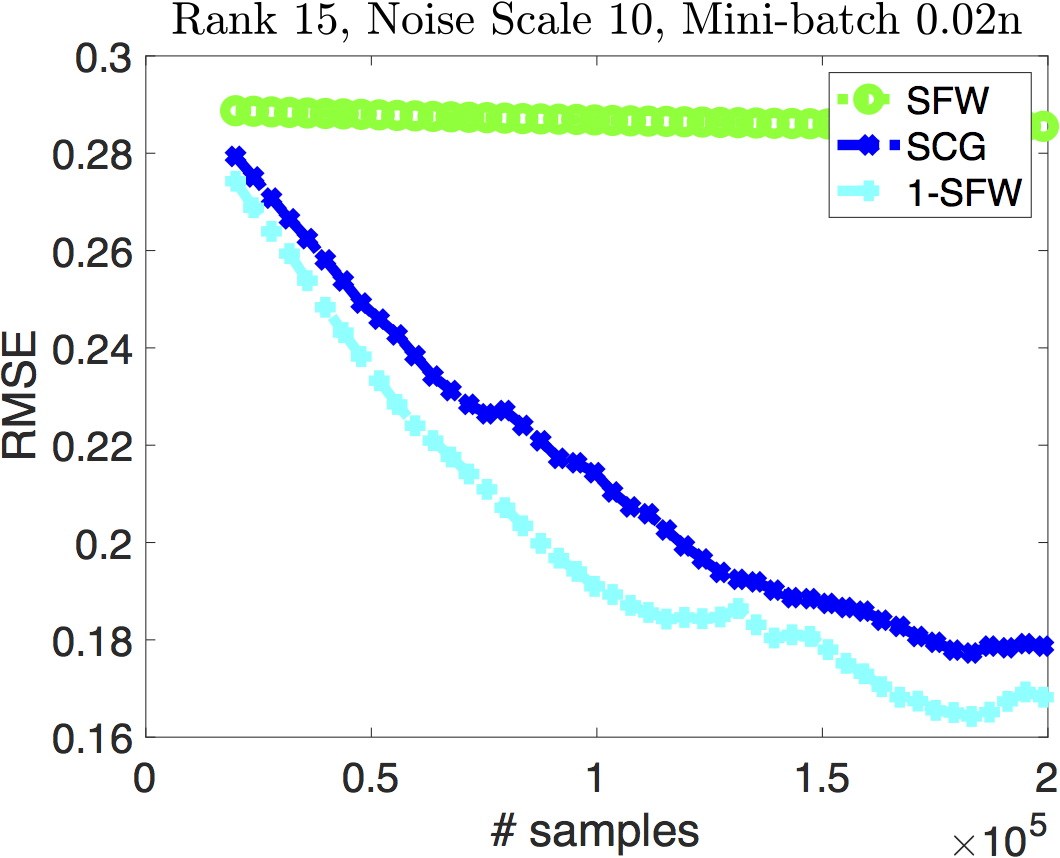}
		\caption{Root Mean Square Error}
		\label{fig:sfw-2-3}
	\end{subfigure} 
	\caption{Matrix Recovery. (a) compares the Frank-Wolfe gap, (b) compares 
				the accuracy of gradient estimation, (c) compares the Root Mean 
		Square Error (RMSE) between 
				the prediction matrix and the underlying true matrix.}
		\label{fig_rlrmr}
\end{figure*}

%
%

\subsection{Discrete Monotone Submodular Maximization with Matroid Constraint}
In this subsection, we consider the discrete monotone submodular maximization 
subject to a matroid constraint via the maximizing the corresponding 
multilinear extension.
Let $V$ be a finite set of $d$ elements and $\IM$ be a collection of its 
subsets.
It is proved that to maximize a discrete monotone submodular function $f: 
2^V\rightarrow\RBB_+$ subject to the matroid constraint $\MM\defi\{V, \IM\}$ is 
equivalent to maximize its multilinear extension, defined as 
\begin{equation} \label{eq:multilinear}
F(x) = \sum_{S \subset [d]} f(S) \prod_{j \in S} [x]_j \prod_{\ell \notin 
	S} ( 1 - [x]_{\ell}),
\end{equation}
subject to the constraint $x \in\CM$, where $\CM$ is the base polytope of 
$\MM$.
Further, it is known that $F$ is monotone DR-submodular.

We now focus on a concrete recommendation problem which can be formulated as 
discrete monotone submodular maximization.
We use $r(u,j)$ to denote user $u$'s rating for item $j \in [d]$ and set 
$r(u,j) = 0$ if item $j$ is not rated by user $u$.
Our goal is to recommend a set of $k=10$ items to all users such that they have 
the highest total rating.
Two types of utility functions can be defined for such task: facility location
\begin{equation}
f(S) = \sum_{u} \max_{j \in S} \ r(u, j), 
\end{equation}
or concave over modular
\begin{equation}
f(S) = \sum_{u} (\sum_{j \in S} \ r(u, j))^{1/2}.
\end{equation}
Here the matroid is $\{V, \IM\defi\{S \subseteq V | |S|=k\} \}$.
Two datasets are used in this experiment, Jester 
1\footnote{http://eigentaste.berkeley.edu/dataset/} and movielens 
1M\footnote{https://grouplens.org/datasets/movielens/} with the results 
presented in Figure \ref{fig_submodular_jester} and Figure 
\ref{fig_submodular_movielens} respectively.
We observe that \sfw always achieves the highest utility after sufficient 
function evaluations.

\begin{figure*}[t]
	\centering 
	\begin{subfigure}[t!]{0.45\linewidth}
		\includegraphics[width=\textwidth]{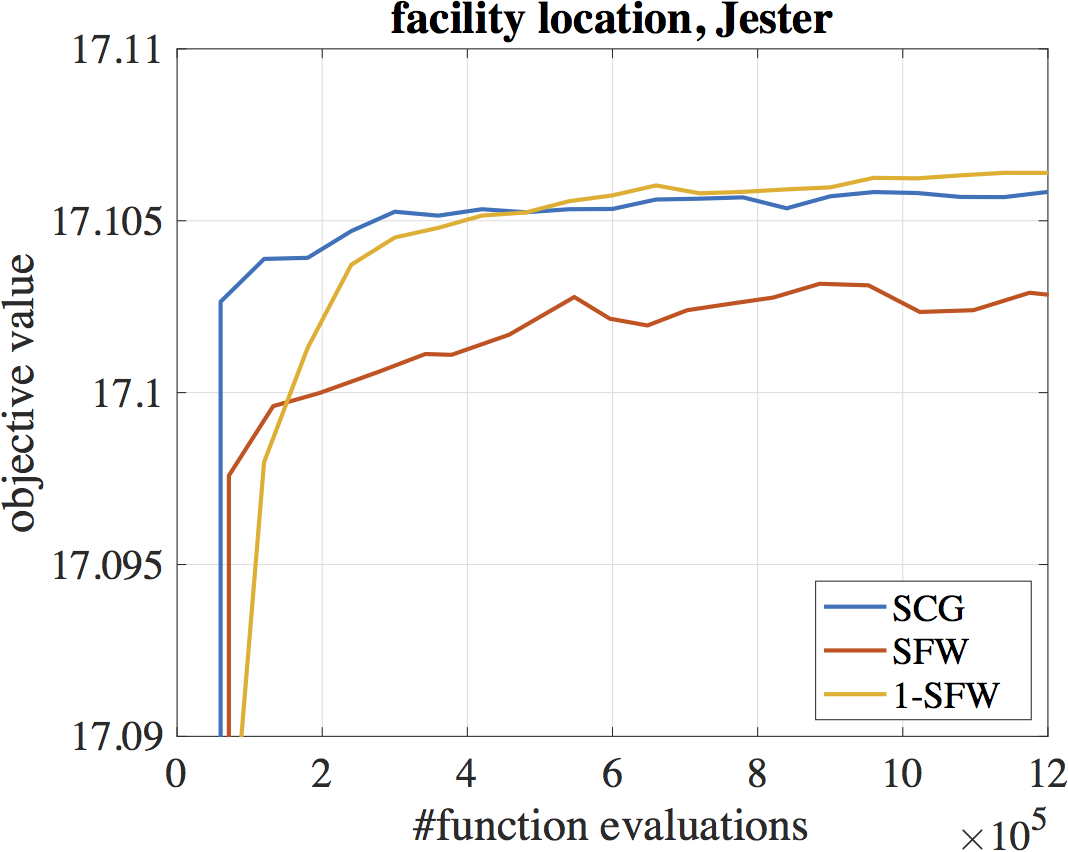}
		\caption{Facility location utility}
		\label{fig:sfw-3-1}
	\end{subfigure}
	\begin{subfigure}[t!]{0.45\linewidth}
		\includegraphics[width=\textwidth]{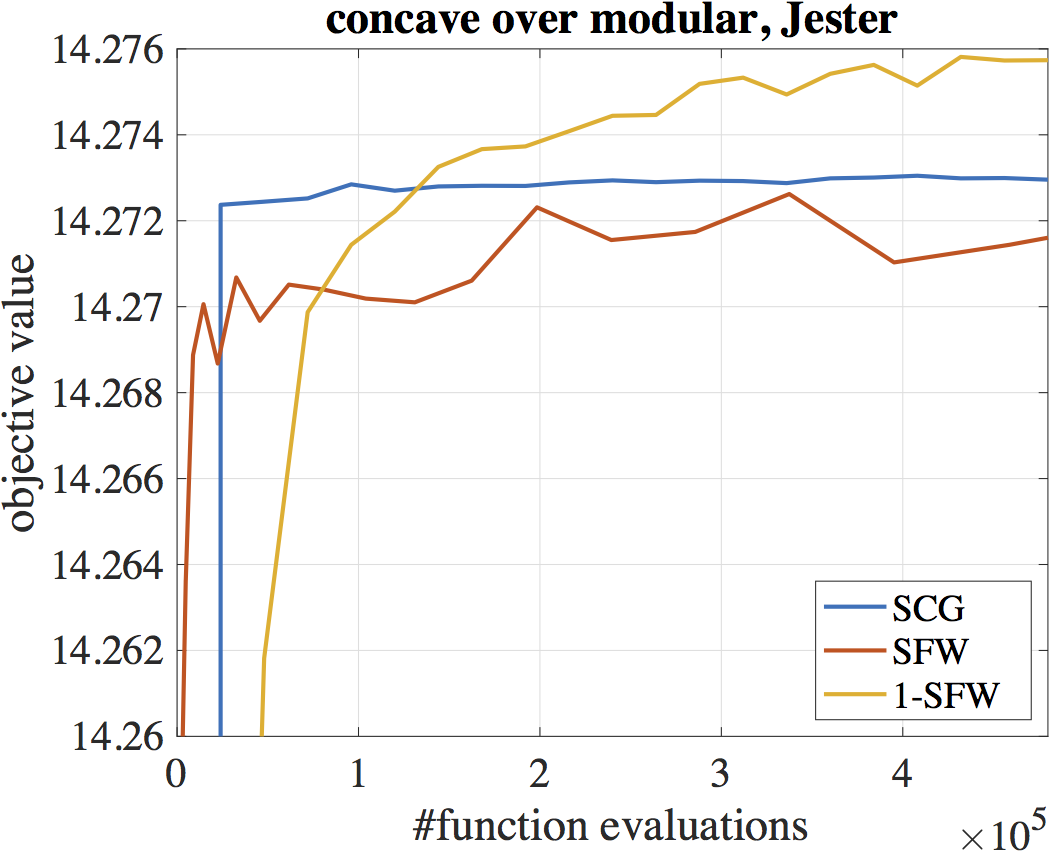}
		\caption{Concave over modular utility}
		\label{fig:sfw-3-2}
	\end{subfigure}
	\caption{Submodular Maximization on Jester dataset. (a) uses the facility 
		location utility and (b) uses the concave over modular utility.}
	\label{fig_submodular_jester}
\end{figure*}

%
%

\begin{figure*}[t]
	\centering 
	\begin{subfigure}[t!]{0.45\linewidth}
		\includegraphics[width=\textwidth]{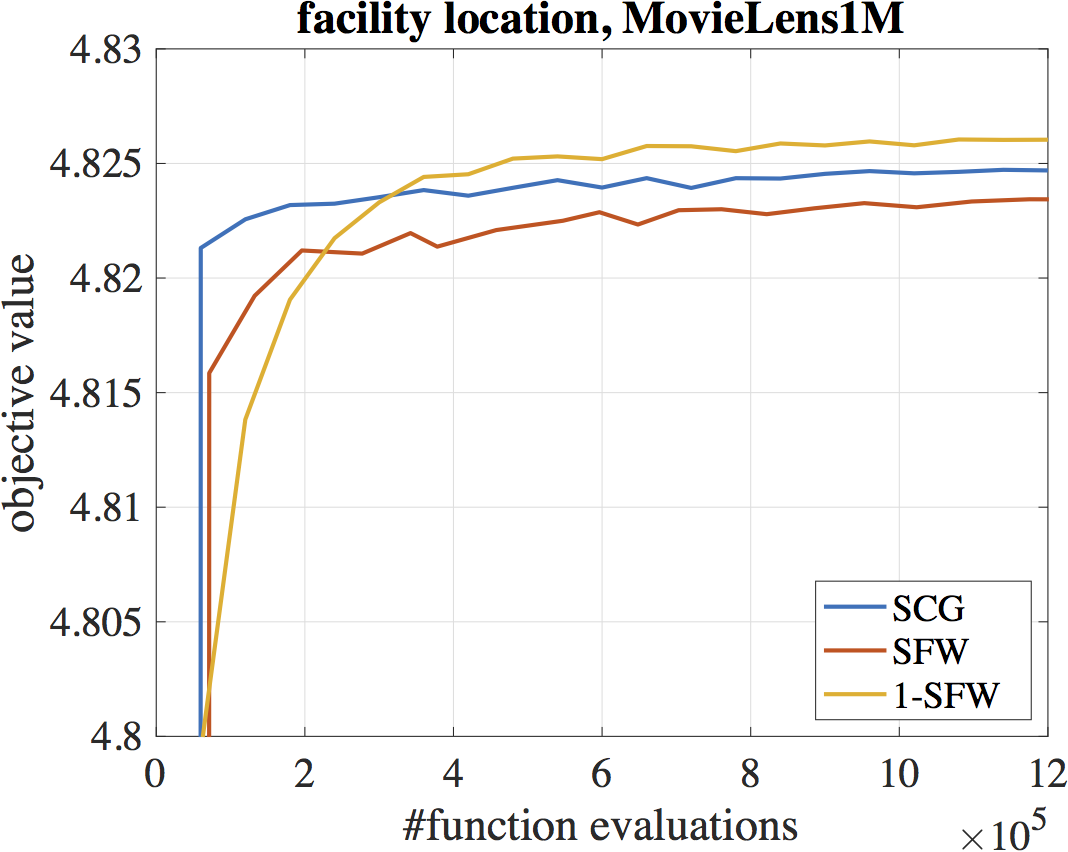}
		\caption{Facility location utility}
		\label{fig:sfw-4-1}
	\end{subfigure}
	\begin{subfigure}[t!]{0.45\linewidth}
		\includegraphics[width=\textwidth]{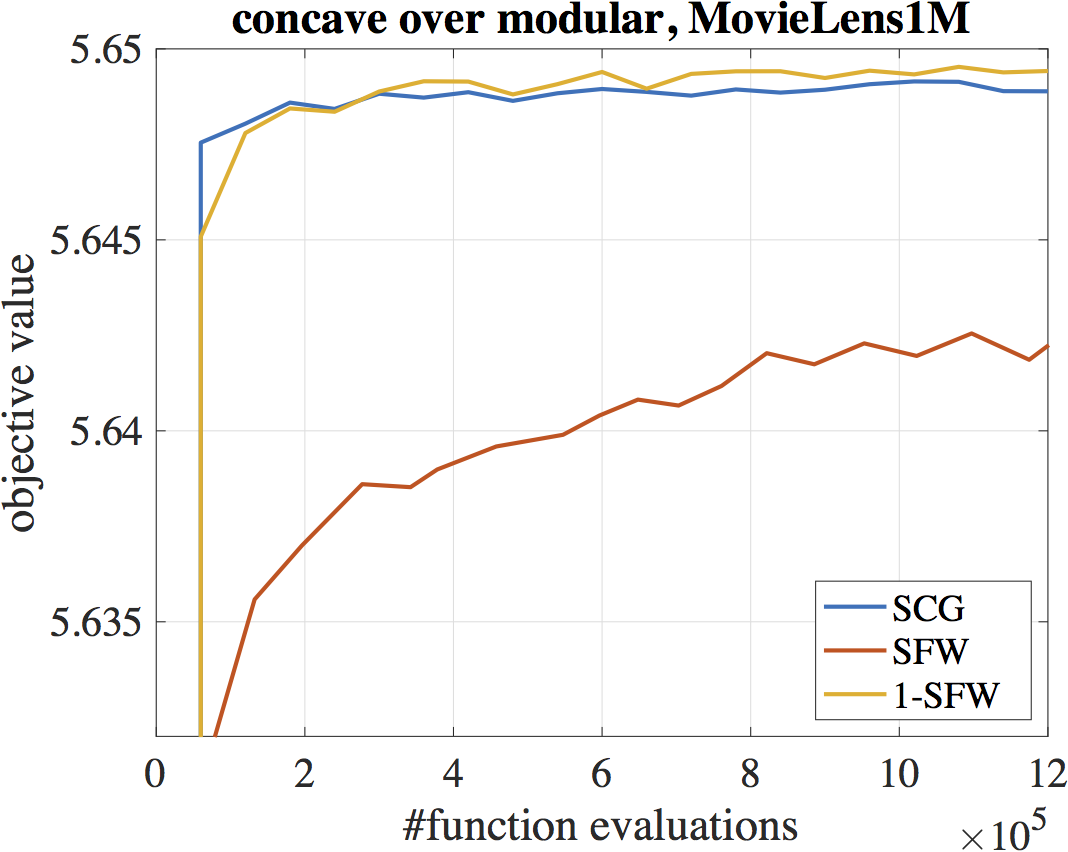}
		\caption{Concave over modular 
			utility}
		\label{fig:sfw-4-2}
	\end{subfigure}
	\caption{Submodular Maximization on Movielens dataset. (a) uses the 
		facility location utility and (b) uses the concave over modular 
		utility.}
	\label{fig_submodular_movielens}
\end{figure*}

\section{Conclusion}
In this chapter, we studied the problem of solving constrained  stochastic 
optimization programs  using projection-free methods. We proposed the first 
stochastic variant of the Frank-Wolfe method, called \sfw, that requires only 
one stochastic sample per iteration while achieving the best known complexity 
bounds for (non-)convex minimization and monotone DR-submodular maximization. 
In particular, we proved that \sfw achieves the best known 
oracle complexity of 
$\mathcal{O}(1/\epsilon^2)$ for reaching an  
$\epsilon$-suboptimal solution in the stochastic convex setting, and a 
$(1-1/e)\text{OPT}-\epsilon$ approximate solution for a stochastic monotone 
DR-submodular 
maximization problem. Moreover, in a non-convex setting, \sfw finds an 
$\epsilon$-first-order stationary point after at most 
$\mathcal{O}(1/\epsilon^3)$ iterations, achieving the best known overall 
complexity.


\section{Proofs}\label{sec:proof-1sfw}
\subsection{Proof of Lemma 2}

\begin{proof}
Let $A_t = \|\nabla F(x_t) - d_t\|^2$. By definition, we have
\begin{equation*}
A_t = \| \nabla F(x_{t-1}) - d_{t-1}  + \nabla F(x_{t}) - 
\nabla F(x_{t-1}) - (d_t - d_{t-1})\|^2.
\end{equation*}

Note that
\begin{equation*}
d_t - d_{t-1} = - \rho_t d_{t-1} + \rho_t \nabla \tF(x_t,z_t) + 
(1-\rho_t) \tilde{\Delta}_t,
\end{equation*}
and define $\Delta_t = \nabla F(x_{t}) - 
\nabla F(x_{t-1})$, we have
\begin{equation*}
\begin{split}
A_t &= \|\nabla F(x_{t-1}) - d_{t-1} + \Delta_t - (1-\rho_t) 
\tilde{\Delta}_t - \rho_t \nabla \tF(x_t,z_t) + \rho_t d_{t-1} \|^2  \\
& = \|\nabla F(x_{t-1}) - d_{t-1} + (1-\rho_t)(\Delta_t -  
\tilde{\Delta}_t) + \rho_t(\nabla F(x_t)-\nabla \tF(x_t,z_t) + 
\rho_t(d_{t-1}-\nabla F(x_{t-1}))) \|^2 \\
&= \|(1-\rho_t)(\nabla F(x_{t-1}) - d_{t-1}) + (1-\rho_t)(\Delta_t -  
\tilde{\Delta}_t) + \rho_t(\nabla F(x_t)-\nabla \tF(x_t,z_t))\|^2.
\end{split}
\end{equation*}

Since $\tilde{\Delta}_t$ is an unbiased estimator of $\Delta_t$, $\expect[A_t]$ 
can be decomposed as
\begin{equation}
\label{eq:decomp2}
\begin{split}
\expect[A_t] &= \expect\{(1-\rho_t)^2 \|\nabla F(x_{t-1}) - d_{t-1} \|^2 + 
(1-\rho_t)^2 \| \Delta_t -  
\tilde{\Delta}_t \|^2  \\
&\quad + \rho_t^2 \|\nabla F(x_t)-\nabla \tF(x_t,z_t) 
\|^2 \\
&\quad +2\rho_t(1-\rho_t)\langle\nabla F(x_{t-1}) - d_{t-1} , \nabla 
F(x_t)-\nabla \tF(x_t,z_t) \rangle \\
&\quad + 2\rho_t(1-\rho_t)\langle \Delta_t -  
\tilde{\Delta}_t , \nabla F(x_t)-\nabla \tF(x_t,z_t) \rangle\}.
\end{split}
\end{equation}

Then we turn to upper bound the items above. First, by \cref{lem:scg++8.1}, we 
have
\begin{equation}
\label{eq:aux1}
\begin{split}
\expect[\|\tilde{\Delta}_t - \Delta_t \|^2] &= \expect[\| 
\tnabla_t^2(x_t-x_{t-1}) - (\nabla F(x_t) - \nabla F(x_{t-1}))]\|^2] \\
&\leq \expect[\|\tnabla_t^2(x_t-x_{t-1}) \|^2] \\
&= \expect[\|\tnabla_t^2(\eta_{t-1}(v_{t-1}-x_{t-1})) \|^2] \\
&\leq \eta_{t-1}^2D^2 \expect[\| \tnabla_t^2\|^2] \\
&\leq \eta_{t-1}^2D^2\bar{L}^2.
\end{split}
\end{equation}

By Jensen's inequality, we have
\begin{equation*}\label{eq:aux2}
\expect[\|\tilde{\Delta}_t - \Delta_t \|] \leq 
\sqrt{\expect[\|\tilde{\Delta}_t - \Delta_t \|^2]} \leq \eta_{t-1} D 
\bar{L},
\end{equation*}
and
\begin{equation*}\label{eq:aux5}
\expect[\|\nabla F(x_t) -d_t \|] = \sqrt{\expect[\|\nabla F(x_t) -d_t 
	\|^2]} = \sqrt{\expect[A_t]}.
\end{equation*}

Note that $z_t$ is sampled according to $p(z; x_t(a))$, where $x_{t}(a) 
= a x_t + (1-a)x_{t-1}$. Thus $\nabla \tF(x_t,z_t)$ is NOT an unbiased 
estimator of $\nabla F(x_t)$ when $a \neq 1$, which occurs with probability 
1. 
However, we will show that $\nabla \tF(x_t, z_t)$ is still a good estimator.
Let $\mathcal{F}_{t-1}$ be the $\sigma$-field generated by all the randomness 
before round $t$, then by Law of Total Expectation, we have
\begin{equation}\label{eq:aux3}
\begin{split}
&\expect[2\rho_t(1-\rho_t)\langle\nabla F(x_{t-1}) - d_{t-1} , \nabla 
F(x_t)-\nabla \tF(x_t,z_t) \rangle] \\
=& \expect[\expect[2\rho_t(1-\rho_t)\langle\nabla F(x_{t-1}) - d_{t-1} , 
\nabla 
F(x_t)-\nabla \tF(x_t,z_t) \rangle|\mathcal{F}_{t-1},x_t(a)]] \\
=& \expect[2\rho_t(1-\rho_t)\langle\nabla F(x_{t-1}) - d_{t-1} , 
\expect[\nabla 
F(x_t)-\nabla \tF(x_t,z_t) |\mathcal{F}_{t-1},x_t(a)]\rangle],
\end{split}
\end{equation}
where 
\begin{equation*}
\expect[\nabla 
F(x_t)-\nabla \tF(x_t,z_t) |\mathcal{F}_{t-1}]\rangle] = \nabla F(x_t) 
- \nabla F(x_t(a)) + \nabla F(x_t(a)) - \expect[\nabla 
\tF(x_t,z_t)|\mathcal{F}_{t-1},x_t(a)].
\end{equation*}

By \cref{lem:scg++8.1}, $F$ is $\bar{L}$-smooth, thus 
\begin{equation*}
\|\nabla F(x_t) - \nabla F(x_t(a)) \| \leq \bar{L} \|x_t - x_t(a) \| = 
\bar{L}(1-a)\| \eta_{t-1}(v_{t-1}-x_{t-1})\| \leq \eta_{t-1}D\bar{L}.
\end{equation*}

We also have
\begin{equation*}
\begin{split}
&\| \nabla F(x_t(a)) - \expect[\nabla 
\tF(x_t,z_t)|\mathcal{F}_{t-1},x_t(a)]\| \\
={}& \| \int [\nabla \tF 
(x_t(a); z) - 
\nabla \tF (x_t; z)] p(z; x_t(a)) \mathrm{d}z \|\\
\leq{}& \int \| \nabla \tF (x_t(a); z) - 
\nabla \tF (x_t; z)\| p(z; x_t(a)) \mathrm{d}z \\
\leq{}& \int L \|x_t(a)-x_t \| p(z; x_t(a)) \mathrm{d}z \\
\leq{}& \eta_{t-1}DL,
\end{split}
\end{equation*}
where the second inequality holds because of \cref{assum:second_order}. Combine 
the analysis above with \cref{eq:aux3},
we have
\begin{equation}\label{eq:aux7}
\begin{split}
&\expect[2\rho_t(1-\rho_t)\langle\nabla F(x_{t-1}) - d_{t-1} , \nabla 
F(x_t)-\nabla \tF(x_t,z_t) \rangle] \\
\leq{} & \expect[2\rho_t(1-\rho_t)\|\nabla F(x_{t-1}) - d_{t-1}\| \cdot 
\|\expect[\nabla 
F(x_t)-\nabla \tF(x_t,z_t) |\mathcal{F}_{t-1}]\|] \\
\leq{} & 2\rho_t(1-\rho_t) \expect[\|\nabla F(x_{t-1}) - d_{t-1}\|] \cdot 
(\eta_{t-1}D\bar{L} + \eta_{t-1}DL) \\
\leq{} & 2\eta_{t-1}\rho_t(1-\rho_t) \sqrt{\expect[A_{t-1}]}D(\bar{L}+L).
\end{split}
\end{equation}

Finally, by \cref{assum:gradient_norm}, we have $\| \nabla F(x_t)-\nabla 
\tF(x_t,z_t) \| \leq 2G$. Thus 
\begin{equation}\label{eq:aux6}
\rho_t^2 \|\nabla F(x_t)-\nabla \tF(x_t,z_t) 
\|^2 \leq 4\rho_t^2G^2,
\end{equation}
and
\begin{equation}\label{eq:aux4}
\begin{split}
&\expect [2\rho_t(1-\rho_t)\langle \Delta_t -  
\tilde{\Delta}_t , \nabla F(x_t)-\nabla \tF(x_t,z_t) \rangle] \\
\le{}& 
\expect[2\rho_t(1-\rho_t)\| \Delta_t -  
\tilde{\Delta}_t\| \cdot \|\nabla F(x_t)-\nabla \tF(x_t,z_t) \|] \\
\leq{}& 4\eta_{t-1}\rho_t(1-\rho_t)GD\bar{L}.
\end{split}
\end{equation}

Combine \cref{eq:decomp2,eq:aux1,eq:aux6,eq:aux7,eq:aux4}, we have
\begin{equation*}
\begin{split}
\expect[A_t] &\leq (1-\rho_t)^2 \expect[A_{t-1}] + (1-\rho_t)^2 
\eta_{t-1}^2D^2\bar{L}^2 + \rho_t^2 4G^2 \\
&\quad + 2\eta_{t-1}\rho_t(1-\rho_t) \sqrt{\expect[A_{t-1}]}D(\bar{L}+L) + 
4\eta_{t-1}\rho_t(1-\rho_t)GD\bar{L}. 
\end{split}
\end{equation*}

For the simplicity of analysis, we replace $t$ by $t+1$, and have
\begin{equation}\label{eq:recursive}
\begin{split}
\expect[A_{t+1}] 
\leq& (1-\rho_{t+1})^2 \expect[A_{t}] + (1-\rho_{t+1})^2 
\eta_{t}^2D^2\bar{L}^2 + \rho_{t+1}^2 4G^2 \\
&\quad + 2\eta_{t}\rho_{t+1}(1-\rho_{t+1}) \sqrt{\expect[A_{t}]}D(\bar{L}+L) + 
4\eta_{t}\rho_{t+1}(1-\rho_{t+1})GD\bar{L} \\
\leq& (1-\frac{1}{t^\alpha})^2 \expect[A_{t}] + 
\frac{D^2\bar{L}^2+4G^2+4GD\bar{L}}{t^{2\alpha}} + 
\frac{2D(\bar{L}+L)}{t^{2\alpha}}\sqrt{\expect[A_{t}]}.
\end{split}
\end{equation}

We claim that $\expect[A_{t}] \leq C t^{-\alpha}$, and prove it by induction. 
Before the proof, we first analyze one item in the definition of $C: 
\frac{2(2G+D\bar{L})^2}{2-2^{-\alpha}-\alpha}$. Define $h(\alpha) = 
2-2^{-\alpha}-\alpha$. Since $h'(\alpha) = 2^{-\alpha}\ln(2)-1\le 0$ for 
$\alpha \in (0,1]$, so $1=h(0) \ge h(\alpha) \ge h(1) = 1/2 >0, \forall\, 
\alpha 
\in (0,1]$. As a result, $2 \le \frac{2}{2-2^{-\alpha}-\alpha}\le 4$.

When $t=1$, we have
\begin{equation*}
\expect[A_{1}] = \expect[\|\nabla F(x_1)-\nabla \tF (x_1;z_1) \|^2] \leq 
(2G)^2 \leq \frac{2(2G+D\bar{L})^2}{2-2^{-\alpha}-\alpha} / 1 \leq C\cdot 
1^{-\alpha}.
\end{equation*}

When $t=2$, since $\rho_2 = 1$, we have
\begin{equation*}
\begin{split}
\expect[A_{2}] = \expect[\|\nabla 
\tF(x_2,z_2) - \nabla F(x_2)\|^2] \leq (2G)^2 \leq 
\frac{2(2G+D\bar{L})^2}{2-2^{-\alpha}-\alpha} / 2 \leq C \cdot 2^{-\alpha}.
\end{split}
\end{equation*}

Now assume for $t \ge 2$, we have $\expect[A_{t}] \leq C t^{-\alpha}$, by 
\cref{eq:recursive} and the definition of $C$, we have
\begin{equation}\label{eq:induction}
\begin{split}
\expect[A_{t+1}] &\leq (1-\frac{1}{t^\alpha})^2 \cdot Ct^{-\alpha} + 
\frac{(2G+D\bar{L})^2}{t^{2\alpha}} + 
\frac{2D(\bar{L}+L)}{t^{(5/2)\alpha}}\sqrt{C}\\
&\leq Ct^{-\alpha} - 2Ct^{-2\alpha} + Ct^{-3\alpha} + 
\frac{(2-2^{-\alpha}-\alpha)C}{2t^{2\alpha}} + \frac{C^{3/4}}{t^{(5/2)\alpha}} 
\\ 
&\leq \frac{C}{t^\alpha} + \frac{-2C+Ct^{-\alpha}+ 
(2-2^{-\alpha}-\alpha)C/2+t^{-\alpha/2}C 
/ C^{1/4}}{t^{2\alpha}} \\
&\leq \frac{C}{t^\alpha} + \frac{C[-2+2^{-\alpha} +(2-2^{-\alpha}-\alpha)/2 
+(2-2^{-\alpha}-\alpha)/2]}{t^{2\alpha}} \\
&\leq \frac{C}{t^\alpha} - \frac{\alpha C}{t^{2\alpha}}.
\end{split}
\end{equation}

Define $g(t) = t^{-\alpha}$, then $g(t)$ is a convex function for $\alpha \in 
(0,1]$. Thus we have $g(t+1) - g(t) \ge g'(t)$, \emph{i.e.}, $(t+1)^{-\alpha} 
- t^{-\alpha} \geq -\alpha t^{-(\alpha+1)}$. So we have
\begin{equation*}
\frac{C}{t^{\alpha}}- \frac{\alpha C}{t^{2\alpha}} \leq C(t^{-\alpha} - \alpha 
t^{-(1+\alpha)}) \leq C (t+1)^{-\alpha}.
\end{equation*}

Combine with \cref{eq:induction}, we have $\expect[A_{t+1}] \leq C 
(t+1)^{-\alpha}$. Thus by induction, we have $\expect[A_{t}] \leq C 
t^{-\alpha}, \forall\, t \ge 1$.
\end{proof}

\subsection{Proof of Lemma 3}
The only difference with the proof of \cref{lem:graident_error} is the bound 
for $\expect \| \tilde{\Delta}_t - \Delta_t\|$. Specifically, we have
\begin{equation*}
\begin{split}
\expect[\|\tilde{\Delta}_t - \Delta_t \|^2] &= \expect[\| 
\tilde{\Delta}_t - \tnabla_t^2(x_t-x_{t-1}) + \tnabla_t^2(x_t-x_{t-1}) 
- (\nabla F(x_t) - \nabla F(x_{t-1}))]\|^2] \\
&= \expect[\|\tilde{\Delta}_t - \tnabla_t^2(x_t-x_{t-1}) \|^2] + 
\expect[\|\tnabla_t^2(x_t-x_{t-1}) 
- (\nabla F(x_t) - \nabla F(x_{t-1})) \|^2]\\
&\leq [D^2L_2\delta_t(1+\tF(x_t(a),z_t))]^2 + \eta_{t-1}^2D^2\bar{L}^2\\
&\leq (1+B)^2L_2^2D^4\delta_t^2 + \eta_{t-1}^2D^2\bar{L}^2 \\
&\leq 4\eta_{t-1}^2D^2\bar{L}^2.
\end{split}
\end{equation*}

Then by the analysis same to the proof of \cref{lem:graident_error}, we have
\begin{equation*}
\expect[A_{t+1}]\leq  (1-\frac{1}{t^\alpha})^2 \expect[A_{t}] + 
\frac{4(D^2\bar{L}^2+G^2+GD\bar{L})}{t^{2\alpha}} + 
\frac{4D(\bar{L}+L)}{t^{2\alpha}}\sqrt{\expect[A_{t}]},
\end{equation*}
and thus $\expect[A_{t+1}] \leq C 
(t+1)^{-\alpha}$, where $C\! =\! \max \! \left\{ 
\!\frac{8(D^2\bar{L}^2+G^2+GD\bar{L})}{2\!-\!2^{-\alpha}\!-\!\alpha},
\left[\frac{2}{2\!-\!2^{-\alpha}\!-\!\alpha} \right]^4\!\!, 
[4D(\bar{L}\!+\!L)]^4 \! \right\}\!.$

\subsection{Proof of Theorem 1}
First, since $x_{t+1} = (1-\eta_t)x_t + \eta_t v_t$ is a convex 
combination of $x_t, v_t$, and $x_1 \in \constraint, v_t \in 
\constraint, \forall\ t$, we can prove $x_t \in \constraint, \forall\ t $ by 
induction. So $x_{T+1} \in \constraint$.

Then we present an auxiliary lemma.
\begin{lemma}[Proof of Theorem 1 in 
	\citep{yurtsever2019conditional}]\label{lem:aryan1}
Under the condition of \cref{thm:convex}, in \cref{alg:one_sample}, we have 
\begin{equation*}
F(x_{t+1}) - F(x^*) \leq (1-\eta_t)(F(x_t) - F(x^*)) + \eta_tD\|\nabla 
F(x_t) - d_t\| + \frac{\bar{L}D^2\eta_t^2}{2}.
\end{equation*}
\end{lemma}
By Jensen's inequality and \cref{lem:graident_error} with $\alpha=1$, we have
\begin{equation*}
\expect[\| \nabla F(x_t) - d_t\| ] \leq \sqrt{\expect[\| \nabla F(x_t) - 
d_t\|^2 ]} \leq \frac{\sqrt{C}}{\sqrt{t}},
\end{equation*}
where $C=\max \{4(2G+D\bar{L})^2, 256, [2D(\bar{L}+L)]^4 \}$. Then by 
\cref{lem:aryan1}, we have
\begin{equation}\label{eq:convex1}
\begin{split}
&\expect[F(x_{T+1}) - F(x^*)] \\
\leq{}& (1-\eta_T)\expect[F(x_T) - F(x^*)] + 
\eta_TD\expect[\|\nabla 
F(x_T) - d_T\|] + \frac{\bar{L}D^2\eta_T^2}{2} \\
={}& \prod_{i=1}^T (1-\eta_i) \expect[F(x_1) - F(x^*)] + D\sum_{k=1}^T 
\eta_k\expect[\|\nabla F(x_k) - d_k\|] \prod_{i=k+1}^{T} (1-\eta_i) \\
&\quad + \frac{\bar{L}D^2}{2} \sum_{k=1}^T \eta_k^2\prod_{i=k+1}^{T} (1-\eta_i) 
\\
\leq{}& 0 + D\sum_{k=1}^T k^{-1}\frac{\sqrt{C}}{\sqrt{k}} \prod_{i=k+1}^{T} 
\frac{i-1}{i} + \frac{\bar{L}D^2}{2} \sum_{k=1}^T k^{-2} 
\prod_{i=k+1}^{T}\frac{i-1}{i} \\
={}& \frac{\sqrt{C}D}{T}\sum_{k=1}^T \frac{1}{\sqrt{k}} + 
\frac{\bar{L}D^2}{2T}\sum_{k=1}^T k^{-1}.
\end{split}
\end{equation}

Since
\begin{equation*}
\sum_{k=1}^T \frac{1}{\sqrt{k}} \leq \int_{0}^T x^{-1/2}\mathrm{d}x = 2\sqrt{T},
\end{equation*}
and
\begin{equation*}
\sum_{k=1}^T k^{-1} \leq 1 + \int_1^T x^{-1}\mathrm{d}x = 1+\ln T,
\end{equation*}
by \cref{eq:convex1}, we have
\begin{equation*}
\expect[F(x_{T+1}) - F(x^*)] \leq \frac{2\sqrt{C}D}{\sqrt{T}} + 
\frac{\bar{L}D^2}{2T}(1+\ln T).
\end{equation*}

\subsection{Proof of Theorem 2}
First, since $x_{t+1} = (1-\eta_t)x_t + \eta_t v_t$ is a convex 
combination of $x_t, v_t$, and $x_1 \in \constraint, v_t \in \constraint, 
\forall\ t$, we can prove $x_t \in \constraint, \forall\ t $ by induction. So 
$x_{o} \in \constraint$.

Note that if we define $v_t^\prime = \argmin_{v \in \constraint}\langle 
v, 
\nabla F(x_t)\rangle$, then $\mathcal{G}(x_t)=\langle 
v^\prime_t-x_t,-\nabla F(x_t)\rangle = -\langle v^\prime_t-x_t,\nabla 
F(x_t)\rangle$. So we have

\begin{equation*}
\begin{split}
F(x_{t+1}) 
\stackrel{(a)}{\leq}{}& F(x_t) + \langle \nabla f(x_t),x_{t+1}-x_t 
\rangle + 
\frac{\bar{L}}{2}\|x_{t+1}-x_t\|^2 \\
={}&F(x_t) + \langle \nabla F(x_t),\eta_t(v_t-x_t) \rangle + 
\frac{\bar{L}}{2}\|\eta_t(v_t-x_t)\|^2 \\
\stackrel{(b)}{\leq}{}& F(x_t) + \eta_t \langle \nabla F(x_t),v_t-x_t 
\rangle+\frac{\bar{L}\eta_t^2D^2}{2} \\
={}& F(x_t) + \eta_t \langle  d_t,v_t-x_t \rangle+ \eta_t \langle 
\nabla 
F(x_t)-d_t,v_t-x_t \rangle  + \frac{\bar{L}\eta_t^2D^2}{2} \\
\stackrel{(c)}{\leq}{}& F(x_t) + \eta_t \langle  d_t,v^\prime_t-x_t 
\rangle+ \eta_t \langle \nabla F(x_t)-d_t,v_t-x_t \rangle + 
\frac{\bar{L}\eta_t^2D^2}{2} \\
={}& F(x_t) + \eta_t \langle \nabla F(x_t),v^\prime_t-x_t \rangle  + 
\eta_t 
\langle  d_t-\nabla F(x_t),v^\prime_t-x_t \rangle \\
&\quad + \eta_t \langle 
\nabla F(x_t)-d_t,v_t-x_t \rangle + 
\frac{\bar{L}\eta_t^2D^2}{2} \\
={}& F(x_t) - \eta_t \mathcal{G}(x_t) + \eta_t \langle \nabla 
F(x_t)-d_t,v_t-v^\prime_t \rangle + \frac{\bar{L}\eta_t^2D^2}{2} \\
\stackrel{(d)}{\leq}{}& F(x_t) - \eta_t \mathcal{G}(x_t) + \eta_t \|\nabla 
F(x_t)-d_t\|\| v_t-v^\prime_t\| + \frac{\bar{L}\eta_t^2D^2}{2} \\
\stackrel{(e)}{\leq}{}& F(x_t) - \eta_t \mathcal{G}(x_t) + \eta_tD \|\nabla 
F(x_t)-d_t\| + \frac{\bar{L}\eta_t^2D^2}{2},
\end{split}
\end{equation*}
where we used the fact that $F$ is $\bar{L}$-smooth in inequality (a). 
Inequalities 
(b), (e) hold because of \cref{assum:constraint}. Inequality (c) is due to the 
optimality of $v_t$, and 
in (d), we applied the Cauchy-Schwarz inequality.

Rearrange the inequality above, we have 
\begin{equation}
\label{eq:bound_on_individual_gap-sfw}
\eta_t \mathcal{G}(x_t) \leq F(x_t)- F(x_{t+1})+ \eta_tD \|\nabla 
F(x_t)-d_t\| + \frac{\bar{L}\eta_t^2D^2}{2}.   
\end{equation}

Apply \cref{eq:bound_on_individual_gap-sfw} recursively for $t = 1, 2, \cdots, 
T$, and take expectations, we attain the following inequality:
\begin{equation*}
\begin{split}
\sum_{t=1}^T \eta_t \expect[\mathcal{G}(x_t)]
\leq F(x_1)-F(x_{T+1})
+D\sum_{t=1}^T \eta_t\expect[\|\nabla F(x_t)-d_t\|] 
+\frac{\bar{L}D^2}{2}\sum_{t=1}^T \eta_t^2.
\end{split}
\end{equation*}

By Jensen's inequality and \cref{lem:graident_error} with $\alpha=2/3$, we have
\begin{equation*}
\expect[\|\nabla F(x_t)-d_t\|] \leq \sqrt{\expect[\|\nabla 
F(x_t)-d_t\|^2]} \leq \frac{\sqrt{C}}{t^{1/3}},
\end{equation*}
where $C = \max \{\frac{2(2G+D\bar{L})^2}{4/3-2^{-2/3}}, \left( 
\frac{2}{4/3-2^{-2/3}}\right)^4, [2D(\bar{L}+L)]^4 \}$. Since $\eta_t = 
T^{-2/3}$, we have
\begin{equation*}
\begin{split}
\expect[\mathcal{G}(x_o)] &= \frac{\sum_{t=1}^T 
\expect[\mathcal{G}(x_t)]}{T} \\
&\leq \frac{1}{T\cdot T^{-2/3}} [F(x_1) - F(x_{T+1}) + D\sum_{t=1}^T 
T^{-2/3}\frac{\sqrt{C}}{t^{1/3}} + \frac{\bar{L}D^2}{2}\sum_{t=1}^T T^{-4/3}] \\
&\leq \frac{1}{T^{1/3}}[2B + 
D\sqrt{C}T^{-2/3}\frac{3}{2}T^{2/3}+\frac{\bar{L}D^2}{2T^{1/3}}] \\
&= \frac{2B+3\sqrt{C}D/2}{T^{1/3}} + \frac{\bar{L}D^2}{2T^{2/3}}, 
\end{split}
\end{equation*}
where the second inequality holds because $\sum_{t=1}^T t^{-1/3} \leq 
\int_{0}^T x^{-1/3}\mathrm{d}x = \frac{3}{2}T^{2/3}$.

\subsection{Proof of Theorem 3}
First, since $x_{t+1} = x_t + \eta_t v_t = x_t + T^{-1} v_t$, we have
$x_{T+1} = \frac{\sum_{t=1}^T v_t}{T} \in \constraint$. Also, because now 
$\|x_{t+1} -x_{t} \| = \|\eta_t v_t \| \leq \eta_t R$, (rather than 
$\eta_t D$), \cref{lem:graident_error} holds with new constant $C = \max 
\{\frac{2(2G+R\bar{L})^2}{2-2^{-\alpha}-\alpha}, 
\left(\frac{2}{2-2^{-\alpha}-\alpha} \right)^4, [2R(\bar{L}+L)]^4 \}$. Since 
$\alpha=1$, we have $C=\max\{4(2G+R\bar{L})^2, 256, [2R(\bar{L}+L)]^4 \}$.
Then by Jensen's inequality, we have
\begin{equation*}
\expect[\|\nabla F(x_t) -d_t \|] \leq \sqrt{\expect[\|\nabla F(x_t) 
-d_t \|^2]}\leq \frac{\sqrt{C}}{\sqrt{t}}.
\end{equation*}	

We observe that 
\begin{equation}\label{eq:sub_aux1}
\begin{split}
F(x_{t+1}) &\stackrel{(a)}{\geq} F(x_t) + \langle \nabla F(x_t), 
x_{t+1}-x_t \rangle -\frac{\bar{L}}{2} \|x_{t+1} -x_t \| \\
&= F(x_t) + \frac{1}{T} \langle \nabla F(x_t), v_t \rangle - 
\frac{\bar{L}}{2T^2} \|v_t \| \\
&\stackrel{(b)}{\geq} F(x_t) + \frac{1}{T} \langle d_t, v_t \rangle + 
\frac{1}{T} \langle \nabla F(x_t) -d_t, v_t \rangle - 
\frac{\bar{L}R^2}{2T^2} \\
&\stackrel{(c)}{\geq} F(x_t) + \frac{1}{T} \langle d_t, x^* \rangle + 
\frac{1}{T} \langle \nabla F(x_t) -d_t, v_t \rangle - 
\frac{\bar{L}R^2}{2T^2} \\
&= F(x_t) + \frac{1}{T} \langle \nabla F(x_t), x^* \rangle + 
\frac{1}{T} \langle \nabla F(x_t) -d_t, v_t - x^*\rangle - 
\frac{\bar{L}R^2}{2T^2} \\
&\stackrel{(d)}{\geq} F(x_t) + \frac{F(x^*)-F(x_t)}{T}- \frac{1}{T} 
\langle \nabla F(x_t) -d_t, -v_t + x^*\rangle - \frac{\bar{L}R^2}{2T^2} 
\\
&\stackrel{(e)}{\geq} F(x_t) + \frac{F(x^*)-F(x_t)}{T}- \frac{1}{T} 
\|\nabla F(x_t) -d_t\|\cdot \|-v_t + x^*\| - \frac{\bar{L}R^2}{2T^2} \\
&\stackrel{(f)}{\geq}  F(x_t) + \frac{F(x^*)-F(x_t)}{T}- \frac{1}{T} 
2R\|\nabla F(x_t) -d_t\| - \frac{\bar{L}R^2}{2T^2},
\end{split} 
\end{equation}
where inequality $(a)$ holds because of the $\bar{L}$-smoothness of $F$, 
inequalities $(b), (e)$ comes from \cref{assum:constraint}. We used the 
optimality of $v_t$ in inequality $(c)$, and applied the Cauchy-Schwarz 
inequality in$(e)$. Inequality $(d)$ is a little involved, since $F$ is 
monotone and concave in positive directions, we have
\begin{equation*}
\begin{split}
F(x^*) - F(x_t) &\leq F(x^* \vee x_t) - F(x_t) \\
&\leq \langle \nabla 
F(x_t), x^* \vee x_t - x_t \rangle \\
&= \langle \nabla F(x_t), (x^*- 
x_t) \vee 0 \rangle \\
&\leq \langle \nabla F(x_t), x^* \rangle.
\end{split}
\end{equation*}  

Taking expectations on both sides of \cref{eq:sub_aux1},
\begin{equation*}
\expect[F(x_{t+1})] \geq \expect[F(x_t)] + 
\frac{F(x^*)-\expect[F(x_t)]}{T}- \frac{2R}{T}\frac{\sqrt{C}}{\sqrt{t}}  - 
\frac{\bar{L}R^2}{2T^2}.
\end{equation*}

Or 
\begin{equation*}
F(x^*) - \expect[F(x_{t+1})] \leq 
(1-\frac{1}{T})[F(x^*)-\expect[F(x_t)]] + 
\frac{2R}{T}\frac{\sqrt{C}}{\sqrt{t}} + \frac{\bar{L}R^2}{2T^2}.
\end{equation*}

Apply the inequality above recursively for $t = 1, 2, \cdots, T$, we have
\begin{equation*}
\begin{split}
F(x^*) - \expect[F(x_{T+1})] &\leq (1-\frac{1}{T})^T [F(x^*)-F(x_1)] + 
\frac{2R\sqrt{C}}{T} \sum_{t=1}^T t^{-1/2} + \frac{\bar{L}R^2}{2T} \\
&\leq e^{-1} F(x^*) + \frac{4R\sqrt{C}}{T^{1/2}} + \frac{\bar{L}R^2}{2T},
\end{split}
\end{equation*}
where the second inequality holds since $\sum_{t=1}^T t^{-1/2} \leq \int_{0}^T 
x^{-1/2}\mathrm{d}x = 2T^{1/2}$. Thus we have
\begin{equation*}
\expect[F(x_{T+1})] \geq (1-e^{-1})F(x^*) - \frac{4R\sqrt{C}}{T^{1/2}} - 
\frac{\bar{L}R^2}{2T}.
\end{equation*}
	
	\chapter{Quantized Frank-Wolfe}\label{cha:qfw}

\section{Introduction}

In this chapter\footnote{\revise{This chapter is based on our work in 
		\citep{zhang2020quantized}.}}, we study the application of Frank-Wolfe 
		methods to large-scale 
problems. To be precise, we develop \AlgQFW (QFW), a general 
communication-efficient distributed FW framework for both convex and non-convex 
objective 
functions. We study the performance of QFW in two widely 
recognized 
settings: 1) stochastic optimization and 2) finite-sum optimization.


To be more specific, let $\ccalK\subseteq \reals^d$ be the constraint set.
In \textit{constrained stochastic optimization} the goal is to 
solve
\begin{equation}\label{stochastic_problem}
\min_{x \in \constraint} f(x)\ :=\ \min_{x \in \constraint} \expect_{z\sim 
P}[\tilde{f}(x,z)],
\end{equation}
where $x\in \reals^d$ is the optimization variable,  $z \in 
\reals^q$ is a random variable drawn from a probability distribution $P$, 
which determines the choice of a stochastic function 
$\tilde{f}:\reals^d\times \reals^q\to \reals$. For \textit{constrained 
finite-sum optimization}, we further assume that  $P$  is a uniform 
distribution 
over $[N] = \{1, 2, \cdots, N\}$ and the goal is to solve a special case of 
problem~\eqref{stochastic_problem}, namely, 
\begin{equation*}\label{finite_sum_problem}
\min_{x \in \constraint} f(x)\ :=\ \min_{x \in \constraint} \frac{1}{N} 
\sum_{j=1}^N f_j(x).
\end{equation*} 
In parallel settings, we suppose that we have a computing system consisting of  a master node and $M$ 
workers, and each worker maintains a local copy of $x$.
At every iteration of the stochastic case, each 
worker 
has access to independent stochastic 
gradients of $f$; whereas in the finite-sum case, we assume $N=Mn$, thus the 
objective function can be 
decomposed as $f(x)=\frac{1}{Mn}{\sum_{m\in[M],j\in[n]}f_{m,j}(x)}$, 
and each worker $m$ has access to the exact gradients of $n$ component 
functions $f_{m,j}(x)$ for all 
$j \in [n]$.

This way the task of computing gradients is divided among the 
workers.  The 
master node aggregates local gradients from the workers, and sends the 
aggregated gradients back to them so that each worker can update the 
model (\emph{i.e.}, their own iterate) locally. 
Thus, by 
transmitting quantized gradients, we can reduce the communication complexity 
(\emph{i.e.}, number of transmitted bits) 
significantly. The workflow diagram of the proposed \AlgQFW
scheme is 
summarized in Figure \ref{fig:workflow2}. 
We should highlight that there is a 
trade-off between gradient 
quantization and information flow. Intuitively, more intensive 
quantization reduces the communication cost, but also loses more 
information, which may decelerate the convergence rate.

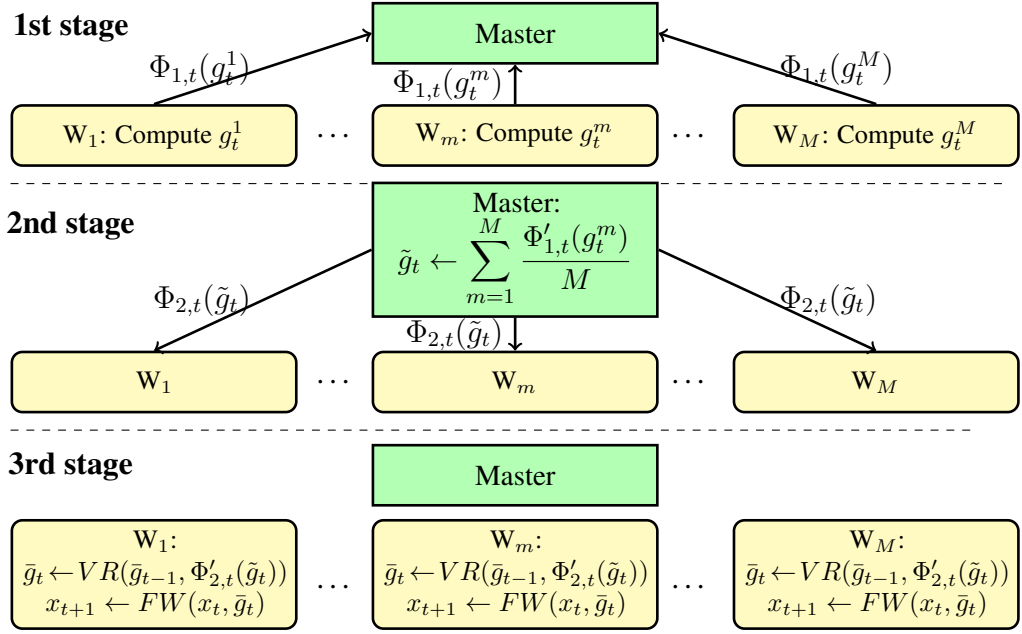
\begin{figure*}[t]
	\begin{center}

\def \thisplotscale {0.8}
\def \unit {1 cm}

%
\tikzstyle{block}         = [ draw,
                              rectangle, rounded corners,
                              minimum height = 0.8*\unit,
                              minimum width  = 1.1*\unit,
                              text width     = 3.5*\unit,
                              text badly centered,
                              line width=1pt,
                              fill = yellow!30, 
                              font = \footnotesize, 
                              anchor = west]
\tikzstyle{bold block}    = [ block,
                              fill = blue!40]
\tikzstyle{light block}   = [ block,
                              fill = blue!10]
\tikzstyle{connector}     = [ draw, 
                              -stealth, 
                              shorten >=2,
                              shorten <=2,]
\tikzstyle{dot dot dot}   = [ draw, 
                              dashed]

\tikzstyle{blockbig}         = [ draw,
                              rectangle, 
                               minimum height = 0.8*\unit,
                              minimum width  = 1.1*\unit,
                              text width     = 3.5*\unit,
                              text badly centered,
                               line width=1pt,
                              fill = green!30, 
                              font = \small, 
                              anchor = west]

\begin{tikzpicture}[scale=0.8]

\node at (-2,15.3) {\textbf{1st stage}};

\path (3,15.2)                 node [blockbig] (C0) {Master}; 

\path (-3,13.5)                 node [block] (C1) {W$_1$: \black{Compute 
$g_t^1$}}; 

\path (3,13.5)                 node [block] (C2) {W$_m$: \black{Compute 
$g_t^m$}}; 
\path (9,13.5)                 node [block] (C3) {W$_M$: \black{Compute 
$g_t^M$}}; 

\node at (2.4,13.5) {\ldots};
\node at (8.3,13.5) {\ldots};

\draw [dashed] (-3,12.7) -- (13.5,12.7);

\draw [->,line width=1pt, black] (C1.north) -- (C0.west)  
node[midway,above,left] {$\black{\Phi_{1,t}(g_t^1)}$} ; 
\draw [->,line width=1pt, black]  (C2.north) -- (C0.south) node[midway,left] 
{$\black{\Phi_{1,t}(g_t^m)}$}; 
\draw [->,line width=1pt, black] (C3.north) -- (C0.east)   node[midway,right] 
{$\black{\Phi_{1,t}(g_t^M)}$}; 

\node at (-2,12) {\textbf{2nd stage}};

\path (3,11.6)                 node [blockbig] (C10) {Master:\\ 
\black{$\displaystyle{\tilde{g}_t \gets \sum_{m=1}^M 
\frac{\Phi_{1,t}^{\prime}(g_t^m)}{M}}$}}; 

\path (-3,9.4)                 node [block] (C11) {W$_1$}; 
\path (3,9.4)                 node [block] (C12) {W$_m$}; 
\path (9,9.4)                 node [block] (C13) {W$_M$}; 

\node at (2.4,9.4) {\ldots};
\node at (8.3,9.4) {\ldots};

\draw [->,line width= 1pt, black]  (C10.west) -- (C11.north) 
node[midway,above,left] {$\black{\Phi_{2,t}(\tilde{g}_t)}$} ; 
\draw [->,line width=1pt, black]  (C10.south) -- (C12.north)  node[midway,left] 
{$\black{\Phi_{2,t}(\tilde{g}_t)}$}; 
\draw [->,line width=1pt, black]  (C10.east)  -- (C13.north)  
node[midway,right] {$\black{\Phi_{2,t}(\tilde{g}_t)}$}; 
\draw [dashed] (-3,8.6) -- (13,8.6);

\node at (-2,8) {\textbf{3rd stage}};

\path (3,7.85)                 node [blockbig] (C90) {Master}; 

\path (-3,6.2)                 node [block] (C91)  {W$_1$:\\
		\black{$\bar{g}_t \!\gets \!VR(\bar{g}_{t-1} 
		,\Phi_{2,t}^{\prime}(\tilde{g}_t))$ $x_{t+1} \gets FW(x_t,\bar{g}_t 
		)$}}; 
\path (3,6.2)                 node [block] (C92) {W$_m$:\\
		\black{$\bar{g}_t \!\gets \!VR(\bar{g}_{t-1} 
		,\Phi_{2,t}^{\prime}(\tilde{g}_t))$ $x_{t+1} \gets FW(x_t,\bar{g}_t 
		)$}}; 
\path (9,6.2)                 node [block] (C93) {W$_M$:\\
		\black{$\bar{g}_t \!\gets \!VR(\bar{g}_{t-1} 
		,\Phi_{2,t}^{\prime}(\tilde{g}_t))$ $x_{t+1} \gets FW(x_t,\bar{g}_t 
		)$}};

\node at (2.4,6) {\ldots};
\node at (8.3,6) {\ldots};
\end{tikzpicture} 
	\end{center}
	\caption{Stages of our general \AlgQFW scheme at time $t$. 
		In the first stage, each worker $m$ computes its local gradient 
		{information} 
		$g_t^m$ and sends the quantized version $\Phi_{1,t}(g_t^m)$ to the 
		master node. In the second stage, master computes the average of 
		decoded received signals $\Phi_{1,t}^{\prime}(g_t^m)$, \emph{i.e.}, 
		$\tilde{g}_t \gets 
		({1}/{M})\sum_{m=1}^M {\Phi_{1,t}^{\prime}(g_t^m)}$ and then sends 
		its 
		quantized version $\Phi_{2,t}(\tilde{g}_t)$ to the workers. Note that 
		the two quantization schemes $\Phi_{1,t}, \Phi_{2,t}$ depend on $t$ and 
		can be different from each other. In the 
		third 
		stage, 
		workers use the decoded gradient information computed by all workers 
		$\Phi'_{2,t}(\tilde{g}_t)$ and their previous gradient estimation 
		$\bar{g}_{t-1}$ to update their new gradient estimation $\bar{g}_{t}$ 
		via a variance reduction (VR) scheme. Once the variance reduced 
		gradient 
		approximation $\bar{g}_{t}$ is evaluated, workers compute the new 
		variable 
		$x_{t+1}$ by following the update of Frank-Wolfe 
		(FW).}\label{fig:workflow2}
	\vspace{2mm}
\end{figure*}


\textbf{Our contributions:} In this chapter, we propose a novel 
distributed projection-free framework that handles quantization for constrained convex and 
non-convex optimization problems in finite-sum and stochastic cases. It is 
well-known that 
unlike projected gradient-based methods, FW methods 
may diverge when fed with stochastic gradient 
\citep{hazan2016variance,mokhtari2018stochastic}. 
Indeed, a similar issue arises in a distributed setting where nodes exchange 
\textit{quantized gradients} which are noisy estimates of the gradients. By 
incorporating appropriate variance reduction techniques, 
we show that with quantized gradients, we can obtain a 
provably convergent method which preserves the 
convergence rates of the state-of-the-art vanilla centralized methods in all 
the considered cases 
\citep{zhang2019one,shen2019complexities,hassani2019stochastic,yurtsever2019conditional}.
We 
believe 
our work presents the first quantized, distributed, and projection-free method. 
Our theoretical results for \AlgQFW (QFW) are summarized in   
\cref{tab:contribution}, where the SFO complexity is the 
required number of stochastic gradients in stochastic case, and the IFO 
complexity is the number of exact gradients for component functions in 
finite-sum case. For the convex case, the complexity indicates 
the 
number of gradients to achieve an $\eps$-suboptimal solution; while in the 
non-convex case, it refers to the number of gradients to find a 
first-order $\epsilon$-stationary point.
We note that since the $M$ workers compute the 
gradients simultaneously, the time to obtain gradients is proportional to the 
SFO/IFO complexity \emph{per worker}. So we report the 
SFO/IFO complexity per worker, as in 
 many other works on parallel optimization 
(\emph{e.g.}, Sign-SGD \citep{bernstein2018signsgd}). The results in \cref{tab:contribution} 
show that more workers can decrease the SFO/IFO complexity per worker 
effectively, 
and thus accelerate the optimization procedure.
All the proofs in this chapter are provided in \cref{sec:proof-qfw}.

\begin{table}[t!]
	\centering
	\caption{SFO/IFO Complexity per worker in different settings ($M$ is 
	the number of workers).}
	\label{tab:contribution}
		\begin{tabular}{lll}  
			\toprule[1.5pt]
			Setting & Function   & SFO/IFO Complexity  \\ 
			\midrule
			\vspace{2mm}
			Finite-sum & Convex  & 
			$\displaystyle{\mathcal{O}\bigg(\frac{N\ln(1/\epsilon)+
					1/\epsilon^2}{M}\bigg)}$ \\ 
			\vspace{2mm}
			Finite-sum & Non-convex & 
			$\displaystyle{\mathcal{O}\bigg(\frac{\sqrt{N}}{
			\epsilon^2\sqrt{M}}\bigg)}$ \\ 
			\vspace{2mm}
			Stochastic & Convex  & $\displaystyle{\mathcal{O}\bigg(\frac{1}{M\epsilon^2 }\bigg)}$ \\
			\vspace{2mm}
			Stochastic & Non-convex  & 
			$\displaystyle{\mathcal{O}\bigg(\frac{1}{\epsilon^3 \sqrt{M}}\bigg)}$ \\
			\bottomrule[1.25pt]
		\end{tabular}
\end{table}

\section{Gradient Quantization Schemes} \label{sec:encoding schemes}
In most distributed optimization algorithms, the task of computing gradients 
is 
divided among the workers, and the master node uses parts of gradients 
at the workers to update the model (iterate) directly or sends the aggregated 
gradients to the worker so that each of them can update the model (iterate) 
locally. Therefore, the information that workers need to send to the master 
is the elements of the objective function gradient. Thus, by transmitting 
quantized gradients, we can reduce the communication bits effectively.
In this section, we introduce a 
quantization 
scheme called  s-\mscheme and explain how this scheme reduces the overall cost of exchanging gradients. Consider the gradient vector $g\in \reals^d$ and let $g_i$ be the $i$-th 
coordinate of the gradient. The s-\mscheme encodes 
$g_i$ into an element from the set 
$\{\pm1,\pm\frac{s-1}{s},\cdots,\pm\frac{1}{s},0\}$ in a random way. To do so, 
we first compute the ratio ${|g_i|}/{\|g\|_\infty}$ and find the 
indicator $l_i \in \{0,1,\cdots,s-1\}$ such that $|g_i|/\|g\|_\infty \in 
[l_i/s, (l_i+1)/s]$. 
Then we define the random variable $b_i$ as
\begin{equation*}\label{s_multi_prob_dist}
b_i=
\begin{cases}
l_i/s, & \quad \text{w.p.}\ \ 
1-\frac{|g_i|}{\|g\|_\infty}s+l_i, \\
(l_i+1)/s, & \quad \text{w.p.}\ \  
\frac{|g_i|}{\|g\|_\infty}s-l_i.
\end{cases}
\end{equation*}
Finally, instead of transmitting $g_i$, we send $\sign(g_i) \cdot b_i$, 
alongside the norm 
$\|g\|_\infty$. It can be verified that 
$\expect[b_i|g]= {|g_i|}/{\|g\|_\infty}$. So we define the corresponding 
decoding scheme as $\phi^{\prime}(g_i) = \sign(g_i) b_i\|g\|_\infty$ to ensure 
that $\phi^{\prime}(g_i)$ is an unbiased estimator of $g_i$. We note that the 
encoding/decoding schemes in \cref{fig:workflow2} are 
denoted as capital $\Phi/\Phi'$, indicating that they can be any general 
schemes. The proposed $s$-\mscheme is denoted by $\phi/\phi'$.  
We also note that this 
quantization scheme is similar to the Stochastic Quantization method in 
\citep{alistarh2017qsgd}, except that we use $\ell_\infty$-norm while they adopt the 
$\ell_2$-norm.
In the $s$-\mscheme,  for each coordinate $i$, we need 1 bit to 
transmit 
$\sign(g_i)$. Moreover, since $b_i \in \{0,{1}/{s},\dots, (s-1)/{s}, 
1\}$, we need $z=\log_2(s+1)$
bits to send $b_i$. Finally, we need 32 bits to transmit $\|g\|_\infty$. Hence, 
the total number of communicated bits is  $32+d(z+1)$. Here, by ``bits'' we 
mean the number of 0’s and 1’s transmitted.

One major advantage of the $s$-\mscheme is that by tuning the 
partition parameter $s$ or the corresponding assigned bits $z$, we can smoothly 
control the trade-off between gradient 
quantization and information loss, which helps distributed 
algorithms to attain their best performance. We proceed to characterize the variance of the $s$-\mscheme. 


\begin{lemma}
	\label{lem:var_multi}
	The variance of  $s$-\mscheme $\phi$ for any $g\in \reals^d$ is 
	bounded by 
	\begin{equation*}\label{eq:var_multi_claim}
	\var{\phi^{\prime}(g)|g} \leq \frac{d}{s^2}\|g\|_\infty^2.
	\end{equation*}
\end{lemma}
\cref{lem:var_multi} demonstrates the trade-off between the 
error of quantization and the communication cost for $s$-\mscheme. In a 
nutshell, for larger choices of $s$, the variance is 
smaller, which in turn results in higher communication cost. 
If we set $s=1$, we obtain the \sscheme,
which requires communicating the encoded 
scalars $\sign(g_i) b_i \in \{\pm1,0\}$ and the norm $\|g\|_{\infty}$. Since 
$z 
= 
\log_2(s+1)=1$, the overall communicated bits for each worker are 
$32+2d$ per round. We characterize its variance in \cref{lem:var_single}. 

\begin{lemma}
\label{lem:var_single}
	The variance of  \sscheme is given by
	\begin{equation*}\label{lem:var_single_claim}
	\var{\phi^{\prime}(g)|g}= \|g\|_1\|g\|_\infty - \|g\|_2^2.
	\end{equation*}
\end{lemma}
\begin{remark}
	For the probability distribution of the random variable $b_i$, 
	instead of   $\|g\|_\infty$, we can use  other norms $\|g\|_p$ (where $ p 
	\geq 1$). But it can be verified that the $\ell_\infty$-norm
	leads to the smallest variance for \sscheme. That is also the reason why we 
	do not use 
	$\ell_2$-norm as in \citep{alistarh2017qsgd}.
\end{remark}
\section{Convex Minimization}
In this section, we analyze the convex minimization problem in both finite-sum 
and stochastic settings. Note that even in the setting without quantization, 
if we use stochastic gradients in the update of FW, it might diverge 
	\citep{hazan2016variance,mokhtari2018stochastic}. So appropriate variance 
	reduction 
	techniques are 
	needed for communicating 
	quantized gradients. 
	\citet{nguyen2017sarah,nguyen2017stochastic,nguyen2019optimal} 
	developed the 
	StochAstic Recursive grAdient algoritHm (SARAH), a stochastic recursive 
	gradient update framework. \citet{fang2018spider} proposed  Stochastic 
	Path-Integrated Differential Estimator (SPIDER) technique, a variant of 
	SARAH, 
	for centralized unconstrained 
	optimization. 
	Recently, 
	\citet{hassani2019stochastic,2019complexities,yurtsever2019conditional} 
	proposed the SPIDER variants of FW method 
	for both 
	convex and non-convex optimization problems. Similar variance reduction 
	idea was also combined with SGD to solve non-convex finite-sum problems in 
	\citep{zhou2018stochastic}.
	In this chapter, we generalize 
	SPIDER to the constrained and distributed settings. 

We first consider the  case where no quantization is performed. Let $\{p_i\} 
\in 
\mathbb{N}^+$ be a sequence of period parameters. At the 
beginning of each period $i$, namely, $t = \sum_{j=1}^{i-1}p_j+1$, each worker 
$m$ samples $S_{i,1}$ component functions in finite-sum case, or stochastic 
functions in stochastic case, which are denoted as $\mathcal{S}_{i,1}^m$. 
We define the local average gradient on set $\mathcal{S}_{i,1}^m$ as
$g_{i,1}^m \triangleq \nabla f_{\mathcal{S}_{i,1}^m}(x_t) = \frac{1}{S_{i,1}} \sum_{j \in 
\mathcal{S}_{i,1}^m} \nabla 
f_j(x_t)$.
Then each 
worker $m$ computes the average of all these local gradients $g_{i,1}^m$ and 
sends 
it to the master. 
Then, master node calculates the average of the $M$ received signals and 
broadcasts 
it to all workers. Then, the workers update their gradient estimation 
$\bar{g}_t$ 
as the averaged signal
$\bar{g}_{t} =  \frac{1}{M}\sum_{m=1}^M g_{i,1}^m.$

Note $\bar{g}_t$ is identical for all the workers. In the rest
of 
that 
period, \emph{i.e.}, $t = \sum_{j=1}^{i-1}p_j+k$, where $2\leq k \leq p_i$, 
each 
worker $m$ samples a set of local functions, denoted as 
$\mathcal{S}_{i,k}^m$, of 
size $S_{i,k}$ uniformly at random, 
and computes the difference of averages of these gradients 
\begin{equation*}\label{eq:g2}
g_{i,k}^m \triangleq 
\nabla 
f_{\mathcal{S}_{i,k}^m}(x_t) - \nabla 
f_{\mathcal{S}_{i,k}^m}(x_{t-1}),
\end{equation*}
and sends it to master. Then 
master node calculates the average of the $M$ signals and broadcasts it to all 
the workers. The workers update their gradient estimation $g_t$ as
\begin{equation*}\label{g_avg_update}
\bar{g}_{t} = \bar{g}_{t-1}+\frac{1}{M}\sum_{m=1}^Mg_{i,k}^m.
\end{equation*}

\begin{algorithm}[t]
	\begin{algorithmic}[1]
		\Require constraint set $\constraint$, total iteration 
		number $T$, No. of workers 
		$M$, period parameters $\{p_i\}$, sample sizes $\{S_{i,k}\}$, learning 
		rate $\eta_t$,
		initial point $x_1 \in \constraint$ 
		\Ensure $x_{T+1}$ or $x_o$,
		where $x_o$ is chosen from $\{x_1, x_2, 
		\cdots, x_T\}$ uniformly at random
		\For{$t=1,2,\dots, T$}
		\State Set $x_{i,k} \gets x_t$, where $t = \sum_{j=1}^{i-1}p_j + k, 
		1\le k 
		\le p_i$
		\State Each worker $m$ computes local 
		gradient $g_{i,k}^m$ by $g_{i,1}^m = \nabla 
		f_{\mathcal{S}_{i,1}^m}(x_{i,k})=\nabla 
		f_{\mathcal{S}_{i,1}^m}(x_{t})$ for $k=1$, or $g_{i,k}^m \triangleq 
		\nabla 
		f_{\mathcal{S}_{i,k}^m}(x_{i,k}) - \nabla 
		f_{\mathcal{S}_{i,k}^m}(x_{i,k-1})=\nabla 
		f_{\mathcal{S}_{i,k}^m}(x_t) - \nabla 
		f_{\mathcal{S}_{i,k}^m}(x_{t-1})$ for $k\ge 2$		
		\State Each worker $m$ encodes $g_{i,k}^m$ as $\Phi_{1,i,k}(g_{i,k}^m)$ 
		and pushes it to the master
		\State Master decodes $\Phi_{1,{i,k}}(g_{i,k}^m)$ as 
		$\Phi_{1,{i,k}}^{\prime}(g_{i,k}^m)$, and computes 
		$\tilde{g}_{i,k} \gets 
		\frac{1}{M}\sum_{m=1}^M \Phi_{1,{i,k}}^{\prime}(g_{i,k}^m)$
		\State Master encodes $\tilde{g}_{i,k}$ as 
		$\Phi_{2,{i,k}}(\tilde{g}_{i,k})$, 
		and broadcasts it to all workers
		\State Workers decode
		$\Phi_{2,{i,k}}(\tilde{g}_{i,k})$ as 
		$\Phi_{2,{i,k}}^{\prime}(\tilde{g}_{i,k})$ 
		\If{$k = 1$}
		\State Workers update $\bar{g}_{i,k} \gets 
		\Phi_{2,{i,k}}^{\prime}(\tilde{g}_{i,k})$
		\Else
		\State Workers update $\bar{g}_{i,k} \gets 
		\Phi_{2,{i,k}}^{\prime}(\tilde{g}_{i,k})+\bar{g}_{i,k-1}$
		\EndIf
		\State Each worker updates
		$x_{t+1} \gets 
		x_t + \eta_{t}(v_{t}-x_t) = x_{i,k} + \eta_{i,k}(v_{i,k}-x_{i,k})$ 
		where $v_{i,k} \gets \argmin_{v 
			\in \constraint} \langle v, \bar{g}_{i,k} \rangle$
		\EndFor
	\end{algorithmic}
	\caption{\AlgQFW (QFW)}\label{alg:convex_dist_fw}
\end{algorithm}

So $\bar{g}_t$ is still identical for all the workers.
In order to incorporate quantization, each worker simply pushes the quantized 
version of the average gradients. Then the 
master decodes the quantizations, encodes the average of decoded 
signals in a quantized fashion, and broadcasts the quantization. Finally, 
each worker 
decodes the quantized signal and updates $x_t$ 
locally. 
To be more specific, in the quantized setting, in each iteration $t$ such that 
$t = \sum_{j=1}^{i-1}p_j+k$ where $1\leq k \leq p_i$, each 
worker $m$ sends the 
quantized version of its local gradient information $\Phi_{1,t}(g_{i,k}^m)$ 
to the master. Once 
master collects all the quantized information, it decodes them, \emph{i.e.}, 
finds 
$\{\Phi_{1,t}'(g_{i,k}^m)\}_{m=1}^{M}$, computes their average 
$\tilde{g}_{t}$, and 
sends its quantized version $\Phi_{2,t}(\tilde{g}_{t})$ to all workers. 
Then, all the workers decode the received quantized signal and use it 
as their new gradient 
approximation $\bar{g_t}$ and update their variable according to the update of 
Frank-Wolfe, \emph{i.e.},
\begin{equation*}\label{FW_update}
x_{t+1} =
x_t 
+ \eta_t(v_t-x_t),
\end{equation*}
where $v_{t} \gets \argmin_{v 
	\in \constraint} \langle v, \bar{g}_t \rangle$.
The description of our proposed \AlgQFW (QFW) method is shown in 
\cref{fig:workflow2} and outlined 
in \cref{alg:convex_dist_fw}.

\begin{remark}
	The model update (Line 13 in \cref{alg:convex_dist_fw}) should 
	be performed at 
	each worker. Since all the 
	linear programming problems (to obtain $v_t$) are solved 
	simultaneously, the total 
	running time is the same with that where the model updating is 
	performed in the master node. However, additional 
	variance 
	would be introduced if the master node updates the model and 
	broadcasts it in quantized manners. Thus the master-updating method 
	lacks 
	theoretical justification regarding convergence, and we adopt the 
	worker-updating approach.
\end{remark}

\subsection{Finite-Sum Setting}
Now we proceed to establish the convergence properties of our proposed QFW in 
the finite-sum setting. Recall that we assume that there are $N$ functions 
and $M$ workers in total, and each worker $m$ has access to $n=N/M$ functions 
$f_{m,j}$ for $j \in [n]$. We first make two assumptions on the 
constraint set and component functions. Let $\| \cdot \|$ denote the $\ell_2$ 
norm in Euclidean space through out the chapter. 

\begin{assump}
	\label{assump_on_K}
	The constraint set $\mathcal{K}$ is convex and 
	compact, with diameter $D = \sup_{x,y \in \mathcal{K}} \|x-y 
	\|$. 
\end{assump}

\begin{assump}
	\label{assump_on_f_i_convex}
	The functions $f_{m,i}$ are convex, $L$-smooth on 
	$\constraint$,
	and satisfy that $\|\nabla f_{m,i}(x)\|_\infty \leq G_\infty, 
		\forall\ m \in [M], i \in [n], x, y \in \constraint$.
\end{assump}

\begin{theorem}[Finite-Sum Convex]
	\label{thm:finite_convex}
	Consider QFW outlined in Algorithm~\ref{alg:convex_dist_fw}.
	Under \cref{assump_on_K,assump_on_f_i_convex},
	if we set $ p_i = 2^{i-1}, \mathcal{S}_{i,1}^m = \{f_{m,j}: j \in [n]\}$ 
	(\emph{i.e.}, each worker $m$ samples 
	all its $n$ component functions), $S_{i,k} = p_i/M = 2^{i-1}/M, 
	\forall\ i \ge 
	1, k \ge 2$, and $\eta_{i,k}=2/(p_i+k)={2}/(2^{i-1}+k)$, and use the 
	$s_{1,i,1}=(\sqrt{\frac{dp_i^2}{M}})$-\mscheme,  
	$s_{2,i,1}=(\sqrt{dp_i^2})$-\mscheme for $k=1$, and 
	$s_{1,i,k}=(\sqrt{\frac{dp_i}{M}})$-\mscheme,  
	$s_{2,i,k}=(\sqrt{dp_i})$-\mscheme for $k \ge 2$
	as 
	$\Phi_{1,i,k}$ and $\Phi_{2,i,k}$ in \cref{alg:convex_dist_fw}, 
	then the output $x_{T+1} \in \constraint$  satisfies
	\begin{equation*}
	\expect[f(x_{T+1})]-f(x^*) \leq 
	\frac{4D\sqrt{2(G^2_\infty+6L^2D^2)}
	+2LD^2}{T},  
	\end{equation*}
	where $x^*$ 
	is a minimizer of $f$ on $\constraint$. 
\end{theorem}

\begin{corollary}\label{cor:convex_finite}
To obtain an $\epsilon$-suboptimal solution, we need to run the QFW method for at most $\mathcal{O}(1/\epsilon)$ iterations. The IFO complexity per worker in this case  is 
$\mathcal{O}(\frac{N\ln(1/\epsilon)+1/\epsilon^2}{M})$. 	
\end{corollary}
\cref{cor:convex_finite} shows IFO complexity per worker is linear in $1/M$,
which implicates that increasing the number of workers $M$ will decrease the 
IFO complexity per worker effectively, thus accelerate the optimization 
procedure. Also, our numerical experiments in 
\cref{sec:experiments-qfw} showed that our proposed 
method requires significantly fewer bits than the unquantized version to achieve a specific accuracy.

\subsection{Stochastic Setting}
QFW can also be applied to the stochastic case. Recall that in the stochastic setting we 
assume that the objective function is $f(x)\ = \expect_{z\sim 
P}[\tilde{f}(x,z)]$ and each worker has access to independent samples 
$\tilde{f}(x,z)$. Before proving the convergence properties of QFW 
for the stochastic setting, we first make a standard assumption on 
$\tilde{f}(x,z)$.

\begin{assump}
	\label{assump_on_f_stoch_convex}
	The stochastic function $\tilde{f}(x,z)$ is convex, $L$-smooth on 
	$\constraint$. The gradient $\nabla \tilde{f}(x,z)$ is an unbiased estimate 
	of $\nabla f(x)$ with bounded variance $\sigma^2$,
	and satisfies that $\|\nabla \tilde{f}(x,z)\|_\infty \leq 
	G_\infty,  
	\forall\ x \in \constraint, z \in \mathbb{R}^q$.
\end{assump}

\begin{theorem}[Stochastic Convex]
	\label{thm:stoch_convex}
	Consider QFW outlined in Algorithm~\ref{alg:convex_dist_fw}.
	Under \cref{assump_on_K,assump_on_f_stoch_convex},
	if we set $ p_i = 2^{i-1}, S_{i,1} =  \frac{\sigma^2 
	p_i^2}{ML^2D^2}$, $S_{i,k} = p_i/M = 2^{i-1}/M, 
	\forall\ i \ge 1, k \ge 2$, and $\eta_{i,k}=2/(p_i+k)={2}/(2^{i-1}+k)$, and 
	use the 
	$s_{1,i,1}=(\sqrt{\frac{dp_i^2}{M}})$-\mscheme,  
	$s_{2,i,1}=(\sqrt{dp_i^2})$-\mscheme for $k=1$, and 
	$s_{1,i,k}=(\sqrt{\frac{dp_i}{M}})$-\mscheme,  
	$s_{2,i,k}=(\sqrt{dp_i})$-\mscheme for $k \ge 2$
	as 
	$\Phi_{1,i,k}$ and $\Phi_{2,i,k}$ in \cref{alg:convex_dist_fw}, 
	then the output $x_{T+1} \in \constraint$  satisfies
	\begin{equation*}
	\expect[f(x_{T+1})]-f(x^*) \leq 
	 \frac{4D\sqrt{13L^2D^2+2G^2_\infty}+2LD^2}{T},  
	\end{equation*}
	where $x^*$ 
	is a minimizer of $f$ on $\constraint$. 
\end{theorem}  
\begin{corollary}\label{cor:convex_stoch}
To obtain an $\epsilon$-suboptimal solution, we need to run the QFW method outlined in \cref{alg:convex_dist_fw} for at most $\mathcal{O}(1/\epsilon)$ 
iterations. The SFO complexity per worker in this case is 
$\mathcal{O}(1/(M\epsilon^2))$.
\end{corollary}
\cref{cor:convex_stoch} shows SFO complexity per worker is linear in $1/M$,
which implies a speed-up for distributed settings. It also shows that the 
dependency of QFW's complexity on $\eps$ for convex settings is optimal.

\begin{remark}
In theory, the partitioning levels of 
quantization do depend 
on the number of iterations. Thus more transmission bits are required 
over 
the optimization procedure. But it will not render our QFW method 
communication-expensive. We set the partitioning levels conservatively to 
achieve the theoretical guarantees. However, as shown in 
the experiments (\cref{sec:experiments-qfw}), 
much smaller 
quantization levels (which are actually constants) are usually 
preferred in practice.
\end{remark}
	
\section{Non-Convex Optimization}

With slightly different parameters, QFW can be applied to 
non-convex settings as well. In \textit{unconstrained} non-convex optimization 
problems, the gradient norm $\|\nabla f\|$ is usually a 
good measure of convergence as $\|\nabla f\| \to 
0$ implies convergence to a stationary point. However, in 
the constrained setting we study the 
Frank-Wolfe Gap \citep{jaggi2013revisiting,lacoste2016convergence} defined as
\begin{equation*}
\label{eq:fw_gap}
\mathcal{G}(x) = \max_{v \in \constraint} \langle v-x, -\nabla f(x)\rangle.
\end{equation*}
For constrained optimization problem, if a point 
$x$ satisfies $\mathcal{G}(x)=0$, then it is a first-order stationary point. 
Also, by  definition, we have $\mathcal{G}(x) 
\geq 0, 
\forall\ x \in \constraint$. 
We first analyze the 
finite-sum setting and then the more general stochastic setting.

\subsection{Finite-Sum Setting}

To extend our results to the non-convex setting we first assume that the 
following condition is satisfied.

\begin{assump}
	\label{assump_on_f_i_nonconvex}
	The component functions $f_{m,i}$ are $L$-smooth on $\constraint$ and 
	uniformly bounded, \emph{i.e.},  
	$\sup_{x \in 
		\constraint}|f_{m,i}(x)| \leq M_0$. Further, $\|\nabla f_{m,i}(x)\|_\infty \leq G_\infty,
	\forall\ m \in [M], i \in [n], x, y \in \constraint$.
\end{assump}

\begin{theorem}[Finite-Sum 
Non-Convex]\label{thm:finite_nonconvex} 
		Under \cref{assump_on_K,assump_on_f_i_nonconvex},		
		if we set $ p_i=\sqrt{n}$, $\mathcal{S}_{i,1}^m = \{f_{m,j}: j \in 
		[n]\}$ (\emph{i.e.}, each worker $m$ samples 
		all its $n$ component functions), $S_{i,k} = \sqrt{n}/M\ \forall\ i 
		\ge1, k \ge 2, \eta_t=T^{-1/2}\ \forall\ t$, 
		and use the  
		$s_{1,i,1}=(\sqrt{\frac{Td}{M}})$-\mscheme,  
		$s_{2,i,1}=(\sqrt{Td})$-\mscheme for $k=1$, and 
		$s_{1,i,k}=(\frac{d^{1/2}n^{1/4}}{\sqrt{M}})$-\mscheme,  
		$s_{2,i,k}=(d^{1/2}n^{1/4})$-\mscheme for $k \ge 2$
		as 
		$\Phi_{1,i,k}$ and $\Phi_{2,i,k}$ in \cref{alg:convex_dist_fw}, 
		then the output $x_o \in \constraint$ satisfies
		\begin{equation*}
		\begin{split}
		\expect[\mathcal{G}(x_o)]
		&\leq 
		\frac{2M_0+D\sqrt{3L^2D^2+2G_\infty^2}+\frac{LD^2}{2}}{\sqrt{T}}.		
		\end{split}
		\end{equation*}
	\end{theorem}
\begin{corollary}\label{cor:nonconvex_finite}
\cref{alg:convex_dist_fw} finds an $\epsilon$-first-order stationary point 
after at most $\mathcal{O}(1/\epsilon^2)$ iterations. The IFO complexity per 
worker is $\mathcal{O}(\sqrt{N}/(\epsilon^2\sqrt{M}))$.
\end{corollary}
\cref{cor:nonconvex_finite} shows IFO complexity per worker is linear in 
$1/\sqrt{M}$, implicating that increasing the number of workers $M$ will 
decrease the IFO complexity per worker effectively, thus accelerate the 
optimization 
procedure.  

\subsection{Stochastic Setting}
For non-convex objective function, the stochastic optimization problem in 
\eqref{stochastic_problem} can be solved 
approximately by the QFW method described in \cref{alg:convex_dist_fw}. 
Specifically, the objective function $f(x) =  
\expect_{z\sim 
	P}[\tilde{f}(x,z)]$ can be approximated by a finite-sum problems with $B$ 
	samples where the samples $\{z_1,\dots,z_B\}$ are independently drawn
according to the probability distribution $P$. Thus we define the surrogate 
function $\hat{f}$
\begin{equation*}
\hat{f}(x) = \frac{1}{B} \sum_{i=1}^B \tilde{f}(x, z_i),
\end{equation*}
as the finite-sum approximation of $f(x)$. As a result, we can apply QFW on 
$\hat{f}$, thus optimize $f$ approximately. The 
algorithm is outlined in \cref{alg:nc_stoch_to_finitesum}.

\begin{algorithm}[t]
	\begin{algorithmic}[1]
		\Require constraint set $\constraint$, iteration 
		number $T$, No. of workers $M$, 
		initial 
		point $x_1 \in \constraint$
		\State Obtain $T$ independent samples of $z_i$, and define finite-sum 
		$\hat{f}(x) = \frac{1}{T} \sum_{i=1}^T \tilde{f}(x, z_i)$ 
		\State Apply \cref{alg:convex_dist_fw} on $\hat{f}$ with $N=T$ and all 
		other parameters being the same as in \cref{thm:finite_nonconvex}
		\Ensure $x_o$, 
		where $x_o$ is chosen from $\{x_1, x_2, 
		\cdots, x_T\}$ uniformly at random
	\end{algorithmic}
	\caption{\sncAlg (SNC-QFW)\label{alg:nc_stoch_to_finitesum}}
\end{algorithm}

In the non-convex setting, if we further assume that $\tilde{f}(x,z)$ is 
$G$-Lipschitz for all $z \in \mathbb{R}^q$, then we have the following lemma:
\begin{lemma}[Theorem 5 of \citep{reddi2016stochastic}]\label{lem:approx}
	 If we define 
	$\hat{\mathcal{G}}(x) = \max_{v \in \constraint} \langle v-x, -\nabla 
	\hat{f}(x) \rangle$, then
	$\expect[\mathcal{G}(x) - \hat{\mathcal{G}}(x)] \leq \frac{GD}{\sqrt{B}}.$ 
	Recall that $D$ is the diameter of $\constraint$ as defined in 
	\cref{assump_on_K}, $\mathcal{G}(x) = \max_{v \in \constraint} \langle v-x, 
	-\nabla f(x)\rangle$. 
	Thus for the output $x_o$, we have
	\begin{equation*}
	\expect[\mathcal{G}(x_o)] \leq \expect[\hat{\mathcal{G}}(x_o)] + 
	\frac{GD}{\sqrt{B}}.
	\end{equation*}	
\end{lemma} 

Baesd on \cref{thm:finite_nonconvex} and 
\cref{lem:approx}, we have the following theoretical guarantee for stochastic 
non-convex minimization.

\begin{theorem}[Stochastic Non-Convex]\label{thm:finite_to_stoch__nonconvex}
	Assuming that for all $z \in \mathbb{R}^q$, $\tilde{f}(x,z)$ is 
	$G$-Lipschitz, 
	$L$-smooth, 
	and satisfies $|\tilde{f}(x,z)| \leq M_0$
	for 
	all $x \in \constraint$. If we obtain $T$ independent samples of $z_i$, and 
	apply \cref{alg:convex_dist_fw} on $\hat{f}(x) = \frac{1}{T} \sum_{i=1}^T 
	\tilde{f}(x, 
	z_i)$ with $N=T, n=T/M$, and all the other parameters set the 
	same as in \cref{thm:finite_nonconvex}, then 
	after $T$ iterations, the output $x_o \in \constraint$  satisfies
	\begin{equation*}
	\expect[\mathcal{G}(x_o)] \leq 
	\frac{2M_0+D\sqrt{3L^2D^2+2G^2}+\frac{LD^2}{2}}{\sqrt{T}} + 
	\frac{GD}{\sqrt{T}}.
	\end{equation*}
\end{theorem}
We note that the algorithm finds an $\epsilon$-first-order stationary 
point with at most $\mathcal{O}(1/\epsilon^2)$ rounds. The SFO complexity per 
worker is 
$\mathcal{O}(\sqrt{N}/(\sqrt{M} 
\epsilon^2))=\mathcal{O}(\frac{1}{\epsilon^3\sqrt{M}})$.
Thus the SFO complexity per worker is linear in 
$1/\sqrt{M}$, which implicates that increasing the number of workers $M$ will 
decrease the SFO complexity per worker.

\section{Experiments}
\label{sec:experiments-qfw}

We evaluate the performance of algorithms by visualizing 
 their loss $f(x_t)$ 
  vs.\ the number of transmitted bits. The experiments were performed on 20 
  Intel Xeon E5-2660 cores 
and thus the number of workers is 20. For each curve in the figures below, we 
ran at least 50 repeated experiments, and the height of shaded regions 
represents two 
standard deviations. 

In our first setup, we consider a multinomial logistic regression problem. Consider the dataset $ \{ (x_i,y_i)  \}_{i=1}^N \subseteq \mathbb{R}^d\times 
\{1,\dots, C\}$ with $N$ samples that have $C$ different labels. We aim to find a model $w$ to classify these sample points under the condition that the solution has a small $\ell_1$-norm. Therefore, we aim to solve the following convex problem
\begin{align}\label{mul_log}
\min_{w} f(w) := &-\sum_{i=1}^N \sum_{c=1}^C 1\{y_i=c\} \log 
\frac{\exp(w_c^\top 
x_i)}{\sum_{j=1}^C \exp(w_j^\top x_i)},\nonumber \\
\text{subject to}\quad &{\|w\|_1 \le 1}.
\end{align}
In our experiments, we use the MNIST dataset and assume that each worker stores 
$ 
3000$ images. Therefore, the overall number of samples in the training set 
is $N=60000$. 

In our second setup, our goal is to minimize the loss of a three-layer neural network under some conditions on the norm of the solution. 
Before stating the problem precisely, let us define the log-loss function as $
h(y,p)\triangleq -\sum_{c=1}^C 1\{ y=c \}\log p_c $ for $ y\in 
\{1,\dots,C\} $ and a $ C $-dimensional probability vector $ p:= (p_1, \cdots, p_C) $. We aim to solve the following non-convex problem 
\begin{align}\label{nn}
\min_{W_1, W_2} &\sum_{i=1}^N h(y_i, 
\phi(W_2\relu(W_1x_i\!+\!b)\!+\!b_2) ), \nonumber\\
\text{subject to }&  \|W_i\|_1\le a_1, 
\|b_i\|_1\le a_2,  
\end{align}
where $\relu(x)\triangleq \max\{0,x\} $ is the ReLU function and $\phi$ is the 
softmax function. The imposed $\ell_1$ constraint on the weights leads to a 
sparse network. We further remark that Frank-Wolfe methods are suitable for 
training a neural network subject to an $\ell_1$ constraint as they are 
equivalent to a dropout 
regularization \citep{ravi2018constrained}. 
 In our setup, the size 
of matrices $ W_1 $ and $ W_2 $ are $ 784\times 50 $ and $ 50\times 10 $, 
respectively, and the constraints parameters are $a_1=a_2=5$. 

For all of the considered settings,  we vary the quantization level, use the 
$s_1$-partition encoding scheme when workers send encoded tensors to the master 
and use the $ s_2 $-partition encoding scheme when the master broadcasts 
encoded 
tensors to the workers ($s_i=uq$ indicates FW without quantization and $  
s_i=thm$ indicates \qfw with the quantization level recommended by our 
theorems, where $ i=1,2 $). We 
also propose the federated learning approach \fl, an effective heuristic based 
on \qfw, where each worker performs its local Frank-Wolfe update 
autonomously 
without communicating with each other and synchronizes the model only at the 
end of each round.  
This method may not enjoy the strong theoretical guarantees of \qfw 
and we observe in our experiments that it is even prone to divergence. 
In stochastic minimization, each worker samples $ 1000 $ images uniformly at 
random and without replacement.


In \cref{fig:convex}, we observe the performance of \fl, FW without 
quantization, and different variants of  \qfw for solving the multinomial 
logistic regression problem in \eqref{mul_log}. The stochastic minimization is 
presented in \cref{fig:sc} and the finite-sum minimization is shown in 
\cref{fig:fc}. We observe that 
 \qfw with \mscheme ($ s_1=1,s_2=3 $) has the best performance in terms of 
 the amount of transmitted bits. Specifically, \qfw with 
 \mscheme ($ s_1=1,s_2=3 $) requires $ 3\times 10^8 $ bits to hit the 
 lowest loss in 
\cref{fig:sc,fig:fc}, while \fl with the same level of quantization only 
achieves a 
suboptimal loss (approximately $ 2.24 $) with the same amount of communication. 
 Furthermore,  FW without quantization requires more than $ 3.8\times 10^9
$ bits to reach the same error, \emph{i.e.}, quantization reduces communication 
load by 
at least an order of magnitude.

\begin{figure*}[t]
	\centering 
	\begin{subfigure}[t!]{0.37\linewidth}
		\includegraphics[height=1.6in]{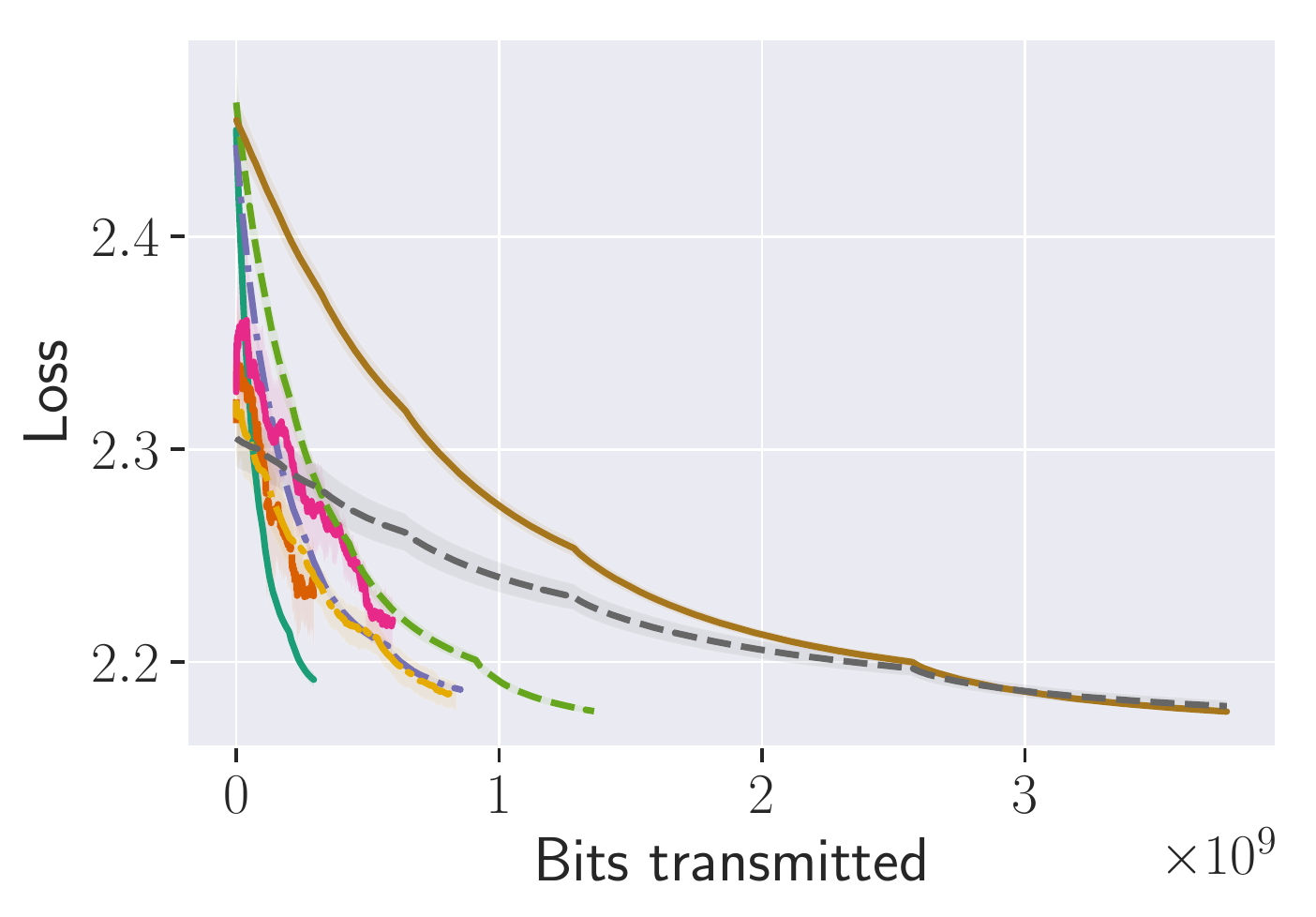} 
		\caption{Stochastic convex minimization}
		\label{fig:sc}
	\end{subfigure}
	\begin{subfigure}[t!]{0.19\linewidth}
		\includegraphics[width=\linewidth]{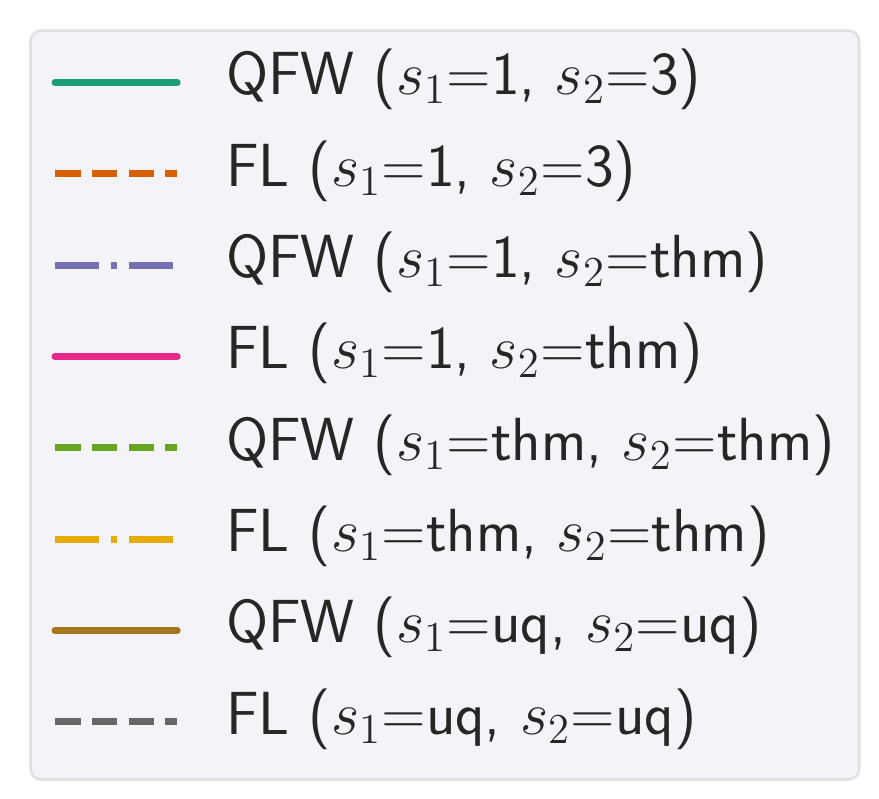} 
	\end{subfigure}
	\begin{subfigure}[t!]{0.37\linewidth}
		\includegraphics[height=1.6in]{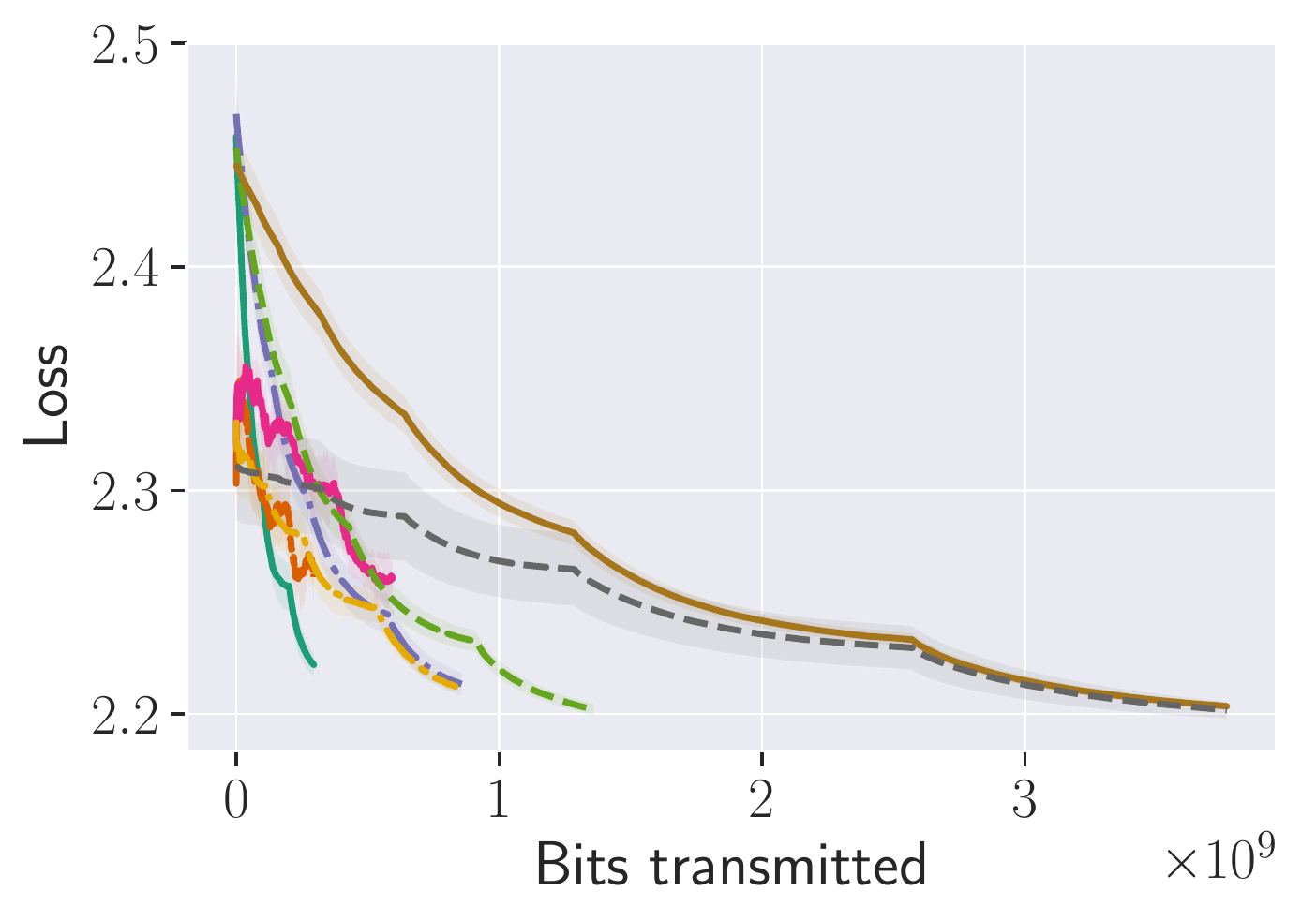} 
		\caption{Finite-sum convex minimization}
		\label{fig:fc}
	\end{subfigure} 
	\caption{Comparison in terms of the loss versus the number of transmitted 
		bits for a multinomial logistic 
		regression problem. The best performance belongs to \qfw with 
		\mscheme 
		($ s_1=1,s_2=3 $), and FW without quantization has the worst 
		performance.
		\label{fig:convex}
	}
\end{figure*}

\cref{fig:nonconvex} demonstrates the performance of \fl, FW without 
quantization, and different variants of  \qfw for solving the three-layer 
neural network in \eqref{nn}. Again we show the stochastic minimization on the 
left (\cref{fig:sn}) and the finite-sum minimization on the right 
(\cref{fig:fn}). 
We observe four divergent curves of the federated 
learning method \fl ($ s_1=3,s_2=1 $; $ s_1=thm,s_2=1 $; $ s_1=1,s_2=thm $; 
and $ s_1=1,s_2=uq $), while all \qfw curves converge. This observation is 
in accordance with the fact that \fl 
has no theoretical guarantee, in contrast to the proposed \qfw method. FW 
without quantization consumes approximately $ 6.5\times 10^9 $ bits to achieve 
the lowest loss. Its amount of communication is twice that of \qfw with 
the quantization levels recommended by 
\cref{thm:finite_nonconvex,thm:finite_to_stoch__nonconvex}.

\begin{figure*}[t]
	\centering
	\begin{subfigure}[t!]{0.36\linewidth}
		\includegraphics[height=1.5in]{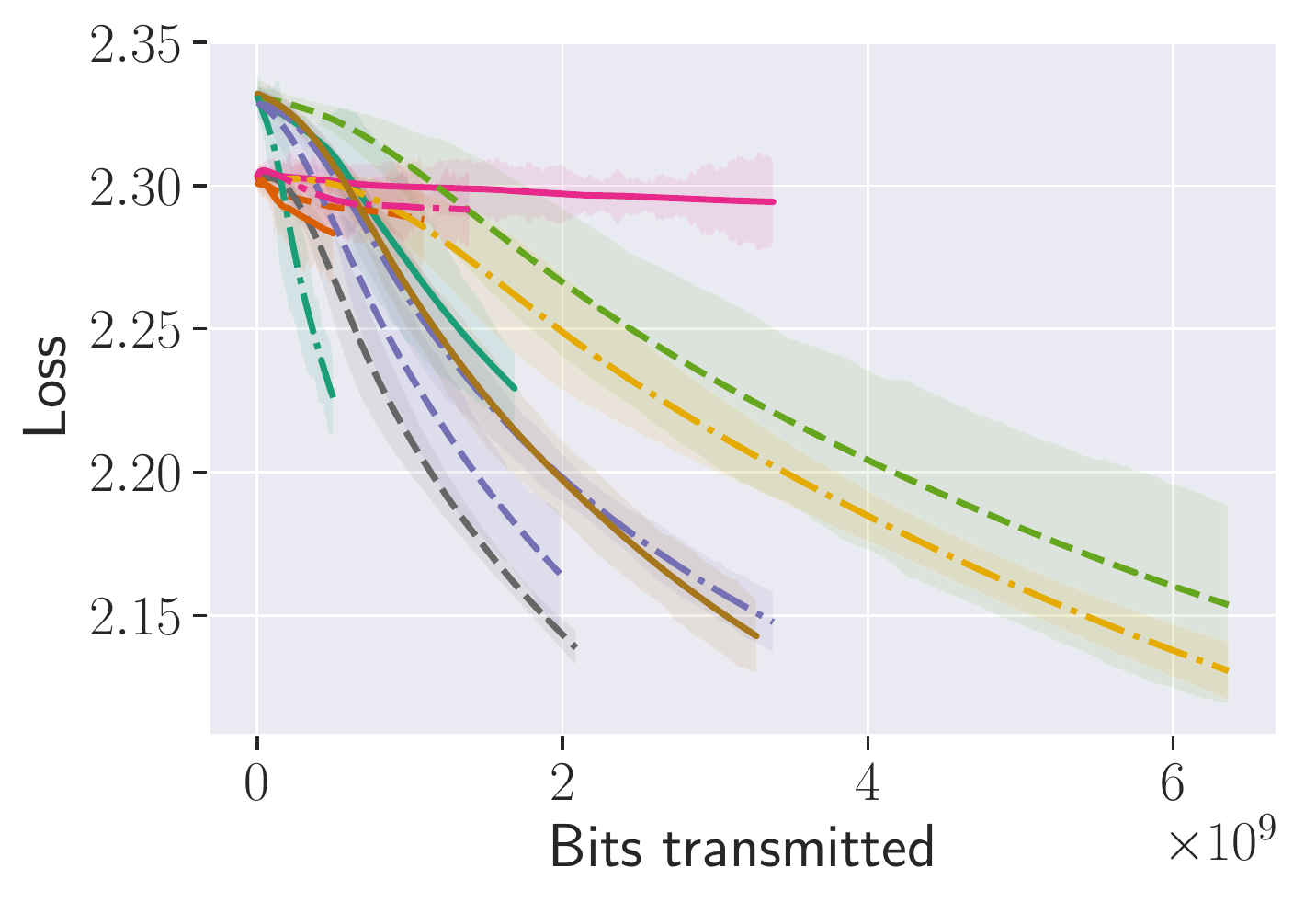} 
		\caption{Stochastic non-convex minimization}
		\label{fig:sn}
	\end{subfigure}
	\begin{subfigure}[t!]{0.19\linewidth}
		\includegraphics[width=\linewidth]{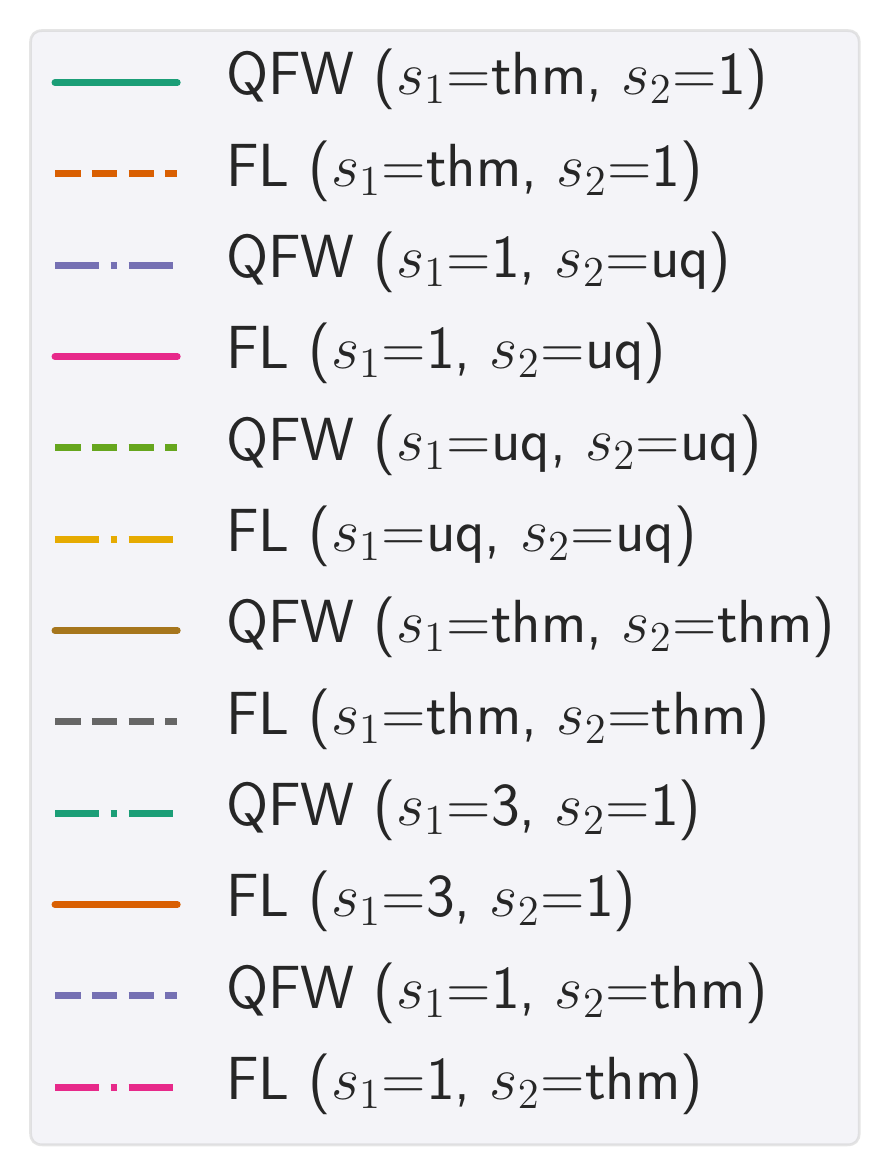} 
	\end{subfigure}
	\begin{subfigure}[t!]{0.36\linewidth}
		\includegraphics[height=1.5in]{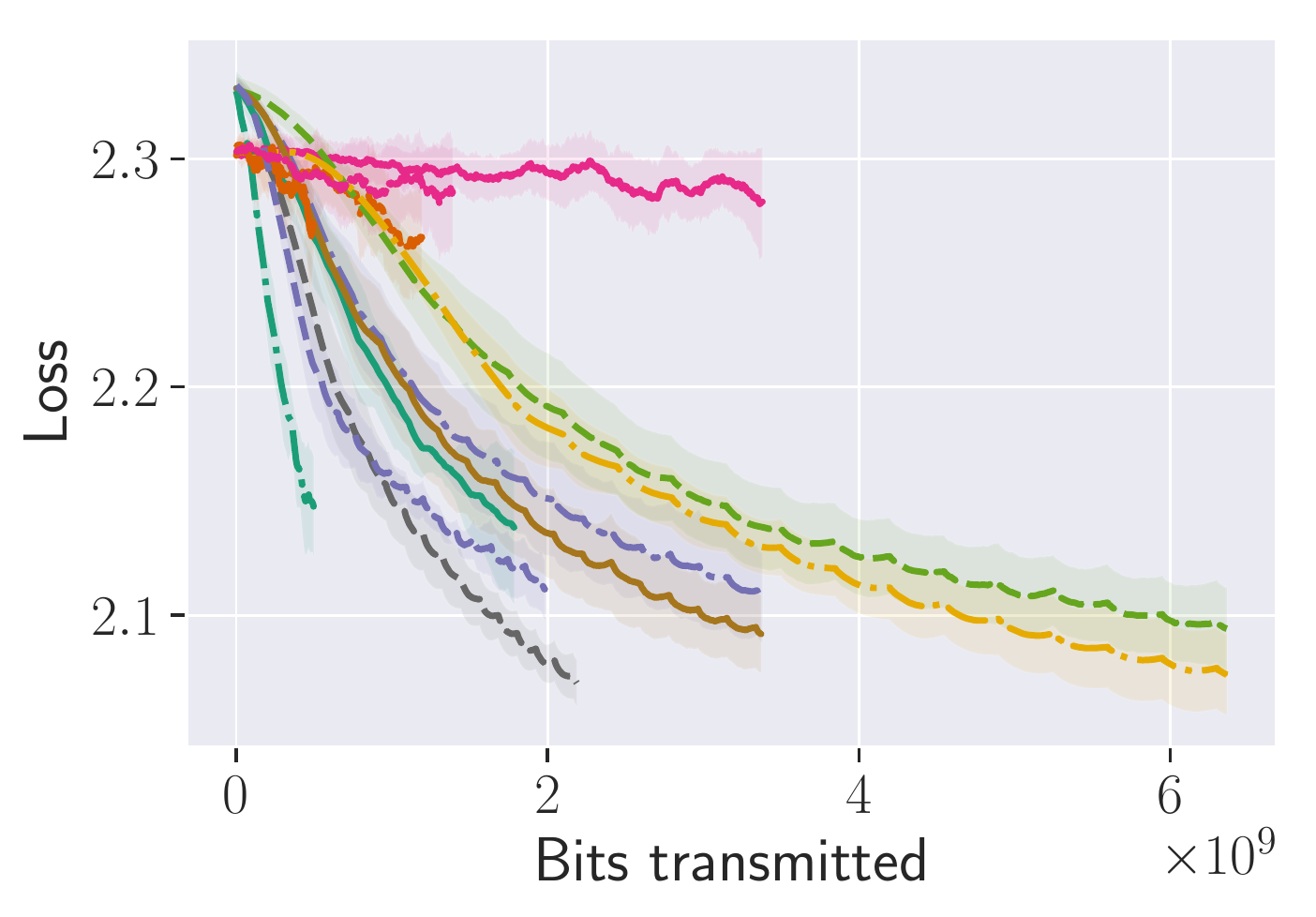} 
		\caption{Finite-sum non-convex minimization}
		\label{fig:fn}
	\end{subfigure} 
	\caption{Comparison of algorithms in terms of the loss versus the number of 
		transmitted bits for a three-layer neural 
		network. FW without quantization ($s=uq$) significantly underperforms 
		the 
		quantized FW methods. We observe four divergent curves of the federated 
		learning method \fl ($ s_1=3,s_2=1 $; $ s_1=thm,s_2=1 $; $ 
		s_1=1,s_2=thm $; 
		and $ s_1=1,s_2=uq $).
		\label{fig:nonconvex}
	}
\end{figure*}

In both convex and non-convex setups, the theoretically guaranteed quantization 
levels recommended by our theorems may be conservative. In fact, a 
\mscheme 
with partitions much fewer than our theorems recommend achieves a similar loss 
level and saves even more communication bits. For example, \qfw with $ 
s_1=1,s_2=3 $ and $ s_1=1,s_2=thm $ exhibits a higher communication efficiency 
than \qfw with $ s_1=thm,s_2=thm $ in \cref{fig:sc,fig:fc}.

\section{Conclusion}
In this chapter, we developed \AlgQFW (\qfw), the first general-purpose 
projection-free 
and communication-efficient 
framework for constrained optimization. Along with proposing various 
quantization schemes, \qfw can address both convex and non-convex 
optimization settings in  stochastic and finite-sum cases. We provided 
theoretical guarantees on the convergence rate of \qfw and validated its 
efficiency  empirically 
on training multinomial logistic regression and neural networks. Our 
theoretical results highlighted the importance of variance reduction techniques 
to stabalize FW and achieve a sweet trade-off between the 
communication complexity and convergence rate in distributed settings. We also 
note that it might be possible to design simpler Quantized FW methods based on 
the new developments \citep{zhang2020one}.

\section{Proofs}\label{sec:proof-qfw}
\subsection{Proof of Lemma 5}\label{app:lem_var_multi}
\begin{proof}
	For any given vector $g\in \reals^d$, the ratio ${|g_i|}/{\|g\|_\infty}$ 
	lies in an 
	interval of the form $[l_i/s,{(l_i+1)}/{s}]$ where $l_i\in \{0,1,\dots, 
	s-1\}$. Hence, for that specific $l_i$, the following inequalities 
	\begin{align}\label{helpppp_1}
	\frac{l_i}{s} \leq \frac{|g_i|}{\|g\|_\infty} \leq \frac{l_i+1}{s}
	\end{align}
	are satisfied. Moreover, based on the probability distribution of $b_i$ we 
	know that  
	\begin{align}\label{helpppp_2}
	\frac{l_i}{s} \leq b_i \leq \frac{l_i+1}{s}.
	\end{align}
	
	Therefore, based on the inequalities in \eqref{helpppp_1} and 
	\eqref{helpppp_2} we can write
	\begin{align}\label{helpppp_3}
	-\frac{1}{s} \leq \frac{|g_i|}{\|g\|_\infty}- b_i \leq \frac{1}{s}.
	\end{align}
	
	Hence, we can show that the variance of $s$-\mscheme is upper bounded by
	\begin{align*}
	\var{\phi^{\prime}(g)|g}
	&= \expect[\|\phi^{\prime}(g)-g\|^2|g] \nonumber\\
	&= \sum_{i=1}^d \expect[(g_i - \sign(g_i)b_i\|g\|_\infty)^2|g] \nonumber\\
	&= \sum_{i=1}^d \expect[(|g_i| - b_i\|g\|_\infty)^2|g] \nonumber\\
	&= \sum_{i=1}^d \|g\|_\infty^2 
	\expect\left[\left(\frac{|g_i|}{\|g\|_\infty} - b_i\right)^2\mid g\right] 
	\nonumber\\
	&\leq \frac{d}{s^2}\|g\|_\infty^2,
	\end{align*}
	where the inequality holds due to \eqref{helpppp_3}.
\end{proof}

\subsection{Proof of Theorem 5 and Corollary 4.3.1}
The key to the proofs of \cref{thm:finite_convex} is to upper bound the 
difference between the true gradient $\nabla f(x_t) = \nabla f(x_{i,k})$ and 
the estimated gradient 
$\bar{g}_{i,k}$. Intuitively, if the error is 
small enough, then we can approximate $\nabla f(x_{i,k})$ by $\bar{g}_{i,k}$. 
Thus 
the algorithm fed with the estimated gradient $\bar{g}_{i,k}$ will still 
converge.

So we first address the bound of $\|\nabla f(x_{i,k}) - \bar{g}_{i,k}\|$, which 
is 
resolved in the following lemma.
\begin{lemma}\label{lem:convex_finite_error}
Under the condition of \cref{thm:finite_convex}, we have
\begin{equation*}
\expect[\| \nabla f(x_{i,k}) - \bar{g}_{i,k}\|^2] \leq 
\frac{2(G^2_\infty+2L^2D^2+4L_\infty^2D^2 )}{p_i^2}.
\end{equation*}
\end{lemma}

\begin{proof}
We first define a few auxiliary variables. On each worker $m$, we 
define the average function of its $n$ component functions as $f^{(m)}(x) = 
\frac{\sum_{j=1}^n f_{m,j}(x)}{n}$, so $f(x) = 
\frac{\sum_{m=1}^M f^{(m)}(x)}{M}$. We also define 
\begin{equation*}
g_{i,k}^{(m)} =
\begin{cases}
g_{i,k}^{m}& \quad k = 1, \\
g_{i,k-1}^{(m)} + g_{i,k}^{m} = \sum_{j=1}^k g_{i,j}^m& \quad k \ge 2,
\end{cases}
\end{equation*}
where $g_{i,k}^{m}$ is defined in \cref{alg:convex_dist_fw}. Then 
$g_{i,k}^{(m)}$ is an unbiased 
estimator of $\nabla f^{(m)}(x_{i,k})$. We define 
the average of $g_{i,k}^{(m)}$ as
\begin{equation*}
g_{i,k} = \frac{\sum_{m=1}^M g_{i,k}^{(m)}}{M}.
\end{equation*}
We also define $\mathcal{F}_{i,k}$ to be the $\sigma$-field generated by 
all the randomness before round $(i,k)$, \emph{i.e,} round $t = 
\sum_{j=1}^{i-1}p_j +k$. We note that given 
$\mathcal{F}_{i,k}$, $x_{i,k}$ is actually determined, and we can verify that 
$\expect[g_{i,k}|\mathcal{F}_{i,k}]=\nabla f(x_{i,k})$, and 
$\expect[\bar{g}_{i,k}|\mathcal{F}_{i,k},g_{i,k}]=g_{i,k}, \forall\ (i,k)$. 
Here, 
with abuse of notation, $\expect[\cdot|g_{i,k}]$ is the conditional expectation 
given not only the value of $g_{i,k}$, but also the sampled gradients $\nabla 
f_{m,j}(x_{i,k}), \nabla f_{m,j}(x_{i,k-1})$(if defined) for all $j \in 
\mathcal{S}_{i,k}^m, m 
\in [M]$.

Then by law of total expectation, we have
\begin{equation}\label{eq:convex_finite_error1}
\begin{split}
\expect[\| \nabla f(x_{i,k}) - \bar{g}_{i,k}\|^2]
={}&\expect[\expect[\| \nabla f(x_{i,k}) - 
\bar{g}_{i,k}\|^2|\mathcal{F}_{i,k}]] \\
={}& \expect[\expect[\| \nabla f(x_{i,k}) - g_{i,k}+ 
g_{i,k}-\bar{g}_{i,k}\|^2|\mathcal{F}_{i,k}]]    \\
={}& \expect[\expect[\| \nabla f(x_{i,k}) - g_{i,k}\|^2|\mathcal{F}_{t-1}]] + 
\expect[\expect[\| g_{i,k} -\bar{g}_{i,k}\|^2|\mathcal{F}_{i,k}]] \\
&\quad + 
2\expect[\expect[\langle \nabla f(x_{i,k}) - g_{i,k}, g_{i,k} -\bar{g}_{i,k} 
\rangle 
|\mathcal{F}_{i,k}]]\\
={}& \expect[\| \nabla f(x_{i,k}) - g_{i,k}\|^2] + \expect[\| g_{i,k} 
-\bar{g}_{i,k}\|^2],
\end{split}    
\end{equation}
where the last equation holds since 
\begin{equation*}
\begin{split}
\expect[\langle \nabla f(x_{i,k}) - g_{i,k}, g_{i,k} -\bar{g}_{i,k} 
\rangle|\mathcal{F}_{i,k}] 
={}& \expect[\expect[\langle \nabla f(x_{i,k}) - g_{i,k}, g_{i,k} 
-\bar{g}_{i,k} 
\rangle|\mathcal{F}_{i,k},g_{i,k}]|\mathcal{F}_{i,k}] \\
={}& \expect[\langle \nabla f(x_{i,k}) - g_{i,k}, \expect[g_{i,k} 
-\bar{g}_{i,k}|\mathcal{F}_{i,k},g_{i,k}] \rangle|\mathcal{F}_{i,k}] \\
={}&0.
\end{split}    
\end{equation*}

Now we turn to bound $\expect[\| \nabla f(x_{i,k}) - g_{i,k}\|^2]$. In fact, we 
have
\begin{equation}\label{eq:finite_convex_1}
\begin{split}
\expect[\| \nabla f(x_{i,k}) - g_{i,k}\|^2] &= \expect[\| \frac{\sum_{m=1}^M 
\nabla 
f^{(m)}(x_{i,k})}{M} - \frac{\sum_{m=1}^M g_{i,k}^{(m)}}{M}\|^2] \\
&= \frac{\sum_{m=1}^M\expect[\|\nabla f^{(m)}(x_{i,k}) - 
g_{i,k}^{(m)}\|^2]}{M^2}.
\end{split}
\end{equation}

For $k\ge 2$, we have
\begin{equation*}
\begin{split}
&\expect[\| \nabla f^{(m)}(x_{i,k}) - g_{i,k}^{(m)}\|^2] \\
={}& \expect[\expect[\| [\nabla f^{(m)}(x_{i,k}) 
-\nabla f^{(m)}(x_{i,k-1})] - g_{i,k}^m \|^2 |\mathcal{F}_{i,k}]]   + 
\expect[\expect[\|\nabla 
f^{(m)}(x_{i,k-1})-g_{i,k-1}^{(m)}\|^2|\mathcal{F}_{i,k}]]  \\	
={}& \expect[\var{g_{i,k}^m|\mathcal{F}_{i,k}}] + \expect[\|\nabla 
f^{(m)}(x_{i,k-1})-g_{i,k-1}^{(m)}\|^2]  \\
={}&\expect[\var{\frac{\sum_{j \in \mathcal{S}_{i,k}^m} \nabla 
f_{j}(x_{i,k})-\nabla 
	f_{j}(x_{i,k-1})}{S_{i,k}}|\mathcal{F}_{i,k}}] + \expect[\|\nabla 
f^{(m)}(x_{i,k-1})-g_{i,k-1}^{(m)}\|^2]\\
={} & \expect[\frac{\sum_{j \in \mathcal{S}_{i,k}^m} \var{\nabla 
	f_{j}(x_{i,k})-\nabla 
	f_{j}(x_{i,k-1})|\mathcal{F}_{i,k}}}{[S_{i,k}]^2}] + 
\expect[\|\nabla 
f^{(m)}(x_{i,k-1})-g_{i,k-1}^{(m)}\|^2]\\
\leq{} & \expect[\frac{\sum_{j\in \mathcal{S}_{i,k}^m} \expect[\|\nabla 
	f_{j}(x_{i,k})-\nabla 
	f_{j}(x_{i,k-1})\|^2|\mathcal{F}_{i,k}]}{[S_{i,k}]^2})] + 
\expect[\|\nabla 
f^{(m)}(x_{i,k-1})-g_{i,k-1}^{(m)}\|^2]\\
\leq{} & \frac{1}{S_{i,k}}(LD\eta_{i,k-1})^2+ \expect[\|\nabla 
f(x_{i,k-1})-g_{i,k-1}\|^2]\\
=& \frac{L^2D^2\eta_{i,k-1}^2}{S_{i,k}}+ \expect[\|\nabla 
f^{(m)}(x_{i,k-1})-g_{i,k-1}^{(m)}\|^2].
\end{split}
\end{equation*}

For $k=1$, we have $g_{i,1}^{(m)} = \nabla f^{(m)}(x_{i,1})$. So
\begin{equation*}
\expect[\| \nabla f^{(m)}(x_{i,k}) - g_{i,k}^{(m)}\|^2] \leq 
L^2D^2\sum_{j=2}^{k}\frac{\eta_{i,j-1}^2}{S_{i,j}} = 
\frac{L^2D^2M}{p_i}\sum_{j=2}^{k}\eta_{i,j-1}^2.
\end{equation*}

Since
\begin{equation*}
\sum_{j=2}^{k}\eta_{i,j-1}^2 = \sum_{j=2}^{k} \frac{4}{(p_i+j-1)^2} \leq 
\sum_{j=2}^{k} \frac{4}{p_i^2} \leq \frac{4}{p_i},
\end{equation*}
we have 
\begin{equation*}
\expect[\| \nabla f^{(m)}(x_{i,k}) - g_{i,k}^{(m)}\|^2] \leq 
\frac{4ML^2D^2}{p_i^2}.  
\end{equation*}

Combine with \cref{eq:finite_convex_1}, we have
\begin{equation}\label{eq:convex_finite_error2}
\expect[\| \nabla f(x_{i,k}) - g_{i,k}\|^2] \leq \frac{M\cdot 
4ML^2D^2}{M^2\cdot 
p_i^2} = 
\frac{4L^2D^2}{p_i^2}.
\end{equation}

Now we only need to bound $\expect[\| g_{i,k} 
-\bar{g}_{i,k}\|^2]$. For $k\ge 2$, we have
\begin{equation*}
\begin{split}
&\expect[\| g_{i,k} - \bar{g}_{i,k} \|^2] \\
={}& \expect[\expect[\| \frac{\sum_{m=1}^M g_{i,k}^m}{M} + g_{i,k-1}-
\phi_{2,i,k}^\prime(\tilde{g}_{i,k})-\bar{g}_{i,k-1}\|^2|\mathcal{F}_{i,k},g_{i,k}]]\\
={}& \expect[\expect[\| \frac{\sum_{m=1}^M g_{i,k}^m}{M} - 
\phi_{2,i,k}^\prime(\tilde{g}_{i,k})\|^2|\mathcal{F}_{i,k},g_{i,k}]] + 
\expect[\|g_{i,k-1} 
-\bar{g}_{i,k-1}\|^2] \\
& + 2\expect[\expect[\frac{\sum_{m=1}^M g_{i,k}^m}{M} 
-\phi_{2,i,k}^\prime(\tilde{g}_{i,k}), g_{i,k-1} -\bar{g}_{i,k-1} 
\rangle|\mathcal{F}_{i,k},g_{i,k-1}]].\\
\end{split}    
\end{equation*}

Moreover
\begin{equation*}
\begin{split}
\expect[\phi_{2,{i,k}}^\prime(\tilde{g}_{i,k})|\mathcal{F}_{i,k},g_{i,k}] 
={}& 
\expect[\tilde{g}_{i,k}|\mathcal{F}_{i,k},g_{i,k}]\\
={}&\expect[\sum_{m=1}^M 
\phi_{1,{i,k}}^\prime (g_{i,k}^m)/M |\mathcal{F}_{i,k},g_{i,k}] \\
={}& \frac{\sum_{m=1}^M g_{i,k}^m}{M},
\end{split}
\end{equation*}
and
\begin{equation*}
\begin{split}
& \expect[\expect[\| \frac{\sum_{m=1}^M g_{i,k}^m}{M} 
-\phi_{2,i,k}^\prime(\tilde{g}_{i,k})\|^2|\mathcal{F}_{i,k},g_{i,k}]] \\    
={}& \expect[\expect[\| \frac{\sum_{m=1}^M g_{i,k}^m}{M}
-\tilde{g}_{i,k} 
+ \tilde{g}_{i,k} 
-\phi_{2,{i,k}}^\prime(\tilde{g}_{i,k})\|^2|\mathcal{F}_{i,k},g_{i,k}]] \\
={}& \expect[\expect[\| \frac{\sum_{m=1}^M g_{i,k}^m}{M} -\sum_{m=1}^M 
\phi_{1,i,k}^\prime (g_{i,k}^m)/M \|^2 |\mathcal{F}_{i,k},g_{i,k}]] + 
\expect[\expect[\|\tilde{g}_{i,k}-\phi_{2,{i,k}}^\prime(\tilde{g}_{i,k})\|^2|
\mathcal{F}_{i,k},g_{i,k},\tilde{g}_{i,k}]]
\\
\leq{} & \frac{1}{M}\frac{d}{s_{1,{i,k}}^2}(\eta_{i,k-1}L D)^2 + 
\frac{d}{s_{2,{i,k}}^2}(\eta_{i,k-1}L D)^2\\ 
={}& 
\frac{\eta_{i,k-1}^2dL^2D^2}{Ms_{1,{i,k}}^2}+\frac{\eta_{i,k-1}^2dL^2D^2}{s_{2,{i,k}}^2},
\end{split}    
\end{equation*}
where in the inequality, we apply \cref{lem:var_multi} with 
$\|g_{i,k}^m \|_\infty = \|\nabla f_{\mathcal{S}_{i,k}^m}(x_{i,k}) - \nabla 
f_{\mathcal{S}_{i,k}^m}(x_{i,k-1})\|_\infty \leq \|\nabla 
f_{\mathcal{S}_{i,k}^m}(x_{i,k}) - \nabla 
f_{\mathcal{S}_{i,k}^m}(x_{i,k-1})\|_2 \leq\eta_{i,k-1}L D$
and 
$\|\tilde{g}_{i,k}\|_\infty 
= \|\sum_{m=1}^M \phi_{1,{i,k}}^\prime (g_{i,k}^m)/M\|_\infty \leq 
\eta_{i,k-1}L D$.
Now for $k \ge 2$ we have,
\begin{equation*}
\expect[\| g_{i,k} -\bar{g}_{i,k}\|^2] \leq 
\frac{\eta_{i,k-1}^2dL^2D^2}{Ms_{1,{i,k}}^2}+
\frac{\eta_{i,k-1}^2dL^2D^2}{s_{2,{i,k}}^2}+ 
\expect[\|g_{i,k-1} -\bar{g}_{i,k-1}\|^2].
\end{equation*}

If $k=1$, we have
\begin{equation*}
\begin{split}
\expect[\| g_{i,k} -\bar{g}_{i,k}\|^2] 
={}& \expect[\|\nabla f(x_{i,k}) -\tilde{g}_{i,k}+\tilde{g}_{i,k}- 
\phi_{2,{i,k}}^\prime(\tilde{g}_{i,k})\|^2] \\
={}& \expect[\expect[\|\nabla f(x_{i,k}) - \frac{\sum_{m=1}^M 
	\phi_{1,{i,k}}^\prime(\nabla f^{(m)}(x_{i,k}))}{M}\|^2 
	|\mathcal{F}_{i,k},g_{i,k}]] \\
&\quad + \expect[\expect[\| 
\tilde{g}_{i,k}- 
\phi_{2,{i,k}}^\prime(\tilde{g}_{i,k})\|^2|\mathcal{F}_{i,k},g_{i,k},\tilde{g}_{i,k}]]
 \\
\leq{}& \frac{1}{M^2} \expect[\sum_{m=1}^M\expect[\|\nabla f^{(m)}(x_{i,k}) -  
\phi_{1,t}^\prime(\nabla f^{(m)}(x_{i,k}))\|^2 |\mathcal{F}_{i,k},g_{i,k}]] + 
\frac{d}{s_{2,{i,k}}^2}G_\infty^2 \\
\leq{} & \frac{dG_\infty^2}{Ms_{1,{i,k}}^2}+\frac{dG_\infty^2}{s_{2,{i,k}}^2},
\end{split}    
\end{equation*}
where in the inequality, we apply \cref{lem:var_multi} with $\| 
\nabla f^{(m)}(x_k) \|_\infty \leq G_\infty$ and $\|\tilde{g}_{i,k}\|_\infty 
= \|\frac{\sum_{m=1}^M \phi_{1,{i,k}}^\prime(\nabla 
	f^{(m)}(x_k))}{M}\|_\infty \leq G_\infty$. Then we have
\begin{equation}\label{eq:convex_finite_error3} 
\begin{split}
\expect[\| g_{i,k} -\bar{g}_{i,k}\|^2] &\leq 
\sum_{j=2}^k\frac{\eta_{i,j-1}^2dL^2D^2}{Ms_{1,{i,j}}^2}+
\sum_{j=2}^k\frac{\eta_{i,j-1}^2dL^2D^2}{s_{2,{i,j}}^2} + 
\frac{dG_\infty^2}{Ms_{1,i,1}^2}+\frac{dG_\infty^2}{s_{2,i,1}^2} \\
&\leq \frac{dL^2D^2}{Ms_{1,i}^2}\sum_{j=2}^k\eta_{i,j-1}^2 + 
\frac{dL^2D^2}{s_{2,i}^2}\sum_{j=2}^k\eta_{i,j-1}^2 +
\frac{dG_\infty^2}{Ms_{1,i,1}^2}+\frac{dG_\infty^2}{s_{2,i,1}^2} \\
&\leq \frac{dL^2D^2}{M\frac{p_id}{M}}\frac{4}{p_i} + 
\frac{dL^2D^2}{p_id}\frac{4}{p_i} + 
\frac{dG_\infty^2}{M\frac{dp_i^2}{M}}+\frac{dG_\infty^2}{dp_i^2} \\
&= \frac{2G^2_\infty+8L^2D^2}{p_i^2}.
\end{split}  	
\end{equation}

Now combine 
\cref{eq:convex_finite_error1,eq:convex_finite_error2,eq:convex_finite_error3},
we have

\begin{align*}
\expect[\| \nabla f(x_{i,k}) - \bar{g}_{i,k}\|^2] \leq 
\frac{2(G^2_\infty+6L^2D^2)}{p_i^2} \triangleq 
\frac{C_1^2}{p_i^2}.
\end{align*}
\end{proof}

Now we turn to prove \cref{thm:finite_convex}. First, since $x_{t+1} = 
(1-\eta_{i,k})x_t + \eta_{i,k} v_{i,k}$ is a convex combination 
of $x_t, v_{i,k}$, and $x_1 \in \constraint, v_{i,k} \in \constraint, \forall\ 
t$, we can prove $x_t \in \constraint, \forall\ t $ by induction. So $x_{T+1} 
\in \constraint$. Then we need the following lemma.
\begin{lemma}[Proof of Theorem 1 in 
\citep{yurtsever2019conditional}]\label{lem:icml_thm1}
Consider \cref{alg:convex_dist_fw}, under the conditions of 
\cref{thm:finite_convex}, we have
\begin{equation*}
\expect[f_{i,k+1}] - f(x^*) \leq (1-\eta_{i,k}) (\expect[f(x_{i,k})-f(x^*)]) + 
\eta_{i,k}D\expect[\| \nabla f(x_i,k) - \bar{g}_{i,k}\|] + 
\eta_{i,k}^2\frac{LD^2}{2}.
\end{equation*}
Moreover, by analyzing the telescopic sum of the inequality over $(i,k)$, we 
have
\begin{equation*}
\expect[f(x_{i,k+1})] - f(x^*) \leq \sum_{(\tau,j)}\left( 
\eta_{\tau,j}D\expect[\| \nabla f(x_{\tau,j}) - \bar{g}_{\tau,j}\|] + 
\eta_{\tau,j}^2\frac{LD^2}{2} \right)\frac{(p_\tau +j-2)(p_\tau 
+j-1)}{(p_i+k-1)(p_i+k)}.
\end{equation*}
\end{lemma}

By \cref{lem:convex_finite_error} and Jensen's inequality, we have
\begin{equation*}
\expect[\| \nabla f(x_i,k) - \bar{g}_{i,k}\|] \leq \sqrt{\expect[\| \nabla 
f(x_i,k) - \bar{g}_{i,k}\|^2]} \leq 
\frac{C_1}{p_i}. 
\end{equation*}

So 
\begin{equation*}
\begin{split}
 &\sum_{(\tau,j)} \eta_{\tau,j}D\expect[\| \nabla f(x_i,k) - \bar{g}_{i,k}\|] 
 \frac{(p_\tau +j-2)(p_\tau 
 	+j-1)}{(p_i+k-1)(p_i+k)} \\
 \leq & \sum_{(\tau,j)} \frac{2}{p_\tau+j}D\frac{C_1}{p_\tau} 
 \frac{(p_\tau +j-2)(p_\tau +j-1)}{(p_i+k-1)(p_i+k)} \\
\leq & \frac{4C_1D}{(p_i+k-1)(p_i+k)}\sum_{(\tau,j)}1 \\
\leq & \frac{4C_1D}{p_i+k}.
\end{split}
\end{equation*}

We also have
\begin{equation*}
\begin{split}
\sum_{(\tau,j)}\eta_{\tau,j}^2\frac{LD^2}{2}\frac{(p_\tau +j-2)(p_\tau 
	+j-1)}{(p_i+k-1)(p_i+k)}&= \sum_{(\tau,j)} 
	\frac{4}{(p_\tau+j)^2}\frac{LD^2}{2}\frac{(p_\tau 
+j-2)(p_\tau +j-1)}{(p_i+k-1)(p_i+k)}\\
&\leq \frac{2LD^2}{(p_i+k-1)(p_i+k)}\sum_{(\tau,j)}1 \\
&\leq \frac{2LD^2}{p_i+k}.
\end{split}
\end{equation*}

Thus by \cref{lem:icml_thm1}, we have
\begin{equation*}
\expect[f(x_{i,k+1})] - f(x^*) \leq \frac{4C_1D+2LD^2}{p_i+k}.
\end{equation*}

By definition, $x_{i,k+1} = x_t$, where $t = \sum_{j=1}^{i-1}p_j+k+1 = p_i+k$. 
When $t=T$, we have 
\begin{equation*}
\expect[f(x_T)] - f(x^*) \leq \frac{4C_1D+2LD^2}{T}.
\end{equation*}

Therefore, to obtain an $\epsilon$-suboptimal solution, we need 
$\mathcal{O}(1/\epsilon)$ 
iterations. Let $T=\sum_{j=1}^{I-1}p_i+K = p_I+K-1$, then $I\le \log_2(T) + 1$, 
and 
thus IFO complexity per 
worker is 
\begin{equation*}
\begin{split}
IFO &\le \sum_{i=1}^I (n+ \sum_{j=2}^{p_i}S_{i,k}) \\
&\leq 
\sum_{i=1}^I(n+\frac{2^{2(i-1)}}{M}) \\
&\leq nI + 2^{2I}/M \\
&\leq  [\log_2(T)+1]N/M 
+ 4T^2/M \\
&= \mathcal{O}(\frac{N\ln(1/\epsilon)+1/\epsilon^2}{M}).
\end{split}
\end{equation*}

\subsection{Proof of Theorem 6 and Corollary 4.3.2}
The proof is quite similar to that of \cref{thm:finite_convex}. 

We first need to upper bound $\expect[\|\nabla f(x_{i,k}) - 
\bar{g}_{i,k}\|^2]$. \cref{eq:convex_finite_error1,eq:convex_finite_error3} 
still hold. Similarly, we also have for $k \ge 2$,
\begin{equation*}
\begin{split}
\expect[\| f(x_{i,k}) - g_{i,k}\|^2] &\leq 
\frac{L^2D^2\eta_{i,k-1}^2}{MS_{i,k}} + \expect[\| f(x_{i,k-1}) - 
g_{i,k-1}\|^2 \\
&=\frac{L^2D^2\eta_{i,k-1}^2}{p_i} + \expect[\| f(x_{i,k-1}) - 
g_{i,k-1}\|^2.
\end{split}
\end{equation*}

For $k=1$,
\begin{equation*}
\expect[\| f(x_{i,k}) - g_{i,k}\|^2] \leq \frac{\sigma^2}{MS_{i,1}} = 
\frac{\sigma^2}{M\frac{\sigma^2 p_i^2}{ML^2D^2}} = \frac{L^2D^2}{p_i^2}.
\end{equation*}

So 
\begin{equation}\label{eq:convex_stoch_aux1}
\expect[\| f(x_{i,k}) - g_{i,k}\|^2] \leq \frac{L^2D^2}{p_i^2} + 
\frac{L^2D^2}{p_i}\sum_{j=2}^k\eta_{i,j-1}^2 \leq \frac{L^2D^2}{p_i^2} + 
\frac{4L^2D^2}{p_i^2} = \frac{5L^2D^2}{p_i^2}.
\end{equation}

Combine 
\cref{eq:convex_finite_error1,eq:convex_finite_error3,eq:convex_stoch_aux1}, we 
have 
\begin{equation*}
\expect[\| f(x_{i,k}) - \bar{g}_{i,k}\|^2] \leq 
\frac{13L^2D^2+2G^2_\infty}{p_i^2} \triangleq \frac{C_2^2}{p_i^2}.
\end{equation*}

Applying \cref{lem:icml_thm1}, we have
\begin{equation*}
\expect[f(x_{i,k+1})] - f(x^*) \leq \frac{4C_2D+2LD^2}{p_i+k}.
\end{equation*}

By definition, $x_{i,k+1} = x_t$, where $t = \sum_{j=1}^{i-1}p_j+k+1 = p_i+k$. 
When $t=T$, we have 
\begin{equation*}
\expect[f(x_T)] - f(x^*) \leq \frac{4C_2D+2LD^2}{T}.
\end{equation*}

Therefore, to obtain an $\epsilon$-suboptimal solution, we need 
$\mathcal{O}(1/\epsilon)$ 
iterations. Let $T=\sum_{j=1}^{I-1}p_i+K = p_I+K-1$, then $I\le \log_2(T) + 1$, 
and thus SFO complexity per worker is 
\begin{equation*}
\begin{split}
SFO &\le \sum_{i=1}^I (\frac{\sigma^2p_i^2}{ML^2D^2}+ \sum_{j=2}^{p_i}S_{i,k}) 
\\
&\leq 
\sum_{i=1}^I(\frac{\sigma^22^{2(i-1)}}{ML^2D^2}+\frac{2^{2(i-1)}}{M}) \\
&\leq \frac{2^{2I}}{M} (\frac{\sigma^2 }{L^2D^2}+1) \\
&\leq 
\frac{4T^2}{M}(\frac{\sigma^2}{L^2D^2}+1) \\
&= \mathcal{O}(1/(M\epsilon^2)).
\end{split}
\end{equation*}

\subsection{Proof of Theorem 7 and Corollary 4.4.1}
First, since $x_{t+1} = (1-\eta_t)x_t + \eta_t v_t$ is a convex combination 
of $x_t, v_t$, and $x_1 \in \constraint, v_t \in \constraint, \forall\ t$, 
we can prove $x_t \in \constraint, \forall\ t $ by induction. So $x_{o} \in 
\constraint$.

Then we turn to upper bound $\expect[\|\nabla f(x_{i,k}) - \bar{g}_{i,k} 
\|^2]$. \cref{eq:convex_finite_error1} still holds. Similarly, we also have for 
$k \ge 2$,
\begin{equation*}
\begin{split}
\expect[\| f^{(m)}(x_{i,k}) - g^{(m)}_{i,k}\|^2] &\leq 
\frac{L^2D^2\eta_{i,k-1}^2}{S_{i,k}} + \expect[\| f^{(m)}(x_{i,k-1}) - 
g^{(m)}_{i,k-1}\|^2 \\
&=\frac{L^2D^2T^{-1}}{\frac{\sqrt{n}}{M}} + \expect[\| f^{(m)}(x_{i,k-1}) - 
g^{(m)}_{i,k-1}\|^2 \\
&=\frac{ML^2D^2}{\sqrt{n}T} + \expect[\| f^{(m)}(x_{i,k-1}) - 
g^{(m)}_{i,k-1}\|^2.
\end{split}
\end{equation*}

For $k=1$, we have $g_{i,1}^{(m)} = \nabla f^{(m)}(x_{i,1})$. So
\begin{equation*}
\expect[\| \nabla f^{(m)}(x_{i,k}) - g_{i,k}^{(m)}\|^2] \leq 
\frac{ML^2D^2}{\sqrt{n}T} (k-1) \leq \frac{ML^2D^2}{\sqrt{n}T} p_i= 
\frac{ML^2D^2}{T}.
\end{equation*}

By \cref{eq:finite_convex_1},
\begin{equation}\label{eq:finite_nonconvex_1}
\expect[\nabla f(x_{i,k}) - g_{i,k}\|^2] \leq \frac{M\frac{ML^2D^2}{T}}{M^2} = 
\frac{L^2D^2}{T}.
\end{equation}

We also have
\begin{equation}\label{eq:finite_nonconvex_2} 
\begin{split}
\expect[\| g_{i,k} -\bar{g}_{i,k}\|^2] &\leq 
\sum_{j=2}^k\frac{\eta_{i,j-1}^2dL^2D^2}{Ms_{1,{i,j}}^2}+
\sum_{j=2}^k\frac{\eta_{i,j-1}^2dL^2D^2}{s_{2,{i,j}}^2} + 
\frac{dG_\infty^2}{Ms_{1,i,1}^2}+\frac{dG_\infty^2}{s_{2,i,1}^2} \\
&\leq \frac{p_idL^2D^2}{TM \frac{d\sqrt{n}}{M}} + 
\frac{p_idL^2D^2}{Td\sqrt{n}} + \frac{dG_\infty^2}{M\frac{Td}{M}} + 
\frac{dG_\infty^2}{dT}\\
&=\frac{2(L^2D^2+G_\infty^2)}{T}.
\end{split}  	
\end{equation}

Combine 
\cref{eq:convex_finite_error1,eq:finite_nonconvex_1,eq:finite_nonconvex_2}
\begin{equation*}\label{eq:nonconvex_finite_error}
\expect[\|\nabla f(x_{i,k}) - \bar{g}_{i,k} \|^2] \leq 
\frac{3L^2D^2+2G_\infty^2}{T}.
\end{equation*}

By \cref{assump_on_f_i_nonconvex}, $f$ is also a bounded (potentially) 
non-convex 
function on $\constraint$ with $L$-Lipschitz continuous gradient. 
Specifically, we have $\sup_{x \in \constraint}|f(x)| \leq M_0$. Note that 
if we define $v_t^\prime = \argmin_{v \in \constraint}\langle v, 
\nabla f(x_t)\rangle$, then $\mathcal{G}(x_t)=\langle 
v^\prime_t-x_t,-\nabla f(x_t)\rangle = -\langle v^\prime_t-x_t,\nabla 
f(x_t)\rangle$. So we have

\begin{equation*}
\begin{split}
f(x_{t+1}) 
\stackrel{(a)}{\leq}{}& f(x_t) + \langle \nabla f(x_t),x_{t+1}-x_t \rangle + 
\frac{L}{2}\|x_{t+1}-x_t\|^2 \\
={}&f(x_t) + \langle \nabla f(x_t),\eta_t(v_t-x_t) \rangle + 
\frac{L}{2}\|\eta_t(v_t-x_t)\|^2 \\
\stackrel{(b)}{\leq}{}& f(x_t) + \eta_t \langle \nabla f(x_t),v_t-x_t 
\rangle+\frac{L\eta_t^2D^2}{2} \\
={}& f(x_t) + \eta_t \langle  \bar{g}_t,v_t-x_t \rangle+ \eta_t \langle 
\nabla 
f(x_t)-\bar{g}_t,v_t-x_t \rangle  + \frac{L\eta_t^2D^2}{2} \\
\stackrel{(c)}{\leq}{}& f(x_t) + \eta_t \langle  \bar{g}_t,v^\prime_t-x_t 
\rangle+ \eta_t \langle \nabla f(x_t)-\bar{g}_t,v_t-x_t \rangle + 
\frac{L\eta_t^2D^2}{2} \\
={}& f(x_t) + \eta_t \langle \nabla f(x_t),v^\prime_t-x_t \rangle  + \eta_t 
\langle  \bar{g}_t-\nabla f(x_t),v^\prime_t-x_t \rangle \\
&\quad + \eta_t \langle 
\nabla f(x_t)-\bar{g}_t,v_t-x_t \rangle + 
\frac{L\eta_t^2D^2}{2} \\
={}& f(x_t) - \eta_t \mathcal{G}(x_t) + \eta_t \langle \nabla 
f(x_t)-\bar{g}_t,v_t-v^\prime_t \rangle + \frac{L\eta_t^2D^2}{2} \\
\stackrel{(d)}{\leq}{}& f(x_t) - \eta_t \mathcal{G}(x_t) + \eta_t \|\nabla 
f(x_t)-\bar{g}_t\|\| v_t-v^\prime_t\| + \frac{L\eta_t^2D^2}{2} \\
\stackrel{(e)}{\leq}{}& f(x_t) - \eta_t \mathcal{G}(x_t) + \eta_tD \|\nabla 
f(x_t)-\bar{g}_t\| + \frac{L\eta_t^2D^2}{2},
\end{split}
\end{equation*}
where we used the assumption that $f$ has $L$-Lipschitz continuous gradient 
in inequality (a). Inequalities (b), (e) hold because of 
\cref{assump_on_K}. Inequality (c) is due to the optimality of $v_t$, and 
in (d), we applied the Cauchy-Schwarz inequality.

Rearrange the inequality above, we have 
\begin{equation}
\label{eq:bound_on_individual_gap}
\eta_t \mathcal{G}(x_t) \leq f(x_t)- f(x_{t+1})+ \eta_tD \|\nabla 
f(x_t)-\bar{g}_t\| + \frac{L\eta_t^2D^2}{2}.   
\end{equation}

Apply \cref{eq:bound_on_individual_gap} recursively for $t = 1, 2, \cdots, 
T$, and take expectations, we attain the following inequality:
\begin{equation*}
\label{eq:bound_on_gap}
\begin{split}
\sum_{t=1}^T \eta_t \expect[\mathcal{G}(x_t)]
\leq f(x_1)-f(x_{T+1})
+D\sum_{t=1}^T \eta_t\expect[\|\nabla f(x_t)-\bar{g}_t\|] 
+\frac{LD^2}{2}\sum_{t=1}^T \eta_t^2.
\end{split}
\end{equation*}

Since we have $\expect[\|\nabla f(x_{i,k}) - \bar{g}_{i,k} \|^2] \leq 
\frac{3L^2D^2+2G_\infty^2}{T} \triangleq \frac{c^2}{T}$, we have
\begin{equation*}
\expect[\|\nabla f(x_t)-\bar{g}_t\|] \leq \sqrt{\expect[\|\nabla 
	f(x_t)-\bar{g}_t\|^2]} \leq \frac{c}{\sqrt{T}}.    
\end{equation*}

With $\eta_t = T^{-1/2}$, we then have
\begin{equation*}
\begin{split}
\sum_{t=1}^T \expect[\mathcal{G}(x_t)]
\leq{}& \sqrt{T}[f(x_1)-f(x_{T+1})] + 
D \sum_{t=1}^T \expect[\|\nabla f(x_t)-\bar{g}_t\|] + \sqrt{T} 
\frac{LD^2}{2} T (T^{-1/2})^2 \\
\leq{}& 2M_0\sqrt{T}+DT\frac{c}{\sqrt{T}}+\frac{LD^2}{2}\sqrt{T}\\
={}& (2M_0+cD+\frac{LD^2}{2})\sqrt{T}.
\end{split}
\end{equation*}

So
\begin{equation*}
\expect[\mathcal{G}(x_o)]=\frac{\sum_{t=1}^T \expect[\mathcal{G}(x_t)]}{T} 
\leq \frac{2M_0+cD+\frac{LD^2}{2}}{\sqrt{T}}.    
\end{equation*}

Therefore, in order to find an $\epsilon$-first-order stationary points, we 
need at most $\mathcal{O}(1/\epsilon^2)$ iterations. The IFO complexity per 
worker is $[n+2(p-1)S_{i,k}]\cdot \frac{T}{p} = 
\mathcal{O}(\sqrt{n}/\epsilon^2)=\mathcal{O}(\sqrt{N}/(\epsilon^2\sqrt{M}))$. 
The average communication bits per round is $\frac{1}{p} \{M[32+d(z_{1,i,1}+1) 
+ (p-1)(32+d(z_{1,i,k}+1)) ] + [32+d(z_{2,i,1}+1) + (p-1)(32+d(z_{2,i,k}+1))]  
\} = (32+d)(M+1) + \frac{Md}{\sqrt{n}}\log_2(\sqrt{\frac{Td}{M}}+1) + 
Md\log_2(\frac{d^{1/2}n^{1/4}}{\sqrt{M}}+1) + 
\frac{d}{\sqrt{n}}\log_2(\sqrt{TD}+1) + d\log_2(d^{1/2}n^{1/4}+1).$

	\chapter{Black-Box Submodular Maximization}\label{cha:black}
	

\section{Introduction}

The focus of this chapter\footnote{\revise{This chapter is based on our work in 
		\citep{chen2020black}.}} is the \textit{constrained} continuous 
DR-submodular maximization 
over a bounded convex body. We aim to design an algorithm that uses only 
zeroth-order 
information while avoiding expensive  projection operations. Note that one way 
the 
optimization methods can deal with constraints is to apply the projection 
oracle once the 
proposed iterates land outside the feasibility region. However, computing 
the projection in 
many constrained 
settings is computationally prohibitive 
(\emph{e.g.}, 
projection over bounded 
trace norm matrices, flow polytope, matroid polytope, rotation matrices). In 
such scenarios, 
projection-free algorithms, \emph{a.k.a.}, Frank-Wolfe 
\citep{frank1956algorithm},  replace 
the projection with a linear program. Indeed, our proposed algorithm combines 
efficiently the 
zeroth-order information with solving a series of linear programs to ensure 
convergence to a 
near-optimal solution.


\textbf{Motivation:}
\revise{Computing the gradient 
of a continuous DR-submodular function has been shown to be computationally prohibitive (or even 
intractable) in many applications. For example, 
the objective function of influence maximization is defined via specific  
stochastic processes \citep{kempe2003maximizing,rodriguez2012influence} and computing/estimating the gradient of the mutliliear extension would require a relatively high computational complexity. In the 
problem of D-optimal experimental design 
, the gradient of the objective function involves inversion of a potentially 
large matrix \citep{chen2018online}.} Moreover, when one 
attacks a submodular recommender model, only black-box information is available 
and the service provider is unlikely to provide additional first-order 
information (this is known as the black-box adversarial attack model) 
\citep{lei2019discrete}.

There has been very recent progress on developing zeroth-order methods for 
constrained optimization 
problems in convex and non-convex settings 
\cite{ghadimi2013stochastic,sahu2018towards}. Such methods typically assume 
the objective function is defined on the whole $\mathbb{R}^d$ so that they 
can sample points from a proper 
distribution defined on $\mathbb{R}^d$. For DR-submodular functions, this 
assumption might be unrealistic, since 
many DR-submodular 
functions might be only defined on a subset of $\mathbb{R}^d$, 
\emph{e.g.}, the multi-linear extension 
\cite{vondrak2008optimal}, a canonical example of DR-submodular functions, is 
only defined on a unit 
cube. Moreover, they can only guarantee to reach a first-order stationary 
point. However, 
\citet{hassani2017gradient} showed that for a monotone 
DR-submodular function, the stationary points can only guarantee $1/2$ 
approximation to the optimum. 
\revise{Therefore, if a state-of-the-art zeroth-order 
	non-convex algorithm is used for maximizing a monotone DR-submodular 
	function, it is likely to terminate at a suboptimal stationary point whose 
	approximation ratio is only $ 1/2 $.}
%

\textbf{Our contributions:} In this chapter, we  propose a derivative-free 
and projection-free algorithm \Algblack (BCG), that maximizes a monotone 
continuous 
DR-submodular 
function over 
a bounded convex body $\constraint\subseteq \mathbb{R}^d$. We 
consider three scenarios:

(1) In the deterministic setting, where function evaluations can be obtained 
exactly, BCG achieves the tight $[(1-1/e)OPT-\epsilon]$ approximation guarantee 
with $\mathcal{O}(d/\epsilon^3)$ function evaluations.

(2) In the stochastic setting, where function evaluations are noisy, BCG 
achieves the tight $[(1-1/e)OPT-\epsilon]$ approximation guarantee with 
$\mathcal{O}(d^3/\epsilon^5)$ function evaluations.


(3) In the discrete setting, \Algdis (DBG), the discrete
version of BCG, achieves the tight 
$[(1-1/e)OPT-\epsilon]$ 
approximation guarantee with $\mathcal{O}(d^5/\epsilon^5)$ function evaluations.

All the theoretical results are summarized in 
\cref{tab:number_of_queries}.

\begin{table*}[t]
	\centering
	\caption{Number of function queries in different 
		settings, where $D_1$ is the diameter of 
		$\constraint$.\label{tab:number_of_queries}}
	\begin{tabular}{ll}
		\toprule[1.5pt]
		Function &Function Queries  \\ 
		\midrule
		continuous DR-submodular 
		& 
		$\mathcal{O}(\max\{G, LD_1\}^3 \cdot \frac{d}{\epsilon^3})$ 
		[\cref{thm:zero}]\\ 
		stochastic continuous DR-submodular  &
		$\mathcal{O}(\max\{G, LD_1\}^3 \cdot \frac{d^3}{\epsilon^5})$ 
		[\cref{thm:zero_stochatic}] \\ 
		discrete submodular &
		$\mathcal{O}(\frac{d^5}{\epsilon^5})$ [\cref{thm:zero_discrete}]
		\\ 
		\bottomrule[1.25pt]
	\end{tabular}
	%
\end{table*}

We would like to note that in the discrete setting, due to  
the conservative upper bounds for the Lipschitz and smooth parameters of 
general multilinear extensions, 
and the variance of 
the gradient estimators 
subject to noisy function evaluations,
the required number of function 
queries in theory is larger than the best known result, 
$\mathcal{O}(d^{5/2}/\epsilon^3)$ 
in  \cite{mokhtari2018conditional,mokhtari2018stochastic}. However, our 
experiments (in \cref{sec:experiment-black}) show 
that empirically, our proposed algorithm often requires significantly fewer 
function evaluations and less running time, while achieving a practically 
similar utility.

\textbf{Novelty of our work:} All the previous results in 
constrained DR-submodular maximization assume access to (stochastic) gradients. 
In this work, we address a harder problem, \emph{i.e.}, we provide the first 
rigorous analysis when only (stochastic) function values can be obtained.
More specifically, with the smoothing trick \citep{flaxman2005online}, one can 
construct an unbiased gradient estimator via function queries. However, this 
estimator has a large $\mathcal{O}(d^2/\delta^2)$ variance which may cause  
FW-type 
methods to diverge. To overcome this issue, we build on the momentum method 
proposed by \citet{mokhtari2018conditional} in which they assumed access to the 
\emph{first-order} information.
	
Given a point $x$, the smoothed version of $F$ at $x$ is 
defined as 
$\expect_{v\sim B^d}[F(x+ \delta v)]$. If $x$ is close to the boundary of the 
domain $\domain, (x + \delta v)$ may fall outside of $\domain$, leaving the 
smoothed function undefined for many instances of DR-submodular functions 
(\emph{e.g.}, the multilinear extension is only defined over the unit cube). 
Thus the vanilla smoothing trick will not work. To this end, we transform the 
domain $\domain$ and constraint set $\constraint$ in a proper way 
and run our zeroth-order method on the transformed constraint set 
$\constraint'$. Importantly, we retrieve the same convergence rate of 
$\mathcal{O}(T^{-1/3})$ as in \cite{mokhtari2018conditional} with a minimum 
number of 
function queries in different settings (continuous, stochastic continuous, 
discrete).

We further note that by using more recent variance reduction techniques 
\citep{zhang2020one}, one might be able to reduce the required number of 
function evaluations. All the proofs in this chapter are provided in 
\cref{sec:proof-black}.
	

\section{Smoothing Trick}\label{sec:preliminaries}
For a function $F$ defined on $\mathbb{R}^d$, its $\delta$-smoothed version is 
given as 
\begin{equation*}
\label{eq:smooth_definition}
\tF_\delta(x)\triangleq \expect_{v\sim B^d}[F(x+\delta v)],    
\end{equation*}
where $v$ is chosen uniformly at random from the $d$-dimensional unit ball 
$B^d$. In words, the function $\tF_\delta$ at any point $x$ is obtained by 
``averaging''  $F$ over a ball of 
radius $\delta$ around $x$. In the sequel, we omit the subscript $\delta$ for 
the sake of simplicity and use  $\tF$ instead of $\tF_\delta$.


\cref{lem:smooth_approx} below shows that under the Lipschitz assumption for $F$, the 
smoothed 
version $\tF$ is a good approximation of $F$, and also inherits the key 
structural properties of  
$F$ (such as monotonicity and submodularity). Thus one can 
(approximately) optimize $F$ via optimizing~$\tF$. 

\begin{lemma}\label{lem:smooth_approx}
	If $F$ is monotone continuous 
	DR-submodular and $G$-Lipschitz continuous on $\mathbb{R}^d$,  
	then 
	so is $\tF$ and 
	\begin{equation*}
	|\tF(x)-F(x)|\le \delta G.
	\end{equation*}   
\end{lemma}

An important property of $\tF$ is that one can obtain an unbiased estimation for its 
gradient 
$\nabla \tF$ by a single query of $F$. This property plays a key 
role in our proposed derivative-free algorithms.


\begin{lemma}[Lemma 6.5 in \citep{hazan2016introduction}]
	\label{lem:gradient_of_smooth} 
	Given a function $F$ on $\mathbb{R}^d$, if we choose $u$ uniformly at 
	random from the $(d-1)$-dimensional unit 
	sphere $S^{d-1}$, then 
	we have 
	\begin{equation*} \label{eq:gradient_smooth}
	\nabla \tF(x)= \expect_{u \sim S^{d-1}}\left[\frac{d}{\delta}F(x+\delta 
	u)u\right].    
	\end{equation*}
\end{lemma}

\section{DR-Submodular Maximization}
In this chapter, we mainly focus on the constrained optimization 
problem: 
\begin{equation*}
\max_{x \in \constraint} F(x),
\end{equation*}
where $F$ is a monotone continuous DR-submodular function on $\mathbb{R}^d$, and
the constraint set $\constraint \subseteq \mathcal{X} \subseteq\mathbb{R}^d$ is 
convex and compact. 
 
For \emph{first-order} monotone DR-submodular maximization, one can use \AlgCG 
\cite{calinescu2011maximizing, bian2017guaranteed}, a variant of Frank-Wolfe 
Algorithm \citep{frank1956algorithm,jaggi2013revisiting,lacoste2015global}, 
to achieve the  $[(1-1/e)OPT-\epsilon]$ approximation 
guarantee. At iteration $t$, the FW variant first maximizes the linearization 
of the objective function $F$: 
\begin{equation*}
v_t = \argmax_{v \in \constraint} \langle v, 
\nabla F(x_t) \rangle.
\end{equation*}
Then the current point $x_t$ moves 
in the direction of $v_t$ with a step size $\gamma_t \in (0,1]$: 
\begin{equation*}
x_{t+1} = x_t + \gamma_t v_t.
\end{equation*}
Hence, by solving linear optimization problems, the iterates 
are updated without resorting to the projection oracle. 

Here we introduce our main algorithm \Algblack which assumes access only 
to function values (\emph{i.e.}, zeroth-order information). This algorithm is 
partially based on the idea of \AlgCG. The basic idea is to utilize the 
function evaluations of $F$ at carefully selected points to obtain unbiased 
estimations of the gradient of the smoothed version, $\nabla \tF$. By 
extending 
\AlgCG to the derivative-free setting 
and using recently proposed variance reduction 
techniques, we can then optimize $\tF$ 
near-optimally. Finally, by \cref{lem:smooth_approx} we show that the obtained 
optimizer also provides a good solution for $F$. 

Recall that continuous DR-submodular functions are 
defined on a box $\mathcal{X}=\Pi_{i=1}^n \mathcal{X}_i$. To simplify the exposition, 
we can assume, without loss of generality, that
the  objective function $F$ is defined on $\domain\triangleq \prod_{i=1}^d 
[0,a_i]$ \cite{bian2017continuous}.
Moreover, we note that 
since $\tF = \expect_{v\sim B^d}[F(x+\delta v)]$, for $x$ close to $\partial 
\domain$ (the boundary 
of $\domain$), the point 
$x + \delta v$ may fall outside of $\domain$, leaving the function $\tF$ 
undefined. 

To circumvent this issue, we shrink the 
domain $\domain$ by $\delta$. Precisely, the shrunk domain is defined as
\begin{equation*} \label{D'}
\domain'_\delta = \{x \in \domain| d(x, \partial \domain) \geq \delta \}.
\end{equation*}
Since we assume  
$\domain = \prod_{i=1}^d [0,a_i]$, the shrunk domain is 
$\domain'_\delta = 
\prod_{i=1}^d[\delta, a_i-\delta]$.
Then for all $x \in \domain'_\delta$, we have $x 
+ \delta v \in \domain$.  So $\tF$ is well-defined on $\domain'_\delta$.
By \cref{lem:smooth_approx}, the optimum of 
$\tF$ on the shrunk domain $\domain'_\delta$ will be 
close to that on the original domain $\domain$, if $\delta$ is small enough. 
Therefore, we can first optimize $\tF$ on $\domain'_\delta$, then
approximately optimize $\tF$ (and thus $F$) on $\domain$. For simplicity of 
analysis, we also translate the shrunk domain 
$\domain'_\delta$
by $-\delta$, and denote it as $\domainsh = \prod_{i=1}^d [0, a_i-2\delta]$. 

Besides the domain $\domain$, 
we also need to consider the transformation on constraint set 
$\constraint$. Intuitively, if there is no translation, we should consider the 
intersection of $\constraint$ and the shrunk domain $\domain'_\delta$. But 
since we translate $\domain'_\delta$ by $-\delta$, the same transformation 
should be performed on $\constraint$. Thus, we define the transformed 
constraint set as the translated intersection (by $-\delta$) of 
$\domain'_\delta$ and $\constraint$:
\begin{equation*} \label{K'}
\constraint'\triangleq  (\domain'_\delta \cap \constraint) - \delta \one 
=\domainsh \cap (\constraint-\delta\one).
\end{equation*} 

It is well known 
that the FW Algorithm is  sensitive to 
the accuracy of gradient, and may have arbitrarily poor performance with 
stochastic gradients \cite{hazan2016variance,mokhtari2018stochastic}. Thus we 
incorporate 
two methods of variance reduction into our 
proposed algorithm \Algblack which 
correspond to Step 5 and Step 6 in \cref{alg:zero_frank_wolfe}, respectively. 
First, instead of the one-point gradient estimation in 
\cref{lem:gradient_of_smooth}, we adopt the two-point estimator of $\nabla 
\tF(x)$
\citep{agarwal2010optimal,shamir2017optimal}:
\begin{equation}
\label{eq:gradient_smooth_two_points}
\frac{d}{2\delta}(F(x+\delta u) - F(x-\delta u)) u,
\end{equation}
where $u$ is chosen uniformly at random from the unit sphere $S^{d-1}$.We note 
that \eqref{eq:gradient_smooth_two_points} is an unbiased gradient estimator  
with less variance w.r.t. the one-point estimator. We also average over a 
mini-batch of $B_t$ independently sampled two-point estimators for further 
variance reduction. The second 
variance-reduction 
technique is the momentum method used in 
\citep{mokhtari2018conditional} to estimate the gradient by a vector 
$\bar{g}_t$ which is updated at each iteration as follows:
\begin{equation*}
\label{eq:averaging}
\bar{g}_t = (1-\rho_t) \bar{g}_{t-1} + \rho_t g_t.
\end{equation*}
%
Here $\rho_t$ is a given step size, $\bar{g}_0$ is initialized as an all zero vector $\mathbf{0}$, and $g_t$ 
is an unbiased estimate of the gradient at iterate $x_t$. As $\bar{g}_t$ is a 
weighted average of previous gradient approximation $\bar{g}_{t-1}$ and the 
newly updated stochastic gradient $g_t$, it has a lower variance compared with 
$g_t$. Although $\bar{g}_t$ is not an unbiased estimation of the true gradient, 
the error of it will approach zero as time proceeds. The detailed description 
of \Algblack is provided in \cref{alg:zero_frank_wolfe}. 


\begin{algorithm}[t!]
	\begin{algorithmic}[1]
		\Require constraint set $\constraint$, iteration 
		number $T$, radius $\delta$, step size $\rho_t$, batch size $B_t$
		\Ensure $x_{T+1}+\delta \one$
		\State $x_1\gets \mathbf{0}, \enspace \bar{g}_0 \gets \mathbf{0}$
		\For{$t= 1,\cdots, T$}
		\State Sample $u_{t,1},\dots, u_{t,B_t}$ i.i.d.\ from $S^{d-1}$
		\State For $i=1$ to $B_t$, let $y_{t,i}^+ \gets \delta \one+x_t+\delta 
		u_{t,i}, y_{t,i}^- \gets \delta \one+x_t-\delta u_{t,i}$ and evaluate 
		$F(y_{t,i}^+), F(y_{t,i}^-)$ 
		\State $g_t\gets \frac{1}{B_t}\sum_{i=1}^{B_t} 
		\frac{d}{2\delta}[F(y_{t,i}^+)-F(y_{t,i}^-)]u_{t,i}$
		\State $\bar{g}_t\gets (1-\rho_t)\bar{g}_{t-1}+\rho_t g_t$
		\State $v_t\gets \argmax_{v\in \constraint' } \langle 
		v,\bar{g}_t\rangle $
		\State $x_{t+1}\gets x_t+\frac{v_t}{T}$
		\EndFor
		\State Output $x_{T+1}+\delta \one$
	\end{algorithmic}
	\caption{\Algblack}\label{alg:zero_frank_wolfe}
\end{algorithm} 

\begin{theorem}
	\label{thm:zero}
	For a monotone continuous DR-submodular function $F$, which is also 
	$G$-Lipschitz continuous and $L$-smooth on a 
	convex and compact constraint set $\constraint$, if we set 
	$\rho_t=2/(t+3)^{2/3}$ in \cref{alg:zero_frank_wolfe}, then we have
	\begin{align*}
	&(1-1/e)F(x^*)-\expect[F(x_{T+1}+\delta\one)] 
	\le
	\frac{3D_1Q^{1/2}}{T^{1/3}}+ \frac{LD_1^2}{2 T} 
	+ \delta G(1+(\sqrt{d}+1)(1-1/e)),
	\end{align*}
	where $Q= \max \{ 4^{2/3}G^2, 
	4cdG^2/B_t+ 6L^2D_1^2 \}, c$ is a constant, $D_1= \diam(\constraint')$, and 
	$x^*$ is the global 
	maximizer of $F$ on $\constraint$.
\end{theorem}
\begin{remark}
By setting $T=\mathcal{O}(1/\epsilon^3)$, $B_t=d$, and 
$\delta=\epsilon/\sqrt{d}$, the error term is guaranteed to be at 
most $\mathcal{O}(\epsilon)$. Also, the total number of function evaluations is 
at 
most $\mathcal{O}(d/\epsilon^3)$.	
\end{remark}

We can also extend \cref{alg:zero_frank_wolfe} to the stochastic case in which
we obtain information 
about $F$ only through its noisy function evaluations $\hat{F}(x)=F(x)+\xi$, 
where $\xi$ is stochastic zero-mean noise. In particular,   
in Step 4 of \cref{alg:zero_frank_wolfe}, we obtain independent stochastic 
function evaluations $\hat{F}(y_{t,i}^+)$ and $ \hat{F}(y_{t,i}^-)$, instead of the 
exact function values $F(y_{t,i}^+)$ and $F(y_{t,i}^-)$. For unbiased function 
evaluation oracles with uniformly bounded variance, we have the following 
theorem.  
\begin{theorem}
	\label{thm:zero_stochatic}
	Under the condition of \cref{thm:zero}, if we further assume that for all 
	$x$, $\expect[\hat{F}(x)]=F(x)$ and $\expect[|\hat{F}(x)-F(x)|^2]\leq 
	\sigma_0^2$, then we have
	\begin{align*}
	&(1-1/e)F(x^*)-\expect[F(x_{T+1}+\delta\one)] 
	\le
	\frac{3D_1Q^{1/2}}{T^{1/3}}+ \frac{LD_1^2}{2 T} 
	+ \delta G(1+(\sqrt{d}+1)(1-1/e)),
	\end{align*}
	where $D_1= \diam(\constraint'), Q= \max \{4^{2/3}G^2, 
	6L^2D_1^2+(4cdG^2+2d^2\sigma_0^2/\delta^2)/B_t\}, c$ is a constant, and
	$x^*$ is the global 
	maximizer of $F$ on $\constraint$.
\end{theorem}
\begin{remark}
By setting $T=\mathcal{O}(1/\epsilon^3)$, 
$B_t=d^3/\epsilon^2$, and $\delta=\epsilon/\sqrt{d}$, 
the error term is at most $\mathcal{O}(\epsilon)$.
The total number of evaluations is at most 
$\mathcal{O}(d^3/\epsilon^5)$.
\end{remark}

	
\section{Discrete Submodular Maximization}\label{sec:discrete}
In this section, we describe how \Algblack can be used to solve a discrete 
submodular maximization problem 
with a general matroid constraint, \emph{i.e.}, $\max_{S \in \mathcal{I}} 
f(S)$, 
where $f$ is a monotone submodular set function and $\mathcal{I}$ is the 
matroid constraint.

In combinatorics, the matroid is an analogue to the notion of linear 
independence in linear algebra. Precisely, consider a ground set $\Omega$ and a 
family of subsets of $\Omega$ denoted as $\mathcal{I}$. We say the pair 
$(\Omega, \mathcal{I})$ is a matroid\footnote{\revise{For a detailed review of 
matroid theory, interested readers refer to \citep{oxley2006matroid}.}} if 

\begin{enumerate}
\item $\emptyset \in \mathcal{I}$.

\item For each $A \in \mathcal{I}$, if $A' \subseteq A$, then $A' \in 
\mathcal{I}$.

\item If $A \in \mathcal{I}, B \in \mathcal{I}, |A| >|B|$, then $\exists x \in 
A\setminus B$, such that $\{x\} \cup B \in \mathcal{I}$.
\end{enumerate}


For any monotone submodular set function $f: 2^\Omega \to \mathbb{R}_{\geq 0}$, 
its multilinear extension $F:[0,1]^d \to \mathbb{R}_{\geq 0}$, defined as 
\begin{equation*}
F(x)=\sum_{S \subseteq \Omega}f(S)\prod_{i \in S}x_i\prod_{j \notin S}(1-x_j),
\end{equation*}
is monotone and DR-submodular \citep{calinescu2011maximizing}. Here, 
$d=|\Omega|$ is
the size of the ground set $\Omega$. 
Equivalently, 
we have 
$F(x) = \expect_{S \sim x}[f(S)],$ where  $S \sim x$ means that the each 
element $i \in \Omega$ is included in $S$ with probability $x_i$ independently.

It can be shown that in lieu of solving the discrete optimization problem 
one can solve the continuous optimization problem $\max_{x \in \mathcal{K}} 
F(x),$
where 
$\mathcal{K} = \text{conv}\{1_I: I 
\in \mathcal{I}\}$ is the matroid polytope \citep{calinescu2011maximizing}. 
This equivalence is 
obtained by showing that (i) the optimal values of the two problems 
are the same, and (ii) for any fractional 
vector $x \in \mathcal{K}$ we can deploy efficient, lossless rounding 
procedures that produce a set $S \in \mathcal{I}$ such that $\mathbb{E}[f(S)] 
\geq F(x)$ (\emph{e.g.}, pipage 
rounding \citep{ageev2004pipage,calinescu2011maximizing} and contention 
resolution \citep{chekuri2014submodular}). So we can view $\tF$ as the 
underlying function that we intend to optimize, and invoke \Algblack. 
As a result, we want that $F$ is $G$-Lipschitz 
and $L$-smooth as in \cref{thm:zero}. The following lemma shows these 
properties are 
satisfied automatically 
if $f$ is bounded. 
\begin{lemma}
	\label{lem:discrete_to_continuous}
	For a submodular set function $f$ defined on $\Omega$ with $\sup_{X 
		\subseteq \Omega} |f(X)| 
	\le M$, its multilinear extension $F$ is $2M\sqrt{d}$-Lipschitz and  
	$4M\sqrt{d(d-1)}$-smooth.      
\end{lemma}

We note that the bounds for Lipschitz and smoothness parameters 
actually 
depend on the norms that we consider. However, different norms are equivalent 
up 
to a factor that may depend on the dimension. If we consider another norm, some 
dimension factors may be absorbed into the norm. Therefore, we only study 
Euclidean norm in \cref{lem:discrete_to_continuous}.

We further note that computing the exact value of 
$F$ is difficult as it requires evaluating $f$ over all the subsets $S \in 
\Omega$.  However, one can construct an unbiased estimate for the value $F(x)$  
by simply sampling a random set $S\sim x$ and returning $f(S)$ as the estimate. 
We present our algorithm in detail in 
\cref{alg:zero_frank_wolfe_discrete},  
where we have $\domain = [0, 1]^d$, since 
$F$ is defined on $[0, 1]^d$, and thus $\domainsh = [0, 
1-2\delta]^d$. 
We state the theoretical result formally in \cref{thm:zero_discrete}.

\begin{algorithm}[t!]
	\begin{algorithmic}[1]
		\Require matroid constraint $\matroid$, transformed 
		constraint set $\constraint' = \domainsh\cap (\constraint - \delta 
		\one)$ where $\constraint = \text{conv}\{1_I: I \in \mathcal{I}\}$,	
		number of iterations $T$, radius $\delta$, step size $\rho_t$, batch 
		size $B_t$, sample size $S_{t,i}$	
		\Ensure $X_{T+1}$
		\State $x_1\gets \mathbf{0}, \enspace \bar{g}_0 \gets \mathbf{0}$, 
		\For{$t=1, \cdots, T$}
		\State Sample $u_{t,1},\dots, u_{t,B_t}$ i.i.d.\ from $S^{d-1}$
		\State For $i=1$ to $B_t$, let $y_{t,i}^+ \gets \delta \one+x_t+\delta 
		u_{t,i}, 
		y_{t,i}^- \gets \delta \one+x_t-\delta u_{t,i}$, independently sample 
		subsets 
		$Y_{t,i}^+$ and $Y_{t,i}^-$ for $S_{t,i}$ times according to 
		$y_{t,i}^+, 
		y_{t,i}^-$, get sampled subsets $Y_{t,i,j}^+, Y_{t,i,j}^-, \forall\ j 
		\in 
		[S_{t,i}]$, evaluate the function values 
		$f(Y_{t,i,j}^+), f(Y_{t,i,j}^-), \forall\ j \in 
		[S_{t,i}]$, and calculate the averages 
		$\bar{f}_{t,i}^+ \gets \frac{\sum_{j=1}^{S_{t,i}} 
		f(Y_{t,i,j}^+)}{S_{t,i}}, 
		\bar{f}_{t,i}^- \gets \frac{\sum_{j=1}^{S_{t,i}} 
		f(Y_{t,i,j}^-)}{S_{t,i}}$ 
		\State $g_t\gets \frac{1}{B_t}\sum_{i=1}^{B_t} 
		\frac{d}{2\delta}(\bar{f}_{t,i}^+-\bar{f}_{t,i}^-)u_{t,i}$
		\State $\bar{g}_t\gets (1-\rho_t)\bar{g}_{t-1}+\rho_t g_t$
		\State $v_t\gets \argmax_{v\in \constraint' } \langle 
		v,\bar{g}_t\rangle $
		\State $x_{t+1}\gets x_t+\frac{v_t}{T}$
		\EndFor
		\State Output $X_{T+1} = \text{round}(x_{T+1}+\delta \one)$
	\end{algorithmic}
	\caption{\Algdis}\label{alg:zero_frank_wolfe_discrete}
\end{algorithm}



\begin{theorem}\label{thm:zero_discrete}
	For a monotone submodular set function $f$ with $\sup_{X \subseteq \Omega} 
|f(X)| \le M$, if we set $\rho_t=2/(t+3)^{2/3}, S_{t,i}=l$ in 
\cref{alg:zero_frank_wolfe_discrete}, then we have	
\begin{align*}
&(1-1/e)f(X^*) - 
\expect[f(X_{T+1})]\\
\le{} & 
\frac{3D_1Q^{1/2}}{T^{1/3}}+ \frac{2M\sqrt{d(d-1)}D_1^2}{T} + 2M\delta 
\sqrt{d}(1+(\sqrt{d}+1)(1-1/e)),
\end{align*}
where $D_1= \diam(\constraint')$, $Q= \max 
\{\frac{2d^2M^2(\frac{1}{l\delta^2}+8c)}{B_t}+ 96d(d-1)M^2D_1^2, 
4^{5/3}dM^2 \}, c$ is a constant, $X^*$ is the global maximizer of $f$ 
under matroid constraint $\matroid$.	
\end{theorem}

\begin{remark}
	By setting $T=\mathcal{O}(d^3/\epsilon^3)$, $B_t=1, l =d^2/\epsilon^2$, and 
	$\delta=\epsilon/d$, the error term is at most 
	$\mathcal{O}(\epsilon)$. 
	The total number of evaluations is at most $\mathcal{O}(d^5/\epsilon^5)$.
\end{remark} 

We note that in \cref{alg:zero_frank_wolfe_discrete}, 
$\bar{f}^+_{t,i}$ is the unbiased estimation of $F(y^+_{t,i})$, and the same 
holds for $\bar{f}^-_{t,i}$ and $F(y^-_{t,i})$. As a result, we can 
analyze the algorithm under the framework of stochastic continuous 
submodular maximization. By applying \cref{thm:zero_stochatic}, 
\cref{lem:discrete_to_continuous}, and the facts $\expect[|\bar{f}^+_{t,i}- 
F(y^+_{t,i})|^2] \le M^2/S_{t,i}, \expect[|\bar{f}^-_{t,i}- 
F(y^-_{t,i})|^2] \le M^2/S_{t,i}$ directly, we can also attain 
\cref{thm:zero_discrete}.
\section{Experiments}\label{sec:experiment-black}

\newcommand{\AlgSCG}{\texttt{Stochastic Continuous Greedy}\xspace}
\newcommand{\AlgZGA}{\texttt{Zeroth-Order Gradient Ascent}\xspace}
\newcommand{\AlgGA}{\texttt{Gradient Ascent}\xspace}

In this section, we will compare \Algblack (BCG) and \Algdis (DBG) with the 
following
baselines: 

(1) \AlgZGA (ZGA) is the projected gradient ascent algorithm 
equipped with the same two-point gradient estimator as BCG uses. Therefore, it 
is a \emph{zeroth-order} projected algorithm. 

(2) \AlgSCG (SCG) is the 
state-of-the-art \emph{first-order} algorithm for maximizing continuous 
DR-submodular functions~\cite{mokhtari2018conditional,mokhtari2018stochastic}. 
Note that it is a projection-free algorithm. 

(3) \AlgGA (GA) is the 
\emph{first-order} projected gradient ascent 
algorithm~\cite{hassani2017gradient}.

The stopping criterion for the algorithms is whenever a given 
number of iterations 
is achieved. Moreover, the batch sizes
$S_{t,i}$ in \cref{alg:zero_frank_wolfe} and $B_t$ in 
\cref{alg:zero_frank_wolfe_discrete} are both 1. Therefore, in the experiments, 
DBG uses 1 query per iteration while SCG uses $\mathcal{O}(d)$ queries.

We perform four sets of experiments which are described in detail in the following. The first two sets of 
experiments are maximization of continuous DR-submodular functions, which 
\Algblack 
is designed to solve. The 
last two are submodular set maximization problems. We will apply \Algdis to solve 
these problems. 
The function values at 
different rounds and the execution times are presented in 
\cref{fig:function,fig:time}. The first-order algorithms (SCG and GA) are 
marked in \textcolor{orange}{orange}, and the zeroth-order algorithms are 
marked in \textcolor{RoyalBlue}{blue}.  

\begin{figure*}[tb]
	\centering
	\begin{subfigure}[t]{0.24\textwidth}
		\includegraphics[width=\textwidth]{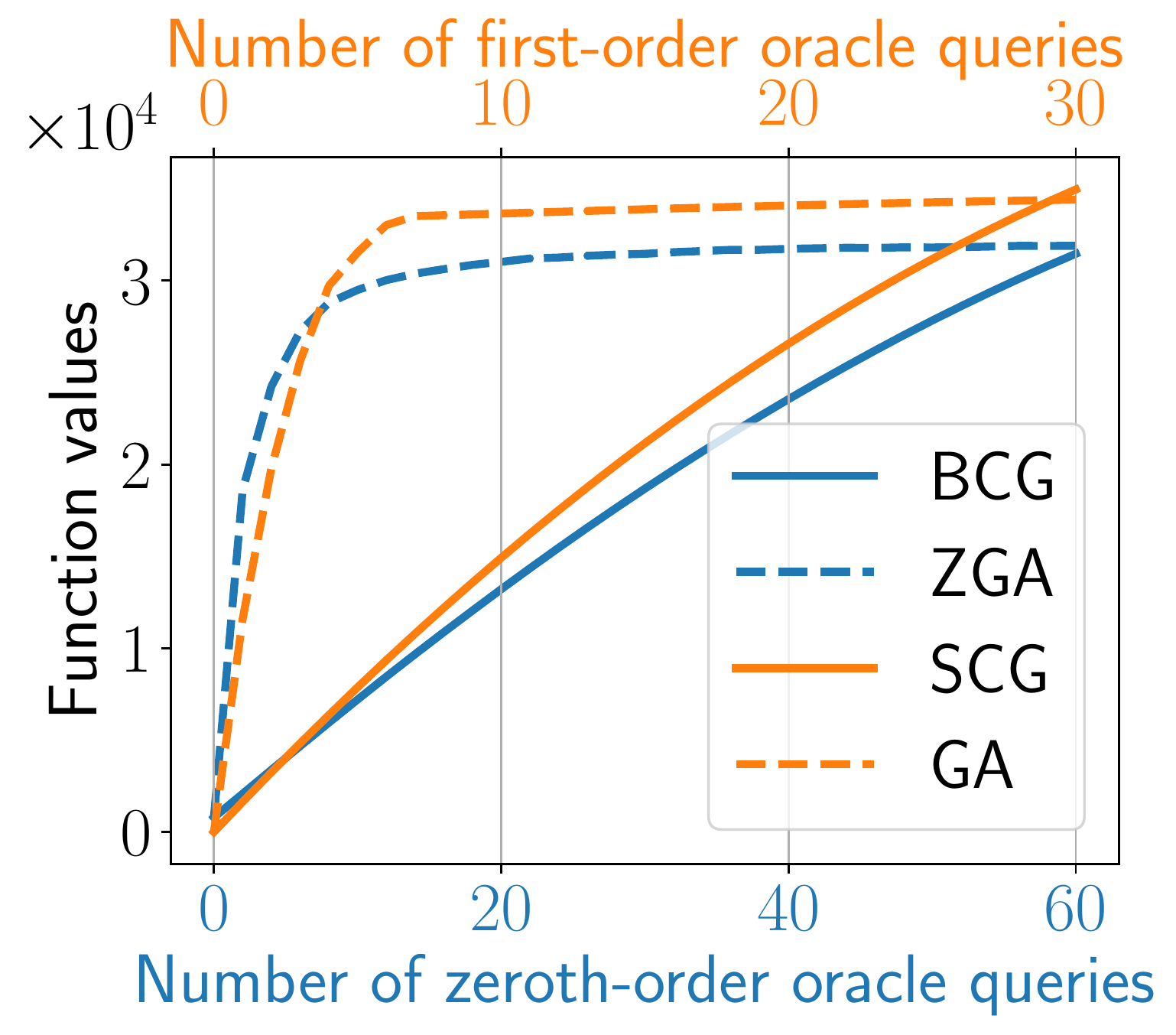}
		\caption{NQP}
		\label{fig:nqp_function}
	\end{subfigure}
	\begin{subfigure}[t]{0.24\textwidth}
		\includegraphics[width=\textwidth]{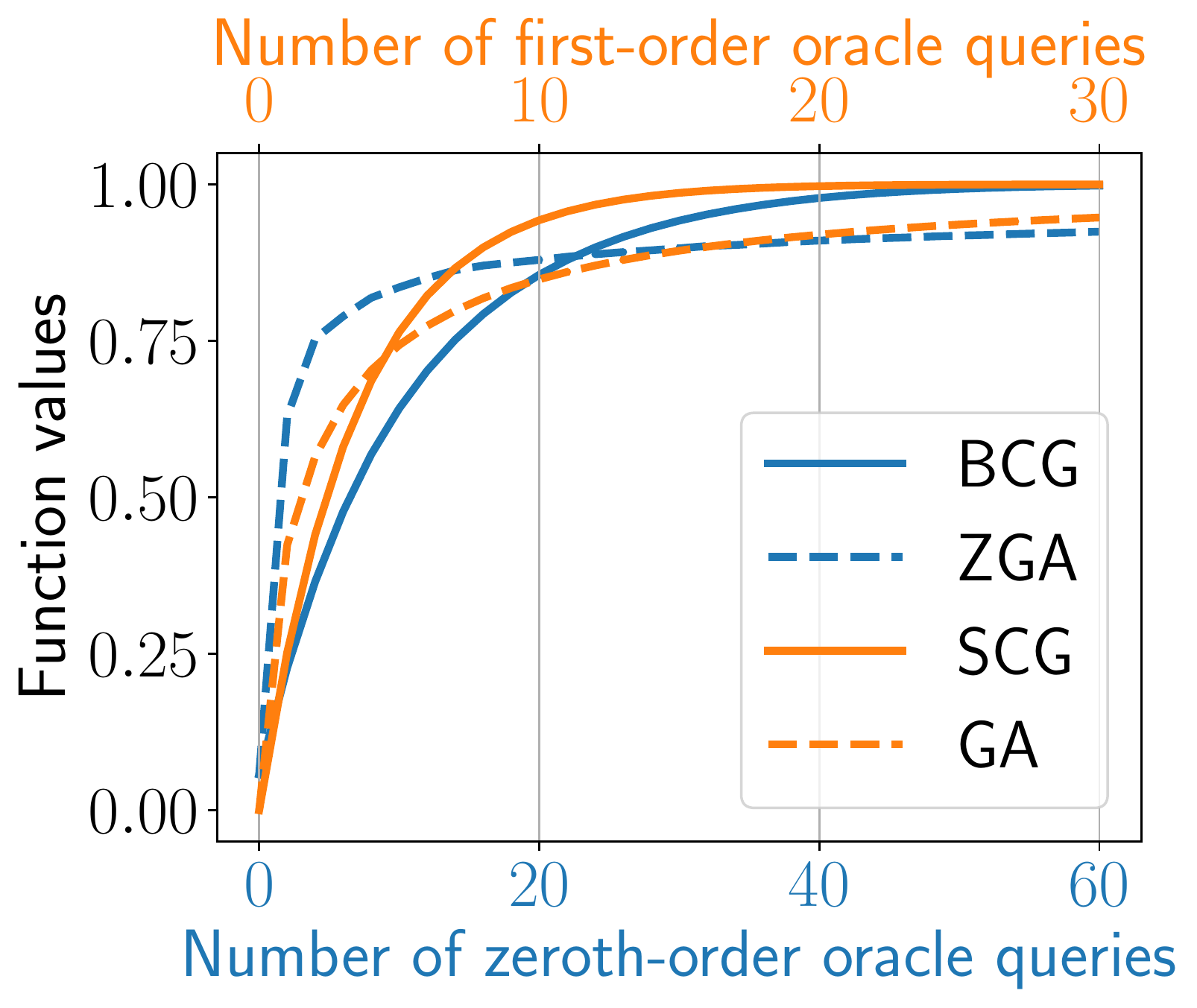}
		\caption{Topic summarization}
		\label{fig:topic_function}
	\end{subfigure}
	\begin{subfigure}[t]{0.24\textwidth}
		\includegraphics[width=\textwidth]{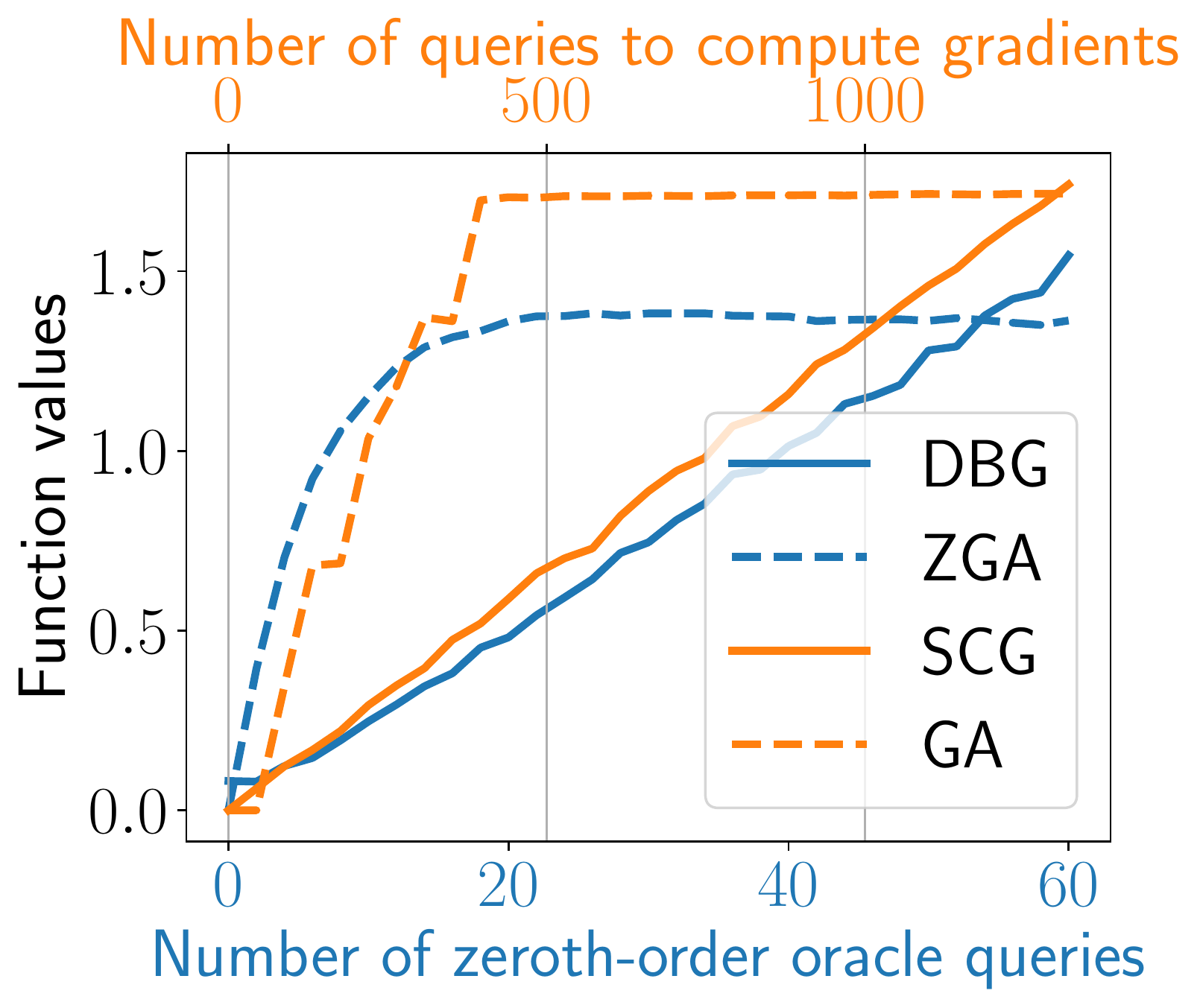}
		\caption{Active set selection}
		\label{fig:active_function}
	\end{subfigure}
	\begin{subfigure}[t]{0.24\textwidth}
		\includegraphics[width=\textwidth]{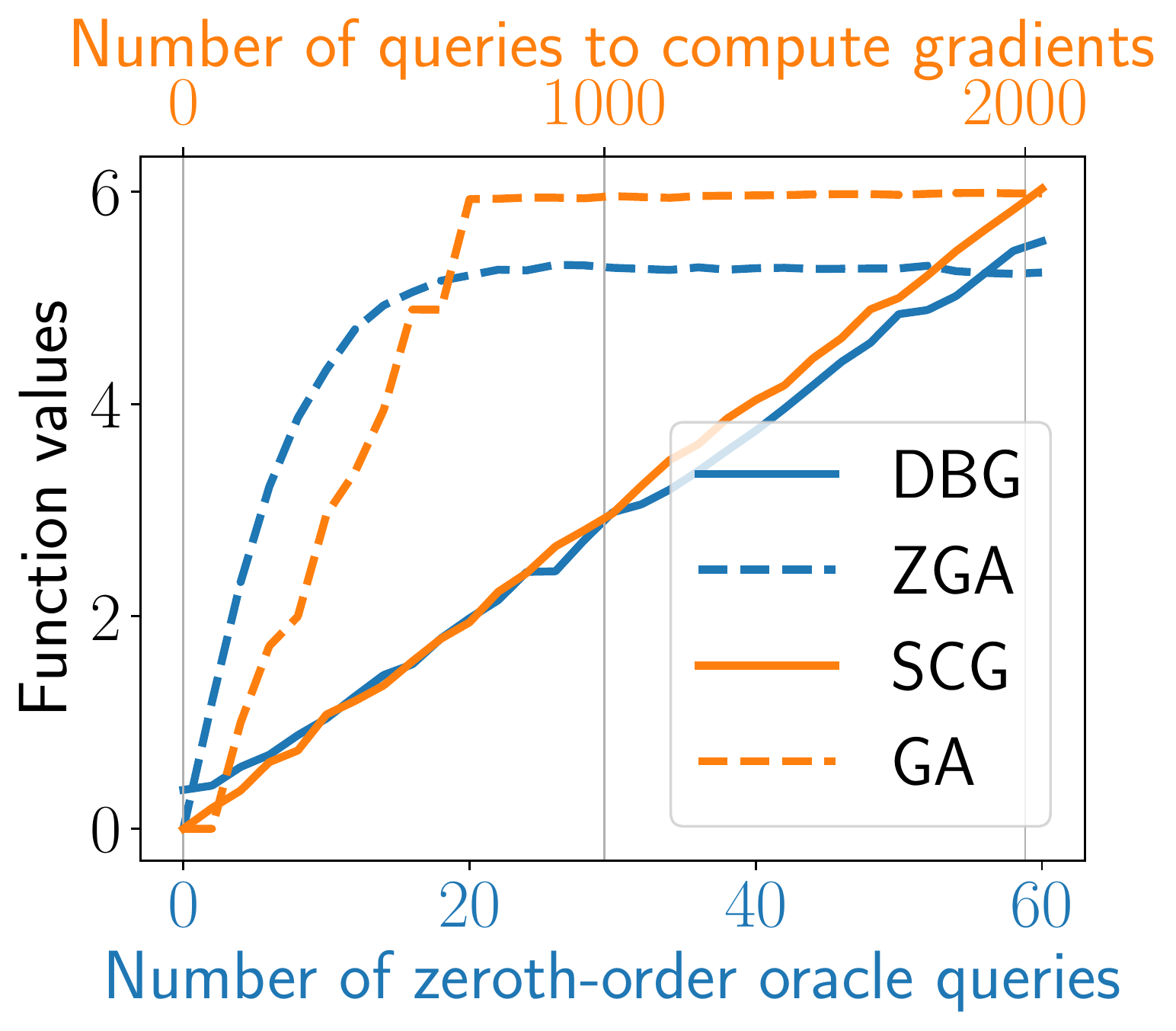}
		\caption{Influence maximization}
		\label{fig:cover_function}
	\end{subfigure}
	\vspace{-0.2cm}
	\caption{Function value vs.\ number of oracle queries. Note that every 
		chart has dual horizontal axes. 
		\textcolor{orange}{Orange} lines use the \textcolor{orange}{orange} 
		horizontal axes above while \textcolor{RoyalBlue}{blue} lines use the 
		\textcolor{RoyalBlue}{blue} ones below. }\label{fig:function}
	\vspace{-0.2cm}
\end{figure*}

\begin{figure*}[tb]
	\centering
	\begin{subfigure}[t]{0.23\textwidth}
		\includegraphics[width=\textwidth]{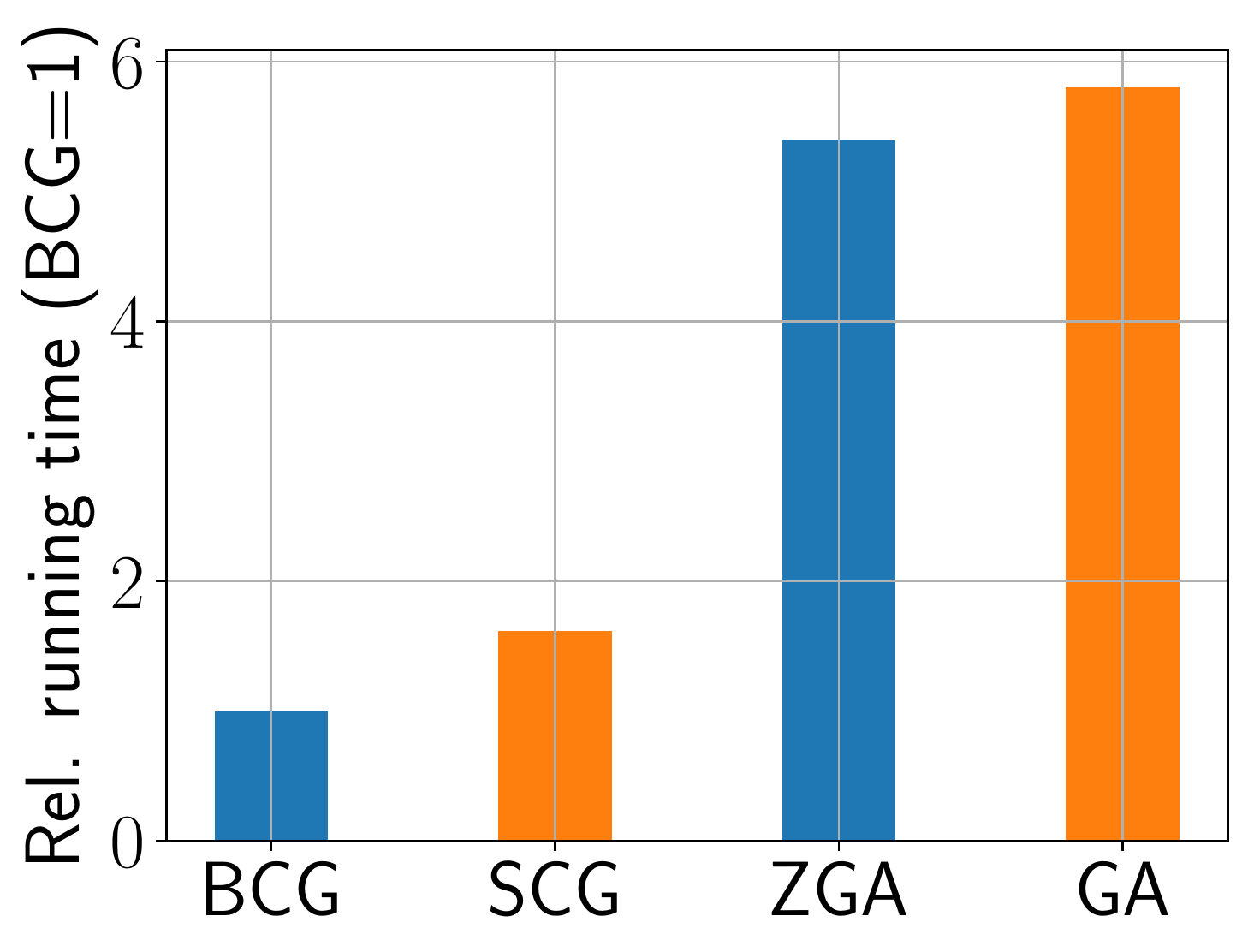}
		\caption{NQP}
		\label{fig:nqp_time}
	\end{subfigure}
	\begin{subfigure}[t]{0.23\textwidth}
		\includegraphics[width=\textwidth]{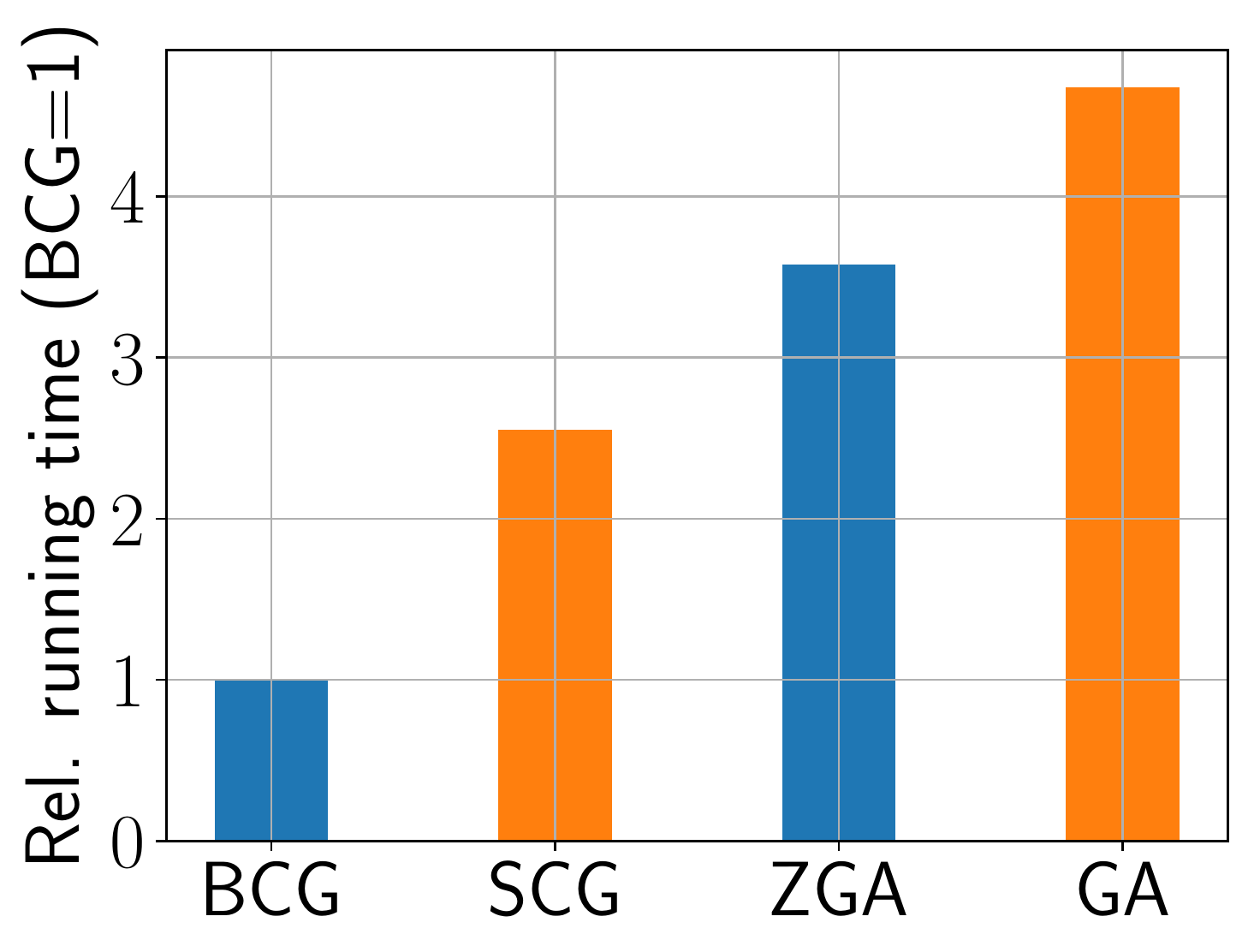}
		\caption{Topic summarization}
		\label{fig:topic_time}
	\end{subfigure}
	\begin{subfigure}[t]{0.245\textwidth}
		\includegraphics[width=\textwidth]{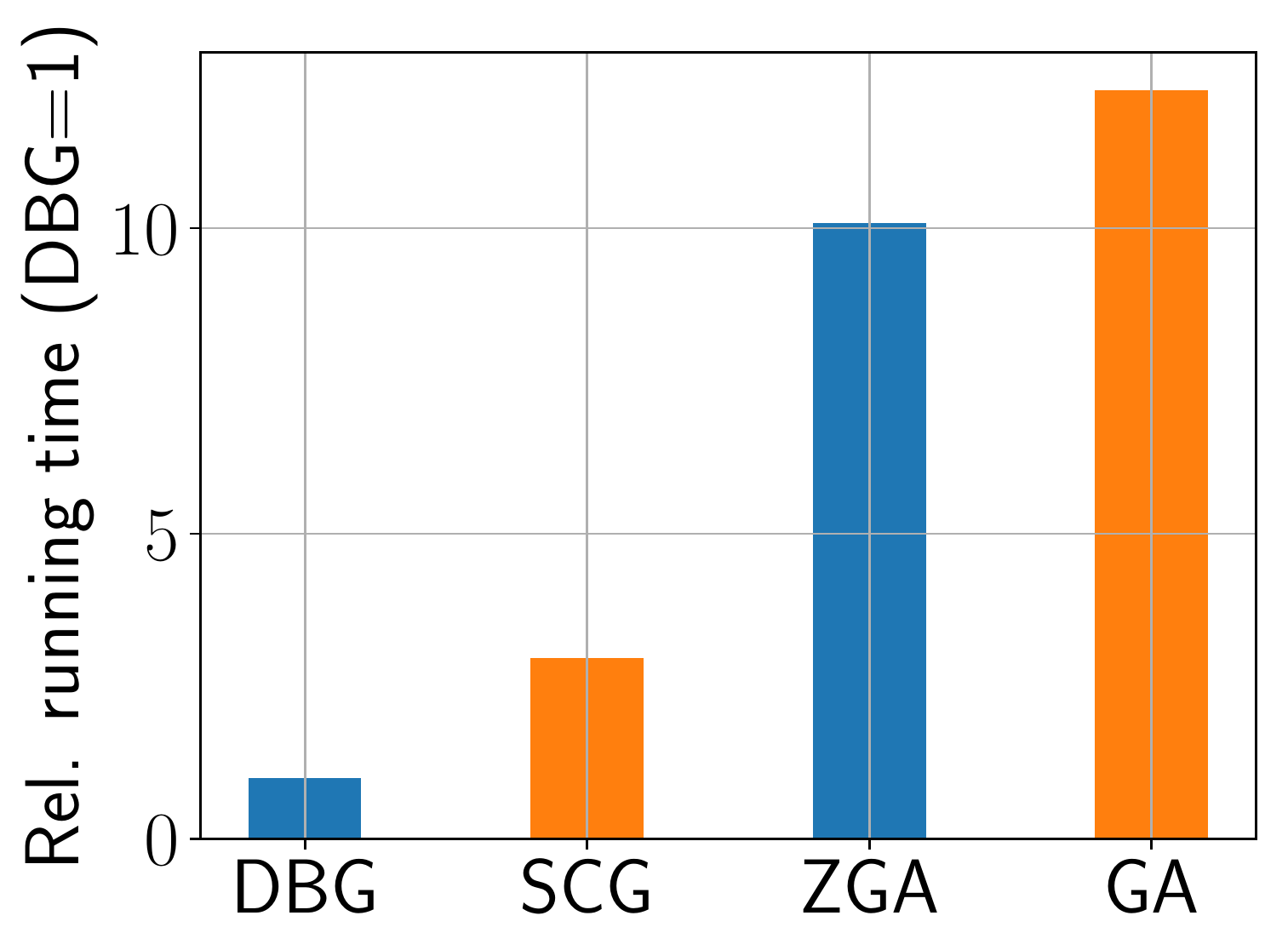}
		\caption{Active set selection}
		\label{fig:active_time}
	\end{subfigure}
	\begin{subfigure}[t]{0.23\textwidth}
		\includegraphics[width=\textwidth]{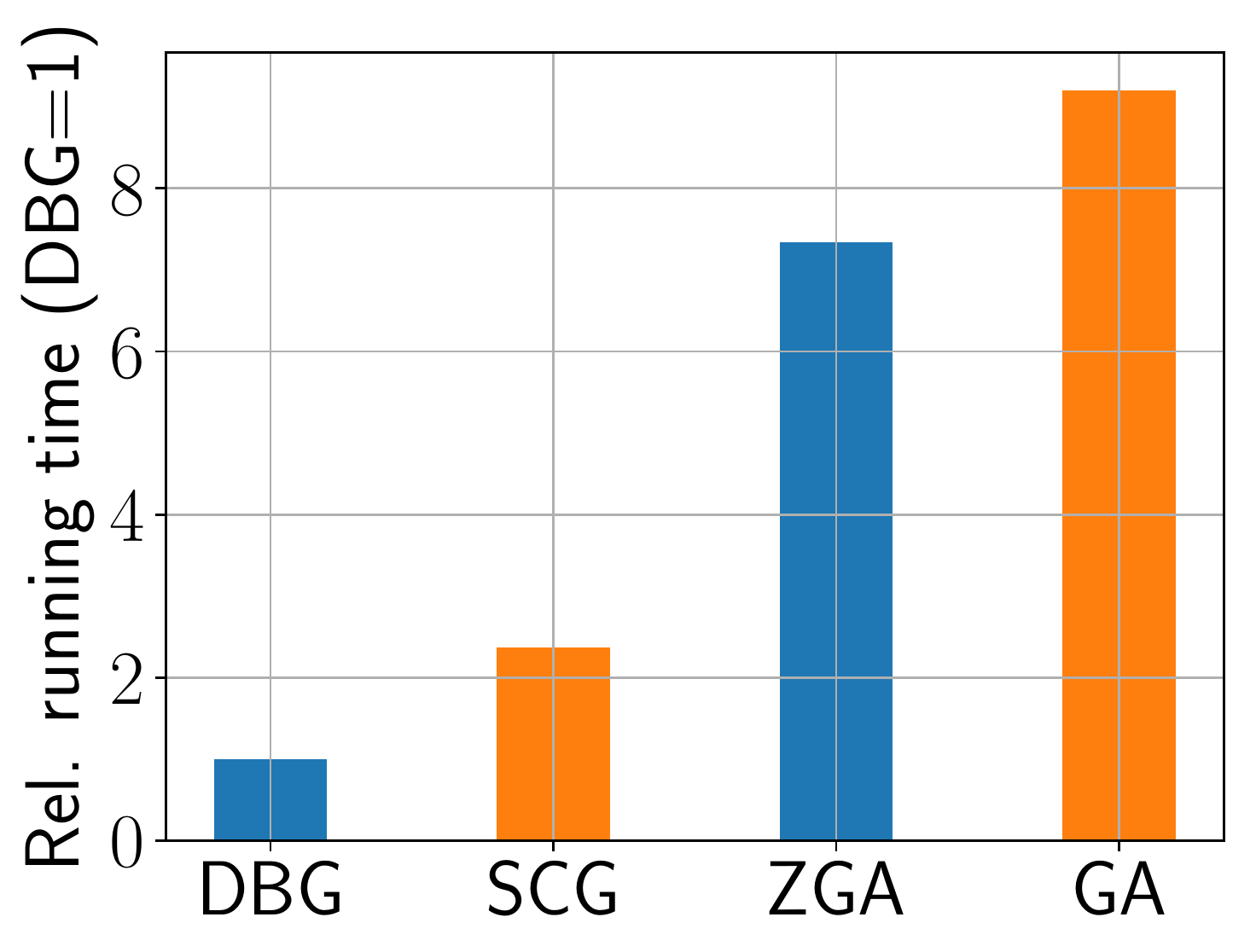}
		\caption{Influence maximization}
		\label{fig:cover_time}
	\end{subfigure}
	\vspace{-0.2cm}
	\caption{Relative running time normalized with respect to BCG (for 
		continuous 
		DR-submodular maximization in the first two sets of experiments) and 
		DBG 
		(for submodular set maximization in the last two sets of experiments). 
	}\label{fig:time}
	\vspace{-0.2cm}
\end{figure*}

\textbf{Non-convex/non-concave Quadratic Programming (NQP):}
In this set of experiments, we apply our proposed algorithm and the baselines 
to the problem of non-convex/non-concave quadratic programming. The objective 
function is of the form $F(x) = \frac{1}{2}x^\top H x + b^\top x$, 
where $ x $ is a 100-dimensional vector, $ H $ is a $ 100$-by-$ 100 $ matrix, 
and every component of $H$ is an i.i.d.\ random variable whose distribution 
is equal to that of the negated absolute value of a standard normal 
distribution. 
The constraints are $ \sum_{i=1}^{30}x_i \le 30 $, $ \sum_{i=31}^{60} x_i \le 20 
$, and $ \sum_{i=61}^{100} x_i \le 20 $. To 
guarantee that the gradient is non-negative, we set $b_t=-H^\top \mathbf{1}$. 
 One can observe from \cref{fig:nqp_function} that 
the 
function 
value that BCG attains is only slightly lower than that of the first-order 
algorithm SCG. The final 
function value that BCG attains is similar to that of ZGA.

\textbf{Topic Summarization:} Next, we consider the topic 
summarization problem~\citep{el2009turning,yue2011linear}, which is to maximize 
the probabilistic coverage of selected articles on news topics. Each news 
article is characterized by its topic distribution, which is obtained by 
applying latent Dirichlet allocation to the corpus of Reuters-21578, 
Distribution 1.0. The number of topics is set to 10. We will choose from 120 
news articles. The probabilistic coverage of a subset of news articles (denoted 
by $ X $) is defined by
$f(X)=\frac{1}{10}\sum_{j=1}^{10}[1-\prod_{a \in X}(1-p_a(j))]$, where 
$p_a(\cdot)$ is the topic distribution of article 
$a$. 
The multilinear extension function of $f$ is 
$F(x)=\frac{1}{10}\sum_{j=1}^{10}[1-\prod_{a \in 
\Omega}(1-p_a(j)x_a)]$, where 
$x \in [0,1]^{120}$~\cite{Iyer2014Monotone}. The constraint is
$\sum_{i=1}^{40}x_i\leq 
25$, $\sum_{i=41}^{80}x_i\leq 30$, $\sum_{i=81}^{120}x_i\leq 35$. 
It can be observed from \cref{fig:topic_function} that the proposed BCG 
algorithm achieves the same function value as the first-ordered algorithm SCG 
and outperforms the other two. 
As shown in \cref{fig:nqp_time}, BCG is the most efficient method. The two 
projection-free algorithms BCG and SCG run faster than the projected methods 
ZGA and GA. We will elaborate on the running time later in this section.

\textbf{Active Set Selection:}
We study the active set selection problem that arises in Gaussian process 
regression~\cite{mirzasoleiman2013distributed}. 
We use the \emph{Parkinsons
Telemonitoring} dataset, which is composed of biomedical voice measurements 
from people with early-stage Parkinson's disease \citep{tsanas2010enhanced}.
Let $ X \in\mathbb{R}^{n\times d} $ denote the data matrix. Each row  $ 
X[i,:] $ is a voice recording while each column $ X[:,j] $ denotes an 
attribute. The covariance matrix $ \Sigma $ is 
defined by
 $ \Sigma_{ij} = \exp(-\| X[:,i]-X[:,j] \|^2)/h^2 $, where $ h $ is set to $ 
 0.75 $. The objective function of the active set selection problem is defined 
 by
 $ f(S) = \log\det(I+\Sigma_{S,S}) $, where $ S\subseteq [d] $ and $ 
 \Sigma_{S,S} $ is the principal submatrix indexed by $ S $. The total number 
 of 22 attributes are 
 partitioned into 5 disjoint subsets with sizes 4, 4, 4, 5 and 5, respectively. 
 The problem is 
 subject to a partition matroid requiring that at most one attribute should be 
 active within each subset.
 Since this is a submodular set maximization 
 problem, in order to evaluate the gradient (\ie, obtain an unbiased estimate 
 of gradient) required by first-order algorithms SCG and GA, it needs $ 2d $ 
 function value queries. To be precise, the $ i $-th component of gradient is $ 
 \expect_{S\sim x}[f(S\cup\{i\})-f(S)] $ and requires two function value 
 queries. It can be observed 
 from \cref{fig:active_function} that DBG outperforms the other zeroth-order 
 algorithm ZGA. Although its performance is slightly worse than the two 
 first-order algorithms SCG and GA, it require significantly less number of function value 
 queries than the other two first-order methods (as discussed above).
 
 \textbf{Influence Maximization:}
In the influence maximization problem, we assume that every node in the network is able to influence all  of its one-hop 
neighbors.   
The objective of influence maximization is to select a subset of nodes in the network, called the seed set 
(and denoted by $ S $), so that the total number of  influenced nodes, 
including the seed nodes, is maximized. We choose the social network of 
Zachary's karate club~\cite{zachary1977information} in this study. The subjects 
in this social network are partitioned into three disjoint groups, whose sizes 
are 10, 14, and 10 respectively.  The chosen seed nodes should be subject to a 
partition matroid; \ie, we will select at most two subjects from each of the 
three groups. Note that this problem is also a submodular set maximization 
problem. Similar to the situation in the active set selection problem, 
first-order algorithms need function value queries to obtain an unbiased 
estimate of gradient. We can observe from \cref{fig:cover_function} that DBG 
attains a better influence coverage than the other zeroth-order algorithm ZGA. 
Again, even though SCG and GA achieve a slightly better coverage, due to their 
first-order nature, they require a significantly larger number of function 
value queries.
%
%

\paragraph{Running Time}
The running times of the our proposed algorithms and the baselines   are
presented in \cref{fig:time} for the above-mentioned experimental set-ups. There are two main conclusions. First, the two 
projection-based algorithms (ZGA and GA) require significantly higher time complexity compared to the projection-free algorithms (BCG, DBG, and SCG), as the projection-based algorithms require 
solving quadratic optimization problems whereas projection-free ones 
require solving linear optimization problems which can be solved  more 
efficiently. Second, when we compare first-order and zeroth-order algorithms, we can 
observe that zeroth-order algorithms (BCG, DBG, and ZGA) run faster than their 
first-order counterparts (SCG and GA).
\paragraph{Summary}
\revise{The above experiment results show the following major 
advantages of our method over the baselines including SCG and ZGA.} 
\begin{enumerate}
\item BCG/DBG is at least twice faster than SCG and ZGA in all 
tasks in terms of running time
(\cref{fig:nqp_time,fig:topic_time,fig:active_time,fig:cover_time}).
\item DBG requires remarkably fewer function evaluations in the discrete
setting (\cref{fig:active_function,fig:cover_function}).
\item In addition to saving function evaluations, BCG/DBG achieves an objective 
function value comparable to that of the first-order baselines SCG and GA. 
\end{enumerate}

\revise{Furthermore, we note that the number of first-order queries 
required by SCG 
is only half the number required by BCG. However, as is shown in 
\cref{fig:nqp_time,fig:topic_time}, BCG runs significantly faster than SCG 
since a zeroth-order evaluation is faster than a first-order one.} 

\revise{In the topic summarization task (\cref{fig:topic_function}), 
BCG exhibits a similar performance to that of the first-order baselines SCG and 
GA, in terms of the attained objective function value. In the other three 
tasks, BCG/DBG runs notably faster while achieving an only slightly inferior 
function value. 
Therefore, BCG/DBG is  particularly 
preferable  in a large-scale machine learning task and an application where the 
total number of function evaluations or 
the running time is subject to a budget. 

\section{Conclusion}
In this chapter, we presented \Algblack, a derivative-free and projection-free  
algorithm  for maximizing a 
monotone and  continuous 
DR-submodular function subject to a general convex body constraint.  We showed 
that \Algblack achieves 
the tight 
$[(1-1/e)OPT-\epsilon]$ 
approximation guarantee with $\mathcal{O}(d/\epsilon^3)$ function evaluations. 
We then 
extended the algorithm to the stochastic continuous setting and the discrete 
submodular 
maximization 
problem. Our experiments on both synthetic and real data validated the 
performance of our proposed algorithms. In particular, we observed that 
\Algblack 
practically achieves the 
same utility as \AlgCG while being way more efficient in terms of number of 
function evaluations. 
}
\section{Proofs}\label{sec:proof-black}
\subsection{Proof of Lemma 10}\label{app:proof_lemma_smooth_approx}
\begin{proof}
	Using 
	the assumption that
	$F$ is $G$-Lipschitz continuous, we have
	\begin{align*}
	|\tF(x)-\tF(y)| ={}& |\expect_{v\sim B^d} [F(x+\delta v) - F(y+\delta v)]| 
	\\
	\le{}& \expect_{v\sim B^d}[ |F(x+\delta v) - F(y+\delta v)|] \\
	\le{}& \expect_{v\sim B^d}[G\|(x+\delta v)-(y+\delta v)\|] \\
	={}& G\|x-y\|,
	\end{align*}
	and
	\begin{align*}
	|\tilde{F}(x)-F(x)| ={}& | \expect_{v\sim B^d}[F(x+\delta v)-F(x)] | \\
	\leq{}& \expect_{v\sim B^d} [|F(x+\delta v)-F(x)|] \\
	\leq{}& \expect_{v\sim B^d}[G \delta \|v \|] \\
	\le{}& \delta G.
	\end{align*}
	
	If $F$ is $G$-Lipschitz continuous and monotone continuous DR-submodular, 
	then 
	$F$ is differentiable. For all $\ x \leq y$, we also have
	\begin{equation*}
	\nabla F(x) \geq \nabla F(y),   
	\end{equation*}
	and
	\begin{equation*}
	F(x) \leq F(y).    
	\end{equation*}
	
	By definition of $\tilde{F}$, we have $\tilde{F}$ is differentiable and for 
	$\forall\ x \leq y$,
	\begin{align*}
	\nabla \tilde{F}(x) - \nabla \tilde{F}(y) ={} & \nabla \expect_{v\sim B^d} 
	[F(x+\delta v)]-  \nabla \expect_{v\sim B^d} 
	[F(y+\delta v)] \\
	={} & \expect_{v\sim B^d}[\nabla F(x+\delta v) - \nabla F(y+\delta v) ] \\
	\geq{} & \expect_{v\sim B^d} [0] \\
	={}& 0,
	\end{align*}
	and
	\begin{align*}
	\tF(x)-\tF(y) ={}& \expect_{v \sim B^d}[F(x+\delta v)]-\expect_{v \sim 
		B^d}[F(y+\delta v)] 
	\\
	={}& \expect_{v \sim B^d}[F(x+\delta v)-F(y+\delta v)] \\
	\leq{} & \expect_{v \sim B^d}[0] \\
	={}& 0,
	\end{align*}
	\emph{i.e.}, $\nabla \tilde{F}(x) \geq \nabla \tilde{F}(y), \tF(x) \leq 
	\tF(y).$ So $\tF$ is also a monotone continuous DR-submodular function.
\end{proof}


	
	

\subsection{Proof of Theorem 9}\label{app:theorem_zero}
In order to prove \cref{thm:zero}, we need the following variance reduction 
lemmas~\citep{shamir2017optimal,pmlr-v80-chen18c}, where the second one is a 
slight improvement of Lemma~2 in~\citep{mokhtari2018conditional} and Lemma~5 
in~\citep{mokhtari2018stochastic}. 

\begin{lemma}[Lemma~10 of \citep{shamir2017optimal}]\label{lem:variance}
	It holds that
	\begin{equation*}
	\expect_{u \sim S^{d-1}}[\frac{d}{2\delta}(F(z+\delta u)-F(z-\delta u))u| 
	z]  = \nabla \tilde{F}(z) ,  
	\end{equation*}
	\begin{equation*}
	\expect_{u \sim S^{d-1}}[\|\frac{d}{2\delta}(F(z+\delta u)-F(z-\delta 
	u))u - \nabla \tF(z) \|^2|z]   \le cdG^2,       
	\end{equation*}
	where $c$ is a constant.
\end{lemma}

\begin{lemma}[Theorem~3 of 
	\citep{pmlr-v80-chen18c}]\label{lem:variance_reduction}
	Let $ \{ 
	a_t\}_{t=0}^{T}$ be a sequence of points in $\mathbb{R}^n$ 
	such that $ \| a_t - 
	a_{t-1} \| 
	\leq G_0/(t+s)  $ for all $1\leq t\leq T $ with fixed constants $ G_0 \geq 
	0 
	$ and $ s\geq 3 $. 
	Let $ \{ \tilde{a}_t\}_{t=1}^T$ be a sequence of random variables such that 
	$ \expect[ 
	\tilde{a}_t|\mathcal{F}_{t-1} ] = a_t $ and $ \expect[ \| \tilde{a}_t - a_t 
	\|^2|\mathcal{F}_{t-1} ] \leq \sigma^2$ for 
	every $ t\geq 0 $, 
	where 
	$\mathcal{F}_{t-1} $ is the $ \sigma $-field generated by 
	$ \{ \tilde{a}_i\}_{i=1}^{t} $ 
	and $ \mathcal{F}_{0} = \varnothing $. Let $\{d_t\}_{t=0}^T$ be a sequence 
	of random 
	variables where $d_0$ is fixed and subsequent $d_{t}$ are obtained by the 
	recurrence 
	\begin{equation*}
	d_t = (1-\rho_t) d_{t-1} +\rho_t \tilde{a}_t
	\end{equation*}
	with $ \rho_t = \frac{2}{(t+s)^{2/3}} $. 
	Then, we have
	\begin{equation*}
	\expect[\| a_t-d_t\|^2 ] \leq \frac{Q}{(t+s+1)^{2/3}},
	\end{equation*}
	where $ Q \triangleq \max \{ \|a_0 - d_0 \|^2 (s+1)^{2/3}, 
	4\sigma^2 + 3G_0^2/2 \} $.
\end{lemma}

Now we turn to prove \cref{thm:zero}.

\begin{proof}[Proof of \cref{thm:zero}]
	First of all, we note that technically we need the iteration number $T \geq 
	4$, which always holds in practical applications.
	
	Then we show that $\forall\ t=1,\dots, T+1$, $x_t\in \domainsh$. By 
	the definition of $x_t$, we have $x_t=\sum_{i=1}^{t-1}\frac{v_i}{T}$. Since 
	$v_t$'s are non-negative vectors, we know that $x_t$'s are also 
	non-negative vectors and that $0=x_1\le x_2 \le \dots \le x_{T+1}$. It 
	suffices to show that $x_{T+1}\in \domainsh$. Since $x_{T+1}$ is a convex 
	combination of $v_1,\dots, v_T$ and $v_t$'s are in $\domainsh$, we conclude 
	that $x_{T+1}\in \domainsh$. In addition, since $v_t$'s are also in 
	$\constraint-\delta \one$, $x_{T+1}$ is also in $\constraint-\delta\one$. 
	Therefore our final choice $x_{T+1}+\delta \one$ resides in the constraint 
	$\constraint$.
	
	Let $z_t\triangleq x_t+\delta \one$ and the shrunk domain (without 
	translation) 
	$\domainsh'\triangleq \domainsh + \delta \one = \prod_{i=1}^d [\delta, 
	a_i-\delta]\subseteq \domain$. 
	By Jensen's inequality and the fact $F$ has $L$-Lipschitz continuous 
	gradients, we have
	\begin{equation*}
	\| \nabla \tF(x) -\nabla \tF(y) \| \le L\|x-y\|.    
	\end{equation*}
	
	Thus,
	\begin{equation}\label{eq:gv}
	\begin{split}
	\tF(z_{t+1})-\tF(z_t)={} & \tF(z_t+\frac{v_t}{T})-\tF(z_t)\\
	\ge{} & \frac{1}{T} \nabla \tF(z_t)^\top v_t - \frac{L}{2 T^2} \| v_t \|^2 
	\\
	\ge{} & \frac{1}{T} \nabla \tF(z_t)^\top v_t - \frac{L}{2 T^2} D_1^2 \\
	={} & \frac{1}{T}\left(\bar{g}_t^\top v_t + (\nabla 
	\tF(z_t)-\bar{g}_t)^\top 
	v_t \right)  - \frac{L}{2 T^2} D_1^2.
	\end{split}
	\end{equation}

	Let $x^*_\delta \triangleq \argmax_{x\in \domainsh'\cap \constraint} 
	\tF(x)$.
	Since $x^*_\delta,z_t\in \domainsh'$, we have $v_t^*\triangleq (x^*_\delta 
	- z_t)\vee 0 \in \domainsh$. We know  $z_t+v_t^* = x_\delta^*\vee z_t\in 
	\domainsh'$ and 
	\begin{equation*}
	v_t^*+\delta \one =(x_\delta^*-x_t)\vee \delta \one \le x_\delta^*.
	\end{equation*}
	
	Since we assume that $F$ is monotone continuous DR-submodular, by 
	\cref{lem:smooth_approx}, $\tF$ is also monotone continuous DR-submodular. 
	As a result, $\tF$ is concave along non-negative directions, and $\nabla 
	\tF$ is 
	entry-wise non-negative. Thus we have
	\begin{align*}
	\tF(z_t+v_t^*)-\tF(z_t)\le{} & \nabla \tF(z_t)^\top v_t^*\\
	\le{} & \nabla \tF(z_t)^\top (x_\delta^*-\delta\one).
	\end{align*}
	
	Since $x_\delta^*-\delta\one\in \constraint'$, we deduce \begin{align*}
	\bar{g}_t^\top v_t \ge{} & \bar{g}_t^\top (x_\delta^*-\delta\one)\\
	={} & \nabla \tF(z_t)^\top (x_\delta^*-\delta\one) + (\bar{g}_t-\nabla 
	\tF(z_t))^\top (x_\delta^*-\delta\one)\\
	\ge{} & \tF(z_t+v_t^*)-\tF(z_t) + (\bar{g}_t-\nabla \tF(z_t))^\top 
	(x_\delta^*-\delta\one)\\
	\ge{} & \tF(x_\delta^*)-\tF(z_t) + (\bar{g}_t-\nabla \tF(z_t))^\top 
	(x_\delta^*-\delta\one).
	\end{align*}
	
	Therefore, we obtain
	\begin{align}
	\bar{g}_t^\top v_t + (\nabla \tF(z_t)-\bar{g}_t)^\top v_t \ge  
	\tF(x_\delta^*)-\tF(z_t) + (\nabla \tF(z_t)-\bar{g}_t)^\top 
	(v_t-(x_\delta^*-\delta\one)).\label{eq:gv2}
	\end{align}
	
	By plugging \cref{eq:gv2} into \cref{eq:gv}, after re-arrangement of the 
	terms, we obtain
	\begin{align}\label{eq:h_t}
	h_{t+1} \le (1-\frac{1}{T})h_t + \frac{1}{T}(\nabla 
	\tF(z_t)-\bar{g}_t)^\top ((x_\delta^*-\delta\one)-v_t)
	+\frac{L}{2 T^2} D_1^2,
	\end{align}
	where $h_t\triangleq \tF(x_\delta^*)-\tF(z_t)$. Next we derive an upper 
	bound for $(\nabla \tF(z_t)-\bar{g}_t)^\top ((x_\delta^*-\delta\one)-v_t)$. 
	By Young's inequality, it can be deduced that for any $\beta_t>0$,
	\begin{equation}\label{eq:inner_product}
	\begin{split}
		(\nabla \tF(z_t)-\bar{g}_t)^\top ((x_\delta^*-\delta\one)-v_t) \le{} & 
	\frac{\beta_t}{2}\| \nabla \tF(z_t)-\bar{g}_t\|^2 + 
	\frac{1}{2\beta_t} \|(x_\delta^*-\delta\one)-v_t \|^2 \\
	\le{} & \frac{\beta_t}{2}\| \nabla \tF(z_t)-\bar{g}_t\|^2 + 
	\frac{1}{2\beta_t} D_1^2. 
	\end{split}
	\end{equation}


	Now let $\mathcal{F}_1 \triangleq \varnothing$ and $\mathcal{F}_t$ be the 
	$\sigma$-field generate by $\{\bar{g}_1,
	\dots,\bar{g}_{t-1} \}$, then by \cref{lem:variance}, we have
	\begin{equation*}
	\expect[\frac{d}{2\delta}(F(y_{t,i}^+)-F(y_{t,i}^-))u_{t,i}|\mathcal{F}_{t-1}]
	= \nabla \tF(z_t) ,  
	\end{equation*}
	and
	\begin{equation*}
	\expect[\|\frac{d}{2\delta}(F(y_{t,i}^+)-F(y_{t,i}^-))u_{t,i}-\nabla 
	\tF(z_t) \|^2|\mathcal{F}_{t-1}]  \leq cdG^2.  
	\end{equation*}
	
	Therefore,
	\begin{align*}
	\expect[g_t|\mathcal{F}_{t-1}] 
	={}&\expect[\frac{1}{B_t}\sum_{i=1}^{B_t}\frac{d}{2\delta}(F(y_{t,i}^+)-F(y_{t,i}^-))u_{t,i}|\mathcal{F}_{t-1}]
	\\
	={}&\nabla \tF(z_t) , 
	\end{align*}
	and
	\begin{equation}\label{eq:variance_upper}
	\begin{split}
	\expect[\|g_t - \nabla \tF(z_t) \|^2|\mathcal{F}_{t-1}] ={}&\frac{1}{B_t^2} 
	\sum_{i=1}^{B_t} 
	\expect[\|\frac{d}{2\delta}(F(y_{t,i}^+)-F(y_{t,i}^-))u_{t,i}-\nabla 
	\tF(z_t) \|^2|\mathcal{F}_{t-1}] \\
	\leq{}& \frac{cdG^2}{B_t}.  
	\end{split}
	\end{equation}

	
	By Jensen's inequality and the assumption $F$ is $L$-smooth, we have
	\begin{equation*}
		\|\nabla \tF(z_t)-\nabla\tF(z_{t-1}) \|\le L\frac{D_1}{T}\le 
		\frac{2LD_1}{t+3}.
	\end{equation*}
	
	Then by \cref{lem:variance_reduction} with $s=3, d_t = \bar{g}_t, \forall\ 
	t 
	\ge 
	0, \tilde{a}_t=g_t, a_t=\nabla \tF(z_t), \forall\ t \ge 1, a_0=\nabla 
	\tF(z_1), G_0=2LD_1$, we have
	\begin{equation}
	\expect[\| \nabla \tF(z_t)-\bar{g}_t\|^2]\le 
	\frac{Q}{(t+4)^{2/3}},\label{eq:bound_on_norm}
	\end{equation}
	where $Q\triangleq \max \{ \|\nabla \tF(x_1+\delta \one) \|^2 4^{2/3}, 
	\frac{4cdG^2}{B_t}+ 6L^2D_1^2 \} $.
	Note that by \cref{lem:smooth_approx}, we have $\| \nabla \tF(x) \| \le G$, 
	thus we can re-define $Q = \max \{ 4^{2/3}G^2, 
	\frac{4cdG^2}{B_t}+ 6L^2D_1^2 \}$.
	
	Using \cref{eq:h_t,eq:inner_product,eq:bound_on_norm} and taking 
	expectation, we obtain
	\begin{equation*}
	\begin{split}
	\expect[h_{t+1}]\le{}& (1-\frac{1}{T})\expect[h_t]+\frac{1}{T}\left( 
	\frac{\beta_t}{2}\cdot \frac{Q}{(t+4)^{2/3}} +\frac{D_1^2}{2\beta_t}
	\right) + \frac{L}{2 T^2}D_1^2 \\
	\le{}& (1-\frac{1}{T})\expect[h_t]+
	\frac{D_1Q^{1/2}}{T(t+4)^{1/3}}
	+ \frac{L}{2T^2}D_1^2,
	\end{split}
	\end{equation*}
	where we set $\beta_t = \frac{D_1(t+4)^{1/3}}{Q^{1/2}}$. Using the above 
	inequality recursively, we have
	\begin{align*}
	\expect[h_{T+1}]\le{} & (1-\frac{1}{T})^T (\tF(x_\delta^*)-\tF(\delta\one)) 
	+\sum_{t=1}^T \frac{D_1Q^{1/2}}{T(t+4)^{1/3}} +  \frac{L}{2 T}D_1^2\\
	\le{} & e^{-1}(\tF(x_\delta^*)-\tF(\delta\one)) + 
	\frac{D_1Q^{1/2}}{T}\int_0^T \frac{\dif x}{(x+4)^{1/3}} + \frac{L}{2 
		T}D_1^2\\
	\leq{} & e^{-1}(\tF(x_\delta^*)-\tF(\delta\one)) + \frac{D_1Q^{1/2}}{T} 
	\frac{3}{2}(T+4)^{2/3} + \frac{L}{2 T}D_1^2 \\
	\leq{} & e^{-1}(\tF(x_\delta^*)-\tF(\delta\one))
	+ \frac{D_1Q^{1/2}}{T} \frac{3}{2} (2T)^{2/3} + \frac{L}{2 
		T}D_1^2 \\
	\le{} & e^{-1}(\tF(x_\delta^*)-\tF(\delta\one)) + 
	\frac{3D_1Q^{1/2}}{T^{1/3}}+ \frac{LD_1^2}{2 T}.
	\end{align*}
	
	By re-arranging the terms, we conclude
	\begin{align*}
	(1-\frac{1}{e})\tF(x_\delta^*)-\expect[\tF(z_{T+1})] \le{} & 
	-e^{-1}\tF(\delta\one)+\frac{3D_1Q^{1/2}}{T^{1/3}}+ \frac{LD_1^2}{2 
		T}\\
	\le{} & \frac{3D_1Q^{1/2}}{T^{1/3}}+ \frac{LD_1^2}{2 T},
	\end{align*}
	where the second inequality holds since 
	the image of $F$ is in $\mathbb{R}_+$.
	
	By \cref{lem:smooth_approx}, we have $\tF(z_{T+1})\le F(z_{T+1})+\delta G$ 
	and 
	\begin{align*}
	\tF(x_{\delta}^*)\ge \tF(x^*)-\delta G\sqrt{d}\ge  F(x^*)-\delta 
	G(\sqrt{d}+1).
	\end{align*}
	
	Therefore,
	\begin{align*}
	(1-\frac{1}{e})F(x^*)-\expect[F(z_{T+1})] \le  
	\frac{3D_1Q^{1/2}}{T^{1/3}}+ \frac{LD_1^2}{2 T} 
	+ \delta G(1+(\sqrt{d}+1)(1-\frac{1}{e})).
	\end{align*}
\end{proof}

\subsection{Proof of Theorem 10}\label{app:theorem_zero_stochastic}
\begin{proof}
	By the unbiasedness of $\hat{F}$ and \cref{lem:variance}, we have
	\begin{align*}
	\expect[\frac{d}{2\delta}(\hat{F}(y_{t,i}^+)-\hat{F}(y_{t,i}^-))u_{t,i}|\mathcal{F}_{t-1}]
	={}& 
	\expect[\expect[\frac{d}{2\delta}(\hat{F}(y_{t,i}^+)-\hat{F}(y_{t,i}^-))u_{t,i}|\mathcal{F}_{t-1},u_{t,i}]|\mathcal{F}_{t-1}]
	\\
	={}& 
	\expect[\frac{d}{2\delta}(F(y_{t,i}^+)-F(y_{t,i}^-))u_{t,i}|\mathcal{F}_{t-1}]
	\\
	={}& \nabla \tF(z_t),
	\end{align*}
	where $z_t= x_t+\delta \one$, and
	\begin{align*}
	&\expect[\|\frac{d}{2\delta}(\hat{F}(y_{t,i}^+)-\hat{F}(y_{t,i}^-))u_{t,i}-\nabla
	\tF(z_t) \|^2|\mathcal{F}_{t-1}] \\
	={}& \expect[\expect[\| 
	\frac{d}{2\delta}(F(y_{t,i}^+)-F(y_{t,i}^-))u_{t,i}-\nabla \tF(z_t) \\
	&+\frac{d}{2\delta}(\hat{F}(y_{t,i}^+)-F(y_{t,i}^+))u_{t,i} \\
	&- \frac{d}{2\delta} 
	(\hat{F}(y_{t,i}^-)-F(y_{t,i}^-))u_{t,i}\|^2|\mathcal{F}_{t-1},u_{t,i}]|\mathcal{F}_{t-1}]
	\\
	={}& \expect[\expect[\| 
	\frac{d}{2\delta}(F(y_{t,i}^+)-F(y_{t,i}^-))u_{t,i}-\nabla 
	\tF(z_t)\|^2|\mathcal{F}_{t-1},u_{t,i}]|\mathcal{F}_{t-1}] \\
	&+ \expect[\expect[\| 
	\frac{d}{2\delta}(\hat{F}(y_{t,i}^+)-F(y_{t,i}^+))u_{t,i}\|^2|\mathcal{F}_{t-1},u_{t,i}]|\mathcal{F}_{t-1}]
	\\
	&+ \expect[\expect[\| \frac{d}{2\delta} 
	(\hat{F}(y_{t,i}^-)-F(y_{t,i}^-))u_{t,i}\|^2|\mathcal{F}_{t-1},u_{t,i}]|\mathcal{F}_{t-1}]
	\\
	\leq{}& \expect[\| 
	\frac{d}{2\delta}(F(y_{t,i}^+)-F(y_{t,i}^-))u_{t,i}-\nabla 
	\tF(z_t)\|^2|\mathcal{F}_{t-1}] \\
	&+ \frac{d^2}{4\delta^2} \expect[\expect[| 
	\hat{F}(y_{t,i}^+)-F(y_{t,i}^+)|^2\cdot 
	\|u_{t,i}\|^2|\mathcal{F}_{t-1},u_{t,i}]|\mathcal{F}_{t-1}] \\
	&+ \frac{d^2}{4\delta^2} \expect[\expect[| 
	\hat{F}(y_{t,i}^-)-F(y_{t,i}^-)|^2\cdot 
	\|u_{t,i}\|^2|\mathcal{F}_{t-1},u_{t,i}]|\mathcal{F}_{t-1}] \\
	\leq{}& 
	cdG^2+\frac{d^2}{4\delta^2}\sigma_0^2+\frac{d^2}{4\delta^2}\sigma_0^2 
	\\
	={}& cdG^2+\frac{d^2}{2\delta^2}\sigma_0^2.
	\end{align*}
	
	Then we have
	\begin{align*}
	\expect[g_t|\mathcal{F}_{t-1}] 
	={}&\expect[\frac{1}{B_t}\sum_{i=1}^{B_t}\frac{d}{2\delta}(
	\hat{F}(y_{t,i}^+)-\hat{F}(y_{t,i}^-))u_{t,i}|\mathcal{F}_{t-1}]\\
	={}&\nabla \tF(z_t),    
	\end{align*}
	and
	\begin{align*}
	\expect[\|g_t - \nabla \tF(z_t) \|^2|\mathcal{F}_{t-1}] ={}&\frac{1}{B_t^2} 
	\sum_{i=1}^{B_t} \expect[\|\frac{d}{2\delta}(
	\hat{F}(y_{t,i}^+)-\hat{F}(y_{t,i}^-))u_{t,i}-\nabla \tF(z_t) 
	\|^2|\mathcal{F}_{t-1}] \\
	\leq{}& \frac{cdG^2+\frac{d^2}{2\delta^2}\sigma_0^2}{B_t}.     
	\end{align*}
	
	Similar to the proof of \cref{thm:zero}, we have
	\begin{equation*}
	\expect[\| \nabla \tF(z_t)-\bar{g}_t\|^2]\le 
	\frac{Q}{(t+4)^{2/3}},
	\end{equation*}
	where $Q= \max 
	\{4^{2/3}G^2,6L^2D_1^2+\frac{4cdG^2+2d^2\sigma_0^2/\delta^2}{B_t}  \} 
	$. Thus we conclude
	\begin{align*}
	(1-\frac{1}{e})F(x^*)-\expect[F(z_{T+1})] \le  
	\frac{3D_1Q^{1/2}}{T^{1/3}}+ \frac{LD_1^2}{2 T} 
	+ \delta G(1+(\sqrt{d}+1)(1-\frac{1}{e})).
	\end{align*}
\end{proof}

\subsection{Proof of 
	Lemma 12}\label{app:lemma_disrite_to_continuous}
\begin{proof}
	Recall that $F(x)=\expect_{X \sim x}[f(X)] = \sum_{S \subseteq 
		\Omega}f(S)\prod_{i \in S}x_i\prod_{j \notin S}(1-x_j)$, then for any 
	fixed 
	$i \in [d]$, where $d=|\Omega|$, we have
	
	\begin{align*}
	|\frac{\partial F(x)}{\partial x_i}| ={}& |\sum_{\substack{S \subseteq 
	\Omega 
			\\ i \in S}}f(S)\prod_{\substack{j 
			\in S \\ j \neq i}} x_j \prod_{\substack{k \notin S \\ k \neq i}} 
	(1-x_k)  - \sum_{\substack{S \subseteq \Omega \\ i \notin 
			S}}f(S)\prod_{\substack{j \in S \\ j \neq i}} x_j 
	\prod_{\substack{k \notin 
			S \\ k \neq i}} (1-x_k)| \\
	\le{} & M[\sum_{\substack{S \subseteq \Omega \\ i \in S}}\prod_{\substack{j 
			\in S \\ j \neq i}} x_j \prod_{\substack{k \notin S \\ k \neq i}} 
	(1-x_k)+ 
	\sum_{\substack{S \subseteq \Omega \\ i \notin S}}\prod_{\substack{j \in S 
			\\ j \neq i}} x_j \prod_{\substack{k \notin S \\ k \neq i}} 
	(1-x_k)] \\
	={}& 2M.
	\end{align*}
	
	So we have
	\begin{equation*}
	\| \nabla F(x) \| \le 2M\sqrt{d}.   
	\end{equation*}
	
	Then $F$ is $2M\sqrt{d}$-Lipschitz.
	
	Now we turn to prove that $F$ has Lipschitz continuous gradients. Thanks to 
	the multilinearity, we have
	\begin{equation*}
	\frac{\partial F}{\partial x_i} = F(x|x_i=1)-F(x|x_i=0).    
	\end{equation*}
	
	Since 
	\begin{equation*}
	F(x|x_i=1)=\sum_{\substack{S \subseteq \Omega \\ i \in 
			S}}f(S)\prod_{\substack{j \in S \\ j \neq i}} x_j 
	\prod_{\substack{k \notin 
			S \\ k \neq i}} (1-x_k),    
	\end{equation*}
	we have
	\begin{equation*}
	\frac{\partial F(x|x_i=1)}{\partial x_i} =0,    
	\end{equation*}
	and for any fixed $j \neq i$,
	\begin{align*}
	|\frac{\partial F(x|x_i=1)}{\partial x_j}| ={}& |\sum_{\substack{S 
	\subseteq 
			\Omega \\ i,j \in S}}f(S)\prod_{\substack{l 
			\in S \\ l \notin \{i,j\}}} x_l \prod_{\substack{k \notin S \\ k 
			\notin 
			\{i,j \}}} (1-x_k)  -\sum_{\substack{S \subseteq \Omega \\ i \in S, 
			j \notin 
			S}}f(S)\prod_{\substack{l \in S \\ l \notin \{i,j\}}} x_l 
	\prod_{\substack{k \notin S \\ k \notin \{i,j \}}} (1-x_k)| \\
	\le{} & M [\sum_{\substack{S \subseteq \Omega \\ i,j \in 
			S}}\prod_{\substack{l \in S \\ l \notin \{i,j\}}} x_l 
	\prod_{\substack{k 
			\notin S \\ k \notin \{i,j \}}} (1-x_k)  + \sum_{\substack{S 
			\subseteq \Omega \\ i \in S, j \notin 
			S}}\prod_{\substack{l \in S \\ l \notin \{i,j\}}} x_l 
	\prod_{\substack{k 
			\notin S \\ k \notin \{i,j \}}} (1-x_k)] \\
	={}& 2M.
	\end{align*}
	
	Similarly, we have $\frac{\partial F(x|x_i=0)}{\partial x_i}=0$, and 
	$|\frac{\partial F(x|x_i=0)}{\partial x_j}|\le 2M$ for $j \neq i$. So we 
	conclude that
	\begin{equation*}
	|\frac{\partial^2 F}{\partial x_j \partial x_i}| \le 
	\begin{cases}
	0,& \quad \text{if } j=i, \\
	4M,& \quad \text{if } j \neq i.
	\end{cases}
	\end{equation*}
	
	Then $\|\nabla \frac{\partial F}{\partial x_i}\| \le 4M\sqrt{d-1} $, 
	\emph{i.e.}, $\frac{\partial F}{\partial x_i}$ is $4M\sqrt{d-1}$-Lipschitz.
	
	Then we deduce that
	\begin{align*}
	\| \nabla F(z_1) - \nabla F(z_2) \| ={}& \left[\sum_{i=1}^{d} \left( 
	\frac{\partial F(z_1)}{\partial x_i} - 
	\frac{\partial F(z_2)}{\partial x_i}\right)^2\right]^{1/2} \\
	\le{} & \left[\sum_{i=1}^{d} (4M\sqrt{d-1})^2 \|z_1-z_2 \|^2\right]^{1/2} \\
	={}& \sqrt{\sum_{i=1}^{d} (4M\sqrt{d-1})^2} \cdot \|z_1-z_2 \| \\
	={}& 4M\sqrt{d(d-1)} \|z_1-z_2 \|.
	\end{align*}
	
	So $F$ is $4M\sqrt{d(d-1)}$-smooth.
\end{proof}

\subsection{Proof of Theorem 11}
\label{app:theorem_discrete}
\begin{proof}
	Recall that we define $z_t= x_t+\delta \one$. Then we have
	\begin{align*}
	\expect[\|g_t - \nabla \tF(z_t) \|^2|\mathcal{F}_{t-1}] 
	={}& \frac{1}{B_t^2} \sum_{i=1}^{B_t} 
	\expect[\|\frac{d}{2\delta}(\bar{f}_{t,i}^+-\bar{f}_{t,i}^-)u_{t,i}-\nabla 
	\tF(z_t) \|^2|\mathcal{F}_{t-1}] \\
	={}&\frac{1}{B_t^2} \sum_{i=1}^{B_t} 
	\expect[\|[\frac{d}{2\delta}(F(y_{t,i}^+)-F(y_{t,i}^-))u_{t,i}-\nabla 
	\tF(z_t)] \\
	&\quad + \frac{d}{2\delta}[\bar{f}_{t,i}^+ - F(y_{t,i}^+)]u_{t,i} - 
	\frac{d}{2\delta} [\bar{f}_{t,i}^- - F(y_{t,i}^-)]u_{t,i} 
	\|^2|\mathcal{F}_{t-1}] \\
	={}& \frac{1}{B_t^2} \sum_{i=1}^{B_t} 
	\expect[\|[\frac{d}{2\delta}(F(y_{t,i}^+)-F(y_{t,i}^-))u_{t,i}-\nabla 
	\tF(z_t)]\|^2 | \mathcal{F}_{t-1}] \\
	& \quad + \frac{1}{B_t^2} \sum_{i=1}^{B_t} 
	\expect[|\frac{d}{2\delta}[\bar{f}_{t,i}^+ - F(y_{t,i}^+)] |^2 | 
	\mathcal{F}_{t-1}] \\
	& \quad + \frac{1}{B_t^2} \sum_{i=1}^{B_t} 
	\expect[|\frac{d}{2\delta}[\bar{f}_{t,i}^- - F(y_{t,i}^-)] |^2 | 
	\mathcal{F}_{t-1}],
	\end{align*}
	where we used the independence of $\bar{f}^{\pm}_{t,i}$ and the facts that 
	$\expect[\bar{f}^{\pm}_{t,i}] = F(y^\pm_{t,i}), 
	\expect[\frac{d}{2\delta}(F(y_{t,i}^+)-F(y_{t,i}^-))u_{t,i}]=\nabla\tF(z_t)$.

	Then same to \cref{eq:variance_upper} and by 
	\cref{lem:discrete_to_continuous}, the first item is no more than 
	$\frac{4cd^2M^2}{B_t}$. To upper bound the last two items, we have for 
	every $i \in [B_t]$,
	\begin{equation*}
	\begin{split}
	\expect[|\frac{d}{2\delta}[\bar{f}_{t,i}^+ - F(y_{t,i}^+)] |^2 | 
	\mathcal{F}_{t-1}] &= \frac{d^2}{4\delta^2} \expect[ 
	[\sum_{j=1}^l 
	[f(Y_{t,i,j}^+) - F(y_{t,i}^+)]/l]^2 |\mathcal{F}_{t-1}] \\
	&\leq \frac{d^2}{4\delta^2}\cdot l \cdot \frac{M^2}{l^2} \\
	&=\frac{d^2M^2}{4l\delta^2}.
	\end{split}  
	\end{equation*}		
	
%
	
	Similarly, we have 
	\begin{equation*}
	\expect[|\frac{d}{2\delta}[\bar{f}_{t,i}^- - F(y_{t,i}^-)] |^2 | 
	\mathcal{F}_{t-1}] \leq \frac{d^2M^2}{4l\delta^2}.
	\end{equation*}
	
%
%
%

As a result, we have 
\begin{equation*}
\begin{split}
\expect[\|g_t - \nabla \tF(z_t) \|^2|\mathcal{F}_{t-1}] 
&\le \frac{4cd^2M^2}{B_t} + \frac{1}{B_t^2}\cdot B_t \cdot 
\frac{d^2M^2}{4l\delta^2} + \frac{1}{B_t^2}\cdot B_t \cdot 
\frac{d^2M^2}{4l\delta^2} \\
&= \frac{4cd^2M^2}{B_t} + \frac{d^2M^2}{2B_tl\delta^2}.
\end{split}
\end{equation*}

	
	Then same to the proof for \cref{thm:zero}, we have
	\begin{equation}\label{eq:discrete_ineq2}
	\begin{split}
	&(1-\frac{1}{e})F(x^*)-\expect[F(z_{T+1})]\\
	\le{}& 
	\frac{3D_1Q^{1/2}}{T^{1/3}}+ \frac{2M\sqrt{d(d-1)}D_1^2}{T} 
	+ 2M\delta 
	\sqrt{d}(1+(\sqrt{d}+1)(1-\frac{1}{e})),
	\end{split}
	\end{equation}
	where $D_1\triangleq \diam(\constraint')$, $Q= \max \{ 4^{5/3}dM^2, 
	\frac{2d^2M^2(8c+\frac{1}{l\delta^2})}{B_t}+ 96d(d-1)M^2D_1^2 \} $, $x^*$ 
	is the global 
	maximizer of $F$ on $\constraint$.
	
	Note that since the rounding scheme is lossless, we have
	\begin{equation}
	(1-\frac{1}{e})f(X^*) - \expect[f(X_{T+1})] \le  (1-\frac{1}{e})F(x^*) - 
	\expect[F(z_{T+1})].   \label{eq:discrete_ineq1} 
	\end{equation}
	
	Combine \cref{eq:discrete_ineq1,eq:discrete_ineq2}, we have
	\begin{align*}
	&(1-\frac{1}{e})f(X^*) - \expect[f(X_{T+1})] \\
	\le{}&  
	\frac{3D_1Q^{1/2}}{T^{1/3}}+ \frac{2M\sqrt{d(d-1)}D_1^2}{T} 
	+ 2M\delta \sqrt{d}(1+(\sqrt{d}+1)(1-\frac{1}{e})).
	\end{align*}
\end{proof}

	\include{Chapter4-Test}

	\include{Chapter5-Fixed-Spec-Avoid}

	\include{Chapter6-Conclusions}

    \addcontentsline{toc}{chapter}{Bibliography} 
    \setstretch{1.0}

    \bibliographystyle{unsrtnat}
    \bibliography{main}

\end{document}